\documentclass{birkau}

\usepackage[mathscr]{eucal}
\usepackage{amssymb,bm}
\usepackage{enumerate,graphicx,color}
\usepackage{verbatim,url}

\usepackage{amssymb}
\usepackage{amsmath}
\usepackage{amsthm}
\usepackage{enumerate,graphicx}
\usepackage[mathscr]{eucal}
\usepackage{verbatim,url}
\usepackage{xy}
\usepackage{calc}

\numberwithin{equation}{section}

\theoremstyle{plain}

\newtheorem{thm}{Theorem}[section]
\newtheorem{lem}[thm]{Lemma}
\newtheorem{prop}[thm]{Proposition}

\theoremstyle{definition}
\newtheorem{df}[thm]{Definition}
\newtheorem{rem}[thm]{Remark}
\newtheorem{eg}[thm]{Example}

\usepackage{pgf,tikz}

\usetikzlibrary{calc,positioning,shapes,arrows.meta}

\tikzset{%
 shaded/.style={draw, shape=circle, fill=black!35, inner sep=1.4pt},
 unshaded/.style={draw, shape=circle, fill=white, inner sep=1.4pt},
 quasi/.style={draw, shape=rectangle, rounded corners=3pt, fill=white, inner sep=2.5pt, minimum height=14.5pt},
 blob/.style={draw, shape=rectangle, rounded corners=12pt, thin, densely dotted},
 arrow/.style={->, thin, >=latex, shorten >=2.5pt, shorten <=2.5pt},
 order/.style={thin},
 curvy/.style={thin, looseness=1.2, bend angle=70},
 fatcurvy/.style={thin, looseness=1.7, bend angle=75},
 label/.style={shape=rectangle, inner sep=6pt},
 auto}

\newcommand{\defn}[1]{{\emph{#1}}}
\newcommand{\eusbA}{\medsub e {\kern-0.75pt\A\kern-0.75pt}}

\newcommand{\cat}[1]{\boldsymbol{\mathscr{#1}}}
\newcommand{\CA}{{\cat A}}

\newcommand{\CJ}{\cat J}

\newcommand{\CM}{\cat M}
\newcommand{\CV}{\cat V}
\newcommand{\CX}{\cat X}

\font\bmi=cmmi8 scaled 1440
\newcommand{\powerset}{\raise.6ex\hbox{\bmi\char'175 }}

\newcommand{\twiddle}[1]{\mathbb{#1}} 

\newcommand{\MT}{\twiddle M}
\newcommand{\JT}{\twiddle J}
\newcommand{\T}{\mathscr{T}}

\newcommand{\A}{\mathbf A}
\newcommand{\B}{\mathbf B}
\newcommand{\C}{\mathbf C}

\newcommand{\K}{\mathbf K}
\newcommand{\F}{\mathbf F}
\newcommand{\M}{\mathbf M}

\renewcommand{\S}{\mathbf S}
\newcommand{\X}{\mathbb X}
\newcommand{\Y}{\mathbb Y}

\newcommand{\Ys}{\Y\kern -2pt _s}

\renewcommand{\t}{\mathrm t} 
\renewcommand{\k}{\mathrm k} 

 \DeclareMathOperator{\id}{id}

 \DeclareMathOperator{\Sub}{Sub} 
  
 \DeclareMathOperator{\Val}{Val} 
 \DeclareMathOperator{\Con}{Con} 
 
 \DeclareMathOperator{\ISP}{\mathsf{ISP}}

 \DeclareMathOperator{\HSP}{\mathsf{HSP}}
 
 \DeclareMathOperator{\IScP}{{\mathsf{IS} _{\mathrm{c}}
 \mathsf{P}^+}}
 
 \DeclareMathOperator{\sg}{sg} 

\newcommand{\comp}{{\setminus}}
\newcommand{\rest}[1]{{\upharpoonright}_{#1}}
\renewcommand{\le}{\leqslant}

\renewcommand{\ge}{\geqslant}

\newcommand{\lez}{\leqslant^0}
\newcommand{\gez}{\geqslant^0}
\newcommand{\lej}{\leqslant^j}

\newcommand{\lek}{\leqslant^k}
\newcommand{\gek}{\geqslant^k}
\newcommand{\lejk}{\leqslant^{jk}}
\newcommand{\du}{\mathbin{\dot\cup}}
\newcommand{\bigdu}{\overset{.}{\bigcup}}

\newcommand{\up}{{\uparrow}}
\newcommand{\down}{{\downarrow}}
\newcommand{\lsem}{[\kern-1.75pt[}
\newcommand{\rsem}{]\kern-1.75pt]}
\newcommand{\conv}{{}^{\kern2pt\raise-3pt\hbox{$\breve{}$}}} \newcommand{\convl}{{}^{\kern1pt\raise-3pt\hbox{$\breve{}$}}} 
\newcommand{\convsub}[1]{\rlap{$^{\kern2pt\raise-3pt\hbox{$\breve{}$}}$}{}_{#1}} 

\newcommand{\SEVEN}{\mathcal{SEVEN}}
\newcommand{\FOUR}{\mathcal{FOUR}}

\newcommand{\JB}{\mathbf{J}}

\newcommand{\mbf}{\boldsymbol{f}}
\newcommand{\mbt}{\boldsymbol{t}}

\newcommand{\mbdf}{\boldsymbol{df}}
\newcommand{\mbdt}{\boldsymbol{dt}}
\newcommand{\mbdT}{\boldsymbol{d}\top}

\newcommand{\bsF}{F} 
\newcommand{\bsT}{T} 
\newcommand{\bF}{F_{\kern -1pt n}} 
\newcommand{\bT}{T_{\kern -1pt n}} 
\newcommand{\bFA}{\mathbf F_{\kern -1pt \A}} 
\newcommand{\bTA}{\mathbf T_{\!\A}} 
\newcommand{\bFB}{\mathbf F_{\kern -1pt \B}} 
\newcommand{\bTB}{\mathbf T_{\kern -1pt \B}} 
\newcommand{\bFC}{\mathbf F_{\kern -1pt \C}} 
\newcommand{\bTC}{\mathbf T_{\kern -1pt \C}} 
\newcommand{\bFn}{\mathbf F_{\kern -1pt n}} 
\newcommand{\bTn}{\mathbf T_{\kern -1pt n}} 

\newcommand{\two}{\boldsymbol 2}
\newcommand{\one}{\boldsymbol 1}
\newcommand{\zero}{\boldsymbol 0}

\begin{document}


\title[Dualities for a new class of default bilattices]{Expanding Belnap: dualities for a new class of default bilattices}

\author[A. P. K. Craig]{Andrew P. K. Craig}
\address{Department of Mathematics and Applied Mathematics\\
University of Johannesburg\\PO Box 524, Auckland Park, 2006\\South~Africa}
\email{acraig@uj.ac.za}

\author[B. A. Davey]{Brian A. Davey}
\address{Department of Mathematics and Statistics\\La Trobe University\\Victoria 3086\\Australia}
\email{b.davey@latrobe.edu.au}

\author[M. Haviar]{Miroslav Haviar}
\address{Department of Mathematics\\Faculty of Natural Sciences, M. Bel University\\Tajovsk\'eho 40, 974~01 Bansk\'a Bystrica\\Slovakia}
\email{miroslav.haviar@umb.sk}

\dedicatory{Dedicated to the memory of Prof.\
Beloslav Rie\v can}



\begin{abstract}
Bilattices provide an algebraic tool with which to model
simultaneously knowledge and truth. They were introduced by Belnap
in 1977 in a paper entitled \emph{How a computer should think}.
Belnap argued that instead of using a logic with two values, for
`true' ($\mbt$) and `false' ($\mbf$), a computer should use a logic
with two further values, for `contradiction' ($\top$) and `no
information' ($\bot$). The resulting structure is equipped with two
lattice orders, a \emph{knowledge order} and a \emph{truth order},
and hence is called a \emph{bilattice}.

Prioritised default bilattices include not only values for `true'
($\mbt_0$), `false' ($\mbf_0$), `contradiction' and `no
information', but also indexed families of default values, $\mbt_1,
\dots, \mbt_n$ and $\mbf_1, \dots, \mbf_n$, for simultaneous
modelling of degrees of knowledge and truth.

We focus on a new family of prioritised default bilattices: $\mathbf
J_n$, for $n \in \omega$.
The bilattice $\mathbf J_0$ is precisely Belnap's seminal example.
We address mathematical rather than logical aspects of our
prioritised default bilattices.
We obtain a single-sorted topological representation for the
bilattices in the quasivariety $\CJ_n$ generated by $\mathbf J_n$,
and separately a multi-sorted topological representation for the
bilattices in the variety $\CV_n$ generated by $\mathbf J_n$. Our
results provide an interesting example where the multi-sorted
duality for the variety has a simpler structure than the
single-sorted duality for the quasivariety.
\end{abstract}

\maketitle


\section{Introduction}\label{sec:intro}

We describe a new class of default bilattices
$\{\, \JB_n \mid n \in \omega\,\}$ for use in prioritised default logic.
While the first of these bilattices ($n=0$) is Belnap's original four-element bilattice~\cite{Bel-think},
for $n \geqslant 1$ these bilattices provide new algebraic structures for dealing
with inconsistent and incomplete information. In particular, the structure of the knowledge order gives a
new method for interpreting contradictory responses from amongst a hierarchy of `default true' and `default false' responses.

We seek representations for algebras in the quasivariety $\CJ_n
= \ISP(\JB_n)$, and more generally in the variety $\CV_n =
\HSP(\JB_n)$, generated by~$\JB_n$. For $n\ge 1$, our bilattices are
not interlaced and hence we lack the much-used product
representation. This leads us to develop a concrete  representation
via the theory of natural dualities. We prove a single-sorted
duality for the quasivariety $\CJ_n$ and a multi-sorted duality for
the variety $\CV_n$. Furthermore, we are able to show that our
dualities are \emph{optimal} in the sense that none of the structure
of the dualising object can be removed without destroying the
duality.

To place both our family of bilattices, and our results concerning them, in an appropriate context, we recall some history. Bilattices were investigated in the late 1980's by Ginsberg~\cite{Gins86, Gins88}
as a method for inference with incomplete and contradictory information.
These investigations built on the simple example introduced by
Belnap~\cite{Bel-think} about a decade earlier.
Belnap proposed that a computer should have a truth value,~$\top$, which would be assigned to any statement that it had been told separately
was both true and false.
This is a very plausible idea in situations
where a computer might receive information from different sources.
Equally important is the ability of a computer to make decisions
based on incomplete information. The truth value $\bot$
is assigned to statements about which the computer has no information.
This idea was represented by the four-element structure shown in Figure~\ref{fig:FOUR}. The elements $\mbt$ and $\mbf$ represent
`true' and `false', while the elements $\top$ and $\bot$ represent `contradiction' and
`no information'. The order represented on the vertical axis in Figure~\ref{fig:FOUR} is the \defn{knowledge order}
($\le_\k$), while the horizontal axis represents the
\defn{truth order} ($\le_\t$).

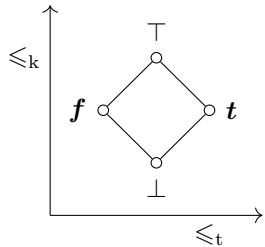
\begin{figure}[ht]
\centering
\begin{tikzpicture}[scale=0.7]
\path[-{>[length=1mm]}] (0,0) edge node[left,near end] {$\le_\k$} (0,4);
\path[-{>[length=1mm]}] (0,0) edge node[near end,below] {$\le_\t$} (4,0);
\begin{scope}[xshift=2cm,yshift=1cm]
  \node[unshaded] (bot) at (0,0) {};
  \node[unshaded] (f) at (-1,1) {};
  \node[unshaded] (t) at (1,1) {};
  \node[unshaded] (top) at (0,2) {};
  \draw[order] (bot) -- (f) -- (top);
  \draw[order] (bot) -- (t) -- (top);
  \node[label,anchor=north] at (bot) {$\bot$};
  \node[label,anchor=east] at (f) {$\mbf$};
  \node[label,anchor=west] at (t) {$\mbt$};
  \node[label,anchor=south] at (top) {$\top$};
\end{scope}
\end{tikzpicture}
\caption{The four truth values proposed by Belnap.}\label{fig:FOUR}
\end{figure}

A statement $p$ which is assigned the truth value $\top$
as a result of contradictory information is less true
than a statement $q$ which is assigned $\mbt$,
 as there is a source saying that $p$ is false. On the other hand, more is known about $p$ than is known about $q$, as there are
at least two different sources providing information.
(The term `information order' is used by some authors to refer to what we call the knowledge order.)

Generalising this example, a bilattice has two lattice orders,
$\le_\k$ (knowledge) and $\le_\t$ (truth)---see
Definitions~\ref{def:prebilat} and~\ref{def:bilat} for details.
While the concept of a truth order is familiar, for example, from
multi-valued logic, the knowledge order is less familiar, and we
discuss it very briefly. The join $\oplus$ in the knowledge order is
called \emph{gullability}: $a\oplus b$ represents the combined
information from $a$ and $b$ with no concern for any inherent
contradictions. The meet $\otimes$ in the knowledge order is called
\emph{consensus}: $a\otimes b$ represents the most information upon
which $a$ and $b$ agree. (See Fitting~\cite{Fit-nice} for an
excellent introduction to bilattices with many motivating examples.)

Belnap's four-element bilattice is often referred to as $\FOUR$. The bilattice $\SEVEN$ was proposed by Ginsberg~\cite[Figure 4]{Gins86}
for use in inference with default logic; see Figure~\ref{fig:SEVEN}, which shows $\SEVEN$ as it is usually depicted in the literature along with its knowledge order $\le_\k$ and truth order~$\le_\t$.
Note that $\SEVEN$ has two additional truth values, $\mbdt$ and $\mbdf$,
which represent `true by default' and `false by default', along with an element $\mbdT$ that represents the contradiction
that arises if a statement is both true by default and
false by default. The idea has been extended to include more default values (cf.~\cite[Figure 7]{Gins88}), where the sequence of
`true by default' truth values is decreasing in both the knowledge
order and truth order, while the sequence of `false by default' truth values is decreasing in the knowledge order but increasing in the truth order---see also Figure~\ref{fig:Kn-t-order}.

\begin{figure}[t]
\centering
\begin{tikzpicture}[scale=1]
\begin{scope}
  \node[unshaded] (bot) at (0,0) {};
  \node[unshaded] (df) at (-0.5,0.5) {};
  \node[unshaded] (dt) at (0.5,0.5) {};
  \node[unshaded] (dtop) at (0,1) {};
  \node[unshaded] (f) at (-1,2) {};
  \node[unshaded] (t) at (1,2) {};
  \node[unshaded] (top) at (0,3) {};
  \draw[order] (bot) -- (df) -- (f) -- (top);
  \draw[order] (bot) -- (dt) -- (t) -- (top);
  \draw[order] (df) -- (dtop) -- (f);
  \draw[order] (dt) -- (dtop) -- (t);
  \node at (0,-1) {$\mathcal{SEVEN}$};
  \node[label,anchor=north] at (bot) {$\bot$};
  \node[label,anchor=east] at (df) {$\mbdf$};
  \node[label,anchor=west] at (dt) {$\mbdt$};
  \node[label,anchor=south,yshift=5pt] at (dtop) {$\mbdT$};
  \node[label,anchor=east] at (f) {$\mbf$};
  \node[label,anchor=west] at (t) {$\mbt$};
  \node[label,anchor=south] at (top) {$\top$};
\end{scope}
\begin{scope}[xshift=4cm] 
  \node[unshaded] (bot) at (0,0) {};
  \node[unshaded] (df) at (-0.75,0.75) {};
  \node[unshaded] (dt) at (0.75,0.75) {};
  \node[unshaded] (dtop) at (0,1.5) {};
  \node[unshaded] (f) at (-0.75,2.25) {};
  \node[unshaded] (t) at (0.75,2.25) {};
  \node[unshaded] (top) at (0,3) {};
  \draw[order] (bot) -- (df) -- (dtop) -- (f) -- (top);
  \draw[order] (bot) -- (dt) -- (dtop) -- (t) -- (top);
  \node at (0,-1) {$\le_\k$};
  \node[label,anchor=north] at (bot) {$\bot$};
  \node[label,anchor=east] at (df) {$\mbdf$};
  \node[label,anchor=west] at (dt) {$\mbdt$};
  \node[label,anchor=east,xshift=-2pt] at (dtop) {$\mbdT$};
  \node[label,anchor=east] at (f) {$\mbf$};
  \node[label,anchor=west] at (t) {$\mbt$};
  \node[label,anchor=south] at (top) {$\top$};
\end{scope}
\begin{scope}[xshift=7cm] 
  \node[unshaded] (bot) at (2,1.5) {};
  \node[unshaded] (df) at (1,0.75) {};
  \node[unshaded] (dt) at (1,2.25) {};
  \node[unshaded] (dtop) at (1,1.5) {};
  \node[unshaded] (f) at (0,0) {};
  \node[unshaded] (t) at (0,3) {};
  \node[unshaded] (top) at (0,1.5) {};
  \draw[order] (f) -- (top) -- (t);
  \draw[order] (df) -- (dtop) -- (dt);
  \draw[order] (f) -- (df) -- (bot) -- (dt) -- (t);
  \node at (1,-1) {$\le_\t$};
  \node[label,anchor=west] at (bot) {$\bot$};
  \node[label,anchor=north,xshift=5pt] at (df) {$\mbdf$};
  \node[label,anchor=south,xshift=5pt] at (dt) {$\mbdt$};
  \node[label,anchor=east,xshift=2pt] at (dtop) {$\mbdT$};
  \node[label,anchor=east,yshift=-3pt] at (f) {$\mbf$};
  \node[label,anchor=east,yshift=3pt] at (t) {$\mbt$};
  \node[label,anchor=east] at (top) {$\top$};
\end{scope}
\end{tikzpicture}
\caption{Ginsberg's bilattice for default logic.}\label{fig:SEVEN}
\end{figure}
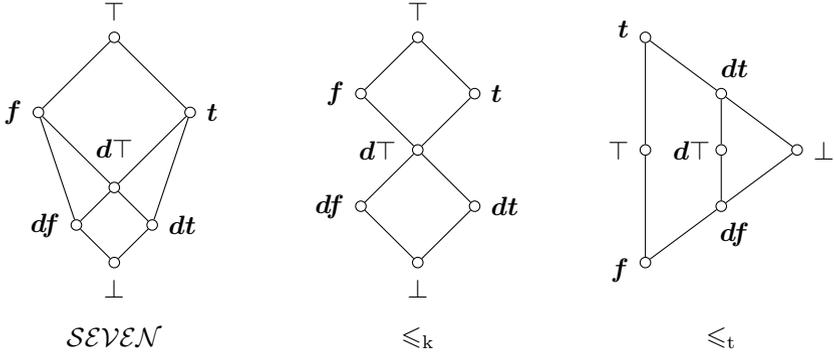

A criticism that can be levelled at Ginsberg's default bilattice $\SEVEN$ is that the element $\mbdT$ is both the $\k$-meet of $\mbt$ and $\mbf$, and the $\k$-join of $\mbdt$ and~$\mbdf$. That is, $\mbt \otimes \mbf = \mbdT = \mbdt \oplus \mbdf$.
If an agent is told that a certain statement is both true and false,
the level of \emph{agreement} or \emph{consensus} is modelled by the
bilattice element $\mbt \otimes \mbf$. The $\k$-join $\mbdt
\oplus \mbdf$ represents the total knowledge that an agent
has if it is told that something is both true by default and false by
default. However, it is not  clear that $\mbt \otimes \mbf$ should
always represent the same degree of knowledge and truth as the
$\k$-join $\mbdt \oplus \mbdf$.
The family $\{\, \JB_n \mid n\in \omega\,\}$ of default bilattices is designed to overcome this criticism.

The main difference between our family of default bilattices and the prioritised default bilattices in the style of $\SEVEN$ is that in our family there is no distinction between the level at which the contradictions or agreements take place.
That is, we propose that, for $n\in \omega$, the bilattice $\JB_n$ should satisfy
\[
\mbt_i \oplus  \mbf_j = \top  \quad \text{and} \quad
\mbt_i \otimes \mbf_j = \bot,
\]
for all $i,j \in \{0,\ldots,n\}$. Any contradictory response that includes some level of truth ($\mbt_i$) and some level of falsity ($\mbf_j$) is registered as a total contradiction ($\top$) and a total lack of consensus ($\bot$).
An illustration of such proposed bilattices with default
truth values drawn in their knowledge order is given in Figure~\ref{fig:468}---see Definition~\ref{def:Jn} for the formal definition.

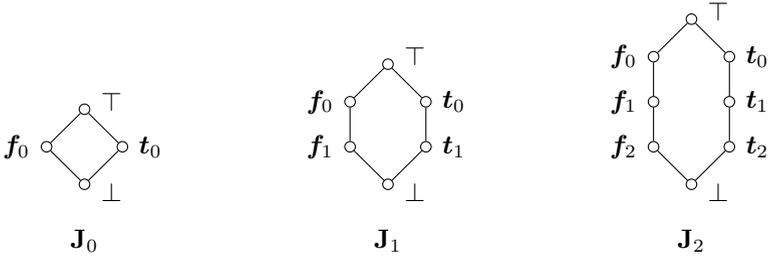
\begin{figure}[ht]
\centering
\begin{tikzpicture}[scale=0.5]
\begin{scope}
  \node at (0,-1.5) {$\JB_0$};
  \node[unshaded] (bot) at (0,0) {};
  \node[unshaded] (f0) at (-1,1) {};
  \node[unshaded] (t0) at (1,1) {};
  \node[unshaded] (top) at (0,2) {};
  \draw[order] (bot) -- (f0) -- (top);
  \draw[order] (bot) -- (t0) -- (top);
  \node[label,anchor=west,yshift=-3pt] at (bot) {$\bot$};
  \node[label,anchor=east] at (f0) {$\mbf_0$};
  \node[label,anchor=west] at (t0) {$\mbt_0$};
  \node[label,anchor=west,yshift=3pt] at (top) {$\top$};
\end{scope}
\begin{scope}[xshift=8cm]
  \node at (0,-1.5) {$\JB_1$};
  \node[unshaded] (bot) at (0,0) {};
  \node[unshaded] (f1) at (-1,1) {};
  \node[unshaded] (t1) at (1,1) {};
  \node[unshaded] (f0) at (-1,2.2) {};
  \node[unshaded] (t0) at (1,2.2) {};
  \node[unshaded] (top) at (0,3.2) {};
  \draw[order] (bot) -- (f1) -- (f0) -- (top);
  \draw[order] (bot) -- (t1) -- (t0) -- (top);
  \node[label,anchor=west,yshift=-3pt] at (bot) {$\bot$};
  \node[label,anchor=east] at (f1) {$\mbf_1$};
  \node[label,anchor=west] at (t1) {$\mbt_1$};
  \node[label,anchor=east] at (f0) {$\mbf_0$};
  \node[label,anchor=west] at (t0) {$\mbt_0$};
  \node[label,anchor=west,yshift=3pt] at (top) {$\top$};
\end{scope}
\begin{scope}[xshift=16cm]
  \node at (0,-1.5) {$\JB_2$};
  \node[unshaded] (bot) at (0,0) {};
  \node[unshaded] (f2) at (-1,1) {};
  \node[unshaded] (t2) at (1,1) {};
  \node[unshaded] (f1) at (-1,2.2) {};
  \node[unshaded] (t1) at (1,2.2) {};
  \node[unshaded] (f0) at (-1,3.4) {};
  \node[unshaded] (t0) at (1,3.4) {};
  \node[unshaded] (top) at (0,4.4) {};
  \draw[order] (bot) -- (f2) -- (f1) -- (f0) -- (top);
  \draw[order] (bot) -- (t2) -- (t1) -- (t0) -- (top);
  \node[label,anchor=west,yshift=-3pt] at (bot) {$\bot$};
  \node[label,anchor=east] at (f2) {$\mbf_2$};
  \node[label,anchor=west] at (t2) {$\mbt_2$};
  \node[label,anchor=east] at (f1) {$\mbf_1$};
  \node[label,anchor=west] at (t1) {$\mbt_1$};
  \node[label,anchor=east] at (f0) {$\mbf_0$};
  \node[label,anchor=west] at (t0) {$\mbt_0$};
  \node[label,anchor=west,yshift=3pt] at (top) {$\top$};
\end{scope}
\end{tikzpicture}
\caption{The bilattices $\JB_0$, $\JB_1$ and $\JB_2$, drawn in their knowledge order.}\label{fig:468}
\end{figure}

Natural duality theory was first applied to the variety of distributive bilattices by Cabrer and Priestley~\cite{CPdbl}.
Initially Craig~\cite{C-thesis}, and later
Cabrer, Craig and Priestley~\cite{CCP15}, considered a family $\{\, \K_n \mid n \in \omega\,\}$ of non-interlaced default bilattices that
generalise Belnap's and Ginsberg's examples;
indeed, $\K_0$ is $\FOUR$ and $\K_1$ is $\SEVEN$---see Figure~\ref{fig:Kn-t-order}.
In both \cite{C-thesis} and \cite{CCP15} the authors applied natural duality theory to produce a duality for the  quasivariety $\ISP(\K_n)$ generated by $\K_n$, and in \cite{CCP15} they also
produced a multi-sorted duality for the variety $\HSP(\K_n)$ generated by~$\K_n$.

While our family of default bilattices overcomes the criticism mentioned above of default bilattices in the style of $\SEVEN$, it comes at a price.
As with the dualities for the quasivariety and the variety generated by~$\K_n$, to obtain our dualities for $\CJ_n$ and $\CV_n$ we are required to analyse the lattices of subuniverses of certain binary products of algebras from $\CV_n$. This turns out to be substantially more difficult in the case of $\JB_n$ than in the case of~$\K_n$ due to the sizes of the subuniverse lattices.
Nevertheless, the dualities we obtain, particularly in the multi-sorted case, are quite natural---see Remark~\ref{rem:compare}.

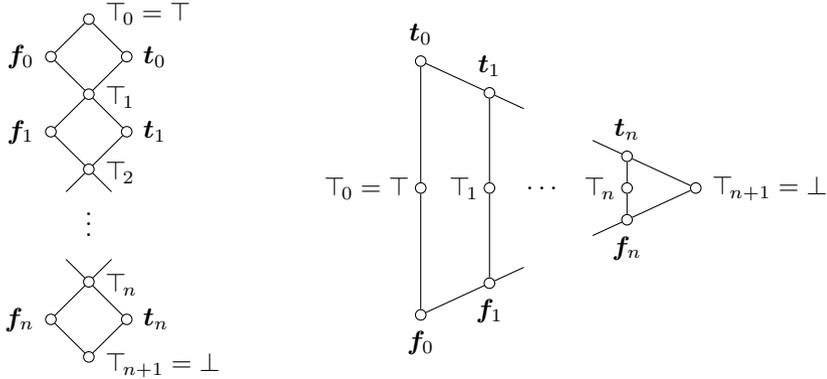
\begin{figure}[t]
\centering
\begin{tikzpicture}
\begin{scope}[scale=0.5]
  \node[unshaded] (bot) at (0,0) {};
  \node[unshaded] (fn) at (-1,1) {};
  \node[unshaded] (tn) at (1,1) {};
  \node[unshaded] (topn) at (0,2) {};
  \node at (0,3.75) {$\vdots$};
  \node[unshaded] (top2) at (0,5) {};
  \node[unshaded] (f1) at (-1,6) {};
  \node[unshaded] (t1) at (1,6) {};
  \node[unshaded] (top1) at (0,7) {};
  \node[unshaded] (f0) at (-1,8) {};
  \node[unshaded] (t0) at (1,8) {};
  \node[unshaded] (top) at (0,9) {};
  \draw[order] (bot) -- (fn) -- (topn) -- ($(topn)+(-0.6,0.6)$);
  \draw[order] (bot) -- (tn) -- (topn) -- ($(topn)+(0.6,0.6)$);
  \draw[order] ($(top2)+(-0.6,-0.6)$) -- (top2) -- (f1) -- (top1) -- (f0) -- (top);
  \draw[order] ($(top2)+(0.6,-0.6)$) -- (top2) -- (t1) -- (top1) -- (t0) -- (top);
  \node[label,anchor=west,yshift=-3pt] at (bot) {$\top_{\!n+1}=\bot$};
  \node[label,anchor=east] at (fn) {$\mbf_n$};
  \node[label,anchor=west] at (tn) {$\mbt_n$};
  \node[label,anchor=west,yshift=-1pt] at (topn) {$\top_{\!n}$};
  \node[label,anchor=west,yshift=-1pt] at (top2) {$\top_{\!2}$};
  \node[label,anchor=east] at (f1) {$\mbf_1$};
  \node[label,anchor=west] at (t1) {$\mbt_1$};
  \node[label,anchor=west,yshift=-1pt] at (top1) {$\top_{\!1}$};
  \node[label,anchor=east] at (f0) {$\mbf_0$};
  \node[label,anchor=west] at (t0) {$\mbt_0$};
  \node[label,anchor=west,yshift=2pt] at (top) {$\top_{\!0}=\top$};
\end{scope}
\begin{scope}[scale=1,xshift=8cm,yshift=2.25cm]
  \node[unshaded] (bot) at (0,0) {};
  \node[unshaded] (fn) at (-155:1) {};
  \node[unshaded] (tn) at (155:1) {};
  \node[unshaded] (topn) at ($0.5*(fn)+0.5*(tn)$) {};
  \node at (-2,0) {$\dots$};
  \node[unshaded] (f1) at (-155:3) {};
  \node[unshaded] (t1) at (155:3) {};
  \node[unshaded] (top1) at ($0.5*(f1)+0.5*(t1)$) {};
  \node[unshaded] (f0) at (-155:4) {};
  \node[unshaded] (t0) at (155:4) {};
  \node[unshaded] (top) at ($0.5*(f0)+0.5*(t0)$) {};
  \draw[order] (f0) -- (top) -- (t0);
  \draw[order] (f1) -- (top1) -- (t1);
  \draw[order] (fn) -- (topn) -- (tn);
  \draw[order] (bot) -- (tn) -- ($(tn)+(155:0.5)$);
  \draw[order] ($(t1)+(-25:0.5)$) -- (t1) -- (t0);
  \draw[order] (f0) -- (f1) -- ($(f1)+(25:0.5)$);
  \draw[order] ($(fn)+(-155:0.5)$) -- (fn) -- (bot);
  \node[label,anchor=west] at (bot) {$\top_{\!n+1}=\bot$};
  \node[label,anchor=north] at (fn) {$\mbf_n$};
  \node[label,anchor=south] at (tn) {$\mbt_n$};
  \node[label,anchor=east,xshift=2pt] at (topn) {$\top_{\!n}$};
  \node[label,anchor=north] at (f1) {$\mbf_1$};
  \node[label,anchor=south] at (t1) {$\mbt_1$};
  \node[label,anchor=east,xshift=2pt] at (top1) {$\top_{\!1}$};
  \node[label,anchor=north] at (f0) {$\mbf_0$};
  \node[label,anchor=south] at (t0) {$\mbt_0$};
  \node[label,anchor=east,xshift=2pt] at (top) {$\top_{\!0}=\top$};
\end{scope}
\end{tikzpicture}
\caption{$\K_n$ in its knowledge order (left) and truth order (right).}\label{fig:Kn-t-order}
\end{figure}

The paper is structured as follows. We define the family $\{\, \JB_n \mid n \in \omega\,\}$  of default bilattices in Section~\ref{sec:Jn}. There we not only describe the algebras themselves, but also derive some properties of
the variety $\CV_n$ generated by~$\JB_n$. In particular, we show that, up to isomorphism,  $\CV_n$ contains $n +1$ subdirectly irreducible members denoted by $\M_0, \dots, \M_n$, each of which is a homomorphic image of~$\JB_n$: the algebra $\M_0$ has size~4 and is term equivalent to $\FOUR$,  and the algebras $\M_1, \dots, \M_n$ have size~6.

At the beginning of Section~\ref{sec:NatDualities}, we note some existing duality and representation results for bilattices, before summarising
the necessary background from the theory of natural dualities as presented in the book by Clark and Davey~\cite{CD98}. We state restricted versions of more general theorems, as these are all that we require.
Section~\ref{sec:CJduality} is devoted to setting up and stating our first important duality result, the single-sorted duality for the quasivariety~$\CJ_n$ (Theorem~\ref{cor:bigduality}). The duality is optimal and, for $n\ge 1$, uses $\frac12(n^2 - n +4)$ relations.
The setup and statement of the multi-sorted duality for $\CV_n$ is in
Section~\ref{sec:CVduality} (Theorem~\ref{cor:bigmultiduality}). There are $n + 1$ sorts, one corresponding to each of the subdirectly irreducible algebras~$\M_k$, for $k\in \{0, \dots, n\}$. Again, the duality is optimal; for $n\ge 1$, it uses a total of $\frac12(n^2 + 3n + 2)$ relations and operations.

A consequence of the underlying lattice structure of our algebras is that the main tool for
proving the duality for the quasivariety $\CJ_n$ is a good description of the subuniverse lattice $\Sub(\JB_n^2)$.
In particular, the identification of the meet-irreducible elements of this subuniverse lattice is crucial. Similarly, the main tool for proving the multi-sorted duality for the variety $\CV_n$ is a good description of the meet-irreducible elements of $\Sub(\M_j \times \M_k)$,
for all $j, k \in \{0, \dots, n\}$. We achieve both of these tasks simultaneously in Section~\ref{sec:CompOnHoms} by studying
$\Sub(\A\times \B)$, where $\A$ and $\B$ are non-trivial homomorphic images of $\JB_n$.
The proofs of the duality theorems, and of their optimality, are given in Sections~\ref{sec:singleproof} to~\ref{sec:optmulti}.

In a follow-up paper, the authors will study the problem of axiomatising the dual categories, the process of translating from our duals to the Priestley duals of the underlying distributive lattices, and
will use the translation to examine the free algebras in $\CV_n$.

\section{The prioritised default bilattice $\JB_n$}\label{sec:Jn}

Most definitions related to bilattices
are originally due to Ginsberg~\cite{Gins88}.
These have evolved over time and in the literature there exists
some variation in notation and terminology.
Our presentation is close to that of Jung and Rivieccio~\cite{JR12}.

\begin{df} \label{def:prebilat}  A \defn{pre-bilattice} is an algebra $\B = \langle B; \otimes, \oplus, \wedge, \vee \rangle$ such that $\langle B; \otimes,\oplus \rangle$ and $\langle B; \wedge,\vee\rangle$ are lattices. We denote by $\le_\k$ the order associated with $\langle B; \otimes, \oplus\rangle$ and by $\le_\t$ the order associated with $\langle B; \wedge, \vee \rangle$.
\end{df}

The definition of a pre-bilattice does not require any kind of relationship
between the two lattice orders. Thus, with a change of signature,
any lattice $\langle L; \sqcap, \sqcup \rangle$ can be considered as
a pre-bilattice where each of $\le_\k$ and $\le_\t$ is either the original order $\sqsubseteq$ from $L$, or its dual $\sqsupseteq$.
Ginsberg's original definition~\cite[Definition 4.1]{Gins88} required that both of the lattices $\langle B; \otimes, \oplus \rangle$ and $\langle B; \wedge,\vee \rangle$ were complete. Recent authors seldom require completeness and our work does not make this requirement.

It is unsurprising that in some contexts there will be some interaction between the two orders.
A \emph{distributive} pre-bilattice $\B$ is one in which
$\bullet$ distributes over $\ast$, for all $\bullet,\ast \in \{\otimes,\oplus, \wedge,\vee\}$. When each set of operations preserves the other order,
i.e., $\otimes$ and $\oplus$ preserve $\le_\t$ and
$\wedge$ and $\vee$ preserve $\le_\k$, then the pre-bilattice is said to be \emph{interlaced}.

\begin{df} \label{def:bilat}  A \defn{bilattice} is an algebra $\B=\langle B; \otimes, \oplus,\wedge,\vee, \neg \rangle$ such that the reduct $\langle B;\otimes,\oplus,\wedge,\vee\rangle$ is a pre-bilattice
and $\neg$ is a unary operation which is $\le_\k$-preserving, $\le_\t$-reversing and involutive.
\end{df}

We note that some authors use the term `bilattice' and `bilattice with negation' to describe
the objects from Definition~\ref{def:prebilat} and~\ref{def:bilat}, respectively. When the lattices are bounded, the upper and lower
bounds of the knowledge order are denoted by $\top$ and $\bot$, and the
upper and lower bounds of the truth order are denoted by $\mbt$ and~$\mbf$.

Bilattices were studied intensively from their first description until the end of 1990's. In recent years there has been
a resurgence of interest from mathematicians, largely catalysed by the work of Rivieccio~\cite{Riv-thesis}.
In the wake of his thesis a number of papers have examined both algebraic and logical aspects of bilattices~\cite{BouR11,BouJanRiv,CPdbl}.
These recent investigations have extended to the related notions of twist structures~\cite{Riv-twist} and trilattices~\cite{CP-IGPL}.

We now define prioritised default bilattices $\JB_n$
which extend, to $n$-levels of default truth values, the motivation behind the six- and eight-element bilattices in Figure~\ref{fig:468}.
These bilattices were originally studied in the first author's DPhil thesis~\cite{C-thesis}.

\begin{df}\label{def:Jn}
For each $n\in \omega$, the underlying set of $\JB_n$ is
\[
J_n = \{ \top, \mbf_0, \dots, \mbf_n,\mbt_0, \dots, \mbt_n, \bot \}.
\]
The knowledge and truth orders, $\le_\k$ and $\le_\t$,
on $\JB_n$ are given in Figure~\ref{fig:Jn}. When necessary we will add a superscript and denote these orders by $\le_\k^n$ and~$\le_\t^n$.
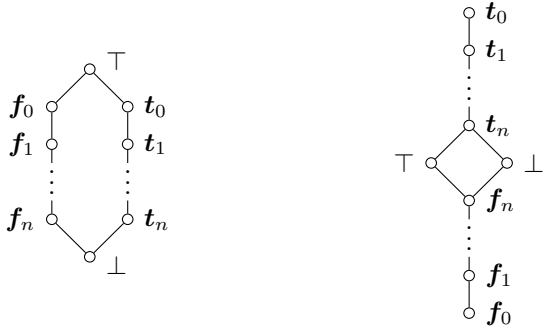
\begin{figure}[th]
\begin{tikzpicture}[scale=0.5]
\begin{scope}
  \node[unshaded] (bot) at (0,0) {};
  \node[unshaded] (fn) at (-1,1) {};
  \node[unshaded] (tn) at (1,1) {};
  \node at (-1,2.2) {$\vdots$};
  \node at (1,2.2) {$\vdots$};
  \node[unshaded] (f1) at (-1,3) {};
  \node[unshaded] (t1) at (1,3) {};
  \node[unshaded] (f0) at (-1,4) {};
  \node[unshaded] (t0) at (1,4) {};
  \node[unshaded] (top) at (0,5) {};
  \draw[order] (bot) -- (fn) -- ($(fn)+(0,0.5)$);
  \draw[order] (bot) -- (tn) -- ($(tn)+(0,0.5)$);
  \draw[order] ($(f1)+(0,-0.5)$) -- (f1) -- (f0) -- (top);
  \draw[order] ($(t1)+(0,-0.5)$) -- (t1) -- (t0) -- (top);
  \node[label,anchor=west,yshift=-3pt] at (bot) {$\bot$};
  \node[label,anchor=east] at (fn) {$\mbf_n$};
  \node[label,anchor=west] at (tn) {$\mbt_n$};
  \node[label,anchor=east] at (f1) {$\mbf_1$};
  \node[label,anchor=west] at (t1) {$\mbt_1$};
  \node[label,anchor=east] at (f0) {$\mbf_0$};
  \node[label,anchor=west] at (t0) {$\mbt_0$};
  \node[label,anchor=west,yshift=3pt] at (top) {$\top$};
\end{scope}
\begin{scope}[xshift=10cm,yshift=-1.5cm]
  \node[unshaded] (f0) at (0,0) {};
  \node[unshaded] (f1) at (0,1) {};
  \node at (0,2.2) {$\vdots$};
  \node[unshaded] (fn) at (0,3) {};
  \node[unshaded] (top) at (-1,4) {};
  \node[unshaded] (bot) at (1,4) {};
  \node[unshaded] (tn) at (0,5) {};
  \node at (0,6.2) {$\vdots$};
  \node[unshaded] (t1) at (0,7) {};
  \node[unshaded] (t0) at (0,8) {};
  \draw[order] (f0) -- (f1) -- ($(f1)+(0,0.5)$);
  \draw[order] ($(fn)+(0,-0.5)$) -- (fn) -- (top) -- (tn) -- ($(tn)+(0,0.5)$);
  \draw[order] (fn) -- (bot) -- (tn);
  \draw[order] ($(t1)+(0,-0.5)$) -- (t1) -- (t0);
  \node[label,anchor=west] at (bot) {$\bot$};
  \node[label,anchor=west] at (fn) {$\mbf_n$};
  \node[label,anchor=west] at (tn) {$\mbt_n$};
  \node[label,anchor=west] at (f1) {$\mbf_1$};
  \node[label,anchor=west] at (t1) {$\mbt_1$};
  \node[label,anchor=west] at (f0) {$\mbf_0$};
  \node[label,anchor=west] at (t0) {$\mbt_0$};
  \node[label,anchor=east] at (top) {$\top$};
\end{scope}
\end{tikzpicture}

\caption{The bilattice $\JB_n$  in  its knowledge order (left) and truth order (right).}\label{fig:Jn}
\end{figure}
A unary involutive operation $\neg$ that preserves the $\le_\k$-order
and reverses the $\le_\t$-order on $J_n$ is given by:
\[
\neg\top = \top, \quad \neg\bot = \bot, \quad  \neg\mbf_m = \mbt_m \text{ and } \neg\mbt_m = \mbf_m, \text{ for all } m\in \{0, \dots, n\}.
\]
We then add every element of $J_n$ as a constant to obtain the
prioritised default bilattice
\[
\JB_n = \langle J_n; \otimes, \oplus, \wedge,\vee, \neg, \top,
\mbf_0, \dots, \mbf_n,\mbt_0, \dots, \mbt_n, \bot \rangle,
\]
where $\otimes$ and $\oplus$ are greatest lower bound and least upper bound in the knowledge order $\le_\k$, and $\wedge$ and $\vee$ are greatest lower bound and least upper bound in the truth order $\le_\t$. To simplify the notation, we let
\[
\bF = \{\mbf_0,\mbf_1,\dots,\mbf_n\}\quad \text{and}\quad
\bT = \{\mbt_0,\mbt_1,\dots,\mbt_n\}.
\]
\end{df}

Note that $\JB_0$ is isomorphic to
Belnap's four-element bilattice, $\FOUR$,  and
the bilattice $\JB_n$ generalises Belnap's bilattice by taking the truth values
$\mbf$ and $\mbt$ and expanding them to create a chain of truth values in each of their places. Moreover, as the following simple proposition shows, $\JB_n$ has a homomorphic image that is term equivalent to Belnap's bilattice.

Let $\JB_{0,n}$ be an algebra in the signature of $\JB_n$ which has $\langle J_0; \otimes, \oplus, \wedge, \vee, \neg\rangle$ as its
bilattice reduct and in which the additional
constants $\mbf_1,\ldots,\mbf_n$ take the value~$\mbf_0$ and the additional
constants $\mbt_1, \ldots, \mbt_n$ take the value~$\mbt_0$. 
Clearly, $\JB_{0,n}$ is term equivalent to $\JB_0$. The following observation is immediate.

\begin{prop}\label{prop:hnn}
For all $n\in \omega$, the equivalence relation $\theta$ with blocks
$\{\top\}$, $\bF$, $\bT$ and $\{\bot\}$
is a congruence on $\JB_n$ with $\JB_n/{\theta} \cong \JB_{0,n}$. Hence $\JB_{0,n}$ is a homomorphic image of $\JB_n$.
\end{prop}

We close this section with some remarks about the congruence lattice of $\JB_n$ and the structure of the variety generated by~$\JB_n$.

\goodbreak

\begin{lem} \label{lem:ConJn}
Let $n\in \omega$.
\begin{enumerate}[\normalfont(1)]

\item Let $\theta$ be an equivalence relation obtained by independently collapsing any collection of the pairs
$(\mbf_0, \mbf_1), \dots, (\mbf_{n-1}, \mbf_n)$,
and the corresponding pairs in $\bT$, and collapsing no other elements of~$J_n$. Then $\theta$ is a congruence on~$\JB_n$. Moreover, every non-trivial congruence on $\JB_n$ arises this way.

\item
$\Con(\JB_n) \cong \two^n \oplus \one$ \textup(i.e., $\two^n$ with a new top adjoined\textup).
\end{enumerate}
\end{lem}

\begin{proof}
We prove only (1) as (2) is an immediate consequence. It is clear that $\theta$ is a congruence on~$\JB_n$. Now let $\alpha$ be a congruence on~$\JB_n$. It is easily seen that if $\top/{\alpha} = \{\top\}$ and $\bot/{\alpha} = \{\bot\}$, then $\alpha$ is of the form described.
It remains to prove that if $\top/{\alpha} \ne \{\top\}$, then $\alpha = J_n^2$ (the other case follows by duality).
Assume that $c\in J_n\comp\{\top\}$ with $c \equiv_\alpha \top$. If $c = \bot$, then we are done, so we may assume that $c\notin \{\top, \bot\}$.
Hence $\bot = c \otimes \neg c \equiv_\alpha \top \otimes \neg\top = \top \otimes \top =\top$,
and again we are done.
\end{proof}

Let $n\in \omega\setminus\{0\}$ and let $k\in \{1, \dots, n\}$. Define $\JB_{1,n,k}$ to be the algebra in the signature of $\JB_n$ that has bilattice reduct $\langle \{\top, \zero, \mbf, \one, \mbt, \bot\};
\otimes, \oplus, \wedge, \vee, \neg\rangle$ isomorphic to the bilattice reduct of~$\JB_1$, as shown in Figure~\ref{fig:J1nm}.

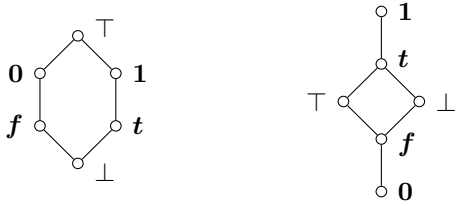
\begin{figure}[ht]
\centering
\begin{tikzpicture}
\begin{scope}[scale=0.5]
  \node[unshaded] (bot) at (0,0) {};
  \node[unshaded] (f) at (-1,1) {};
  \node[unshaded] (t) at (1,1) {};
  \node[unshaded] (0) at (-1,2.4) {};
  \node[unshaded] (1) at (1,2.4) {};
  \node[unshaded] (top) at (0,3.4) {};
  \draw[order] (bot) -- (f) -- (0) -- (top);
  \draw[order] (bot) -- (t) -- (1) -- (top);
  \node[label,anchor=west,yshift=-3pt] at (bot) {$\bot$};
  \node[label,anchor=east] at (f) {$\mbf$};
  \node[label,anchor=west] at (t) {$\mbt$};
  \node[label,anchor=east] at (0) {$\zero$};
  \node[label,anchor=west] at (1) {$\one$};
  \node[label,anchor=west,yshift=3pt] at (top) {$\top$};
\end{scope}
\begin{scope}[scale=0.5,xshift=8cm,yshift=-0.75cm]
  \node[unshaded] (0) at (0,0) {};
  \node[unshaded] (f) at (0,1.4) {};
  \node[unshaded] (top) at (-1,2.4) {};
  \node[unshaded] (bot) at (1,2.4) {};
  \node[unshaded] (t) at (0,3.4) {};
  \node[unshaded] (1) at (0,4.8) {};
  \draw[order] (0) -- (f) -- (top) -- (t) -- (1);
  \draw[order] (f) -- (bot) -- (t);
  \node[label,anchor=west] at (bot) {$\bot$};
  \node[label,anchor=west,yshift=-2pt] at (f) {$\mbf$};
  \node[label,anchor=west,yshift=2pt] at (t) {$\mbt$};
  \node[label,anchor=west] at (0) {$\zero$};
  \node[label,anchor=west] at (1) {$\one$};
  \node[label,anchor=east] at (top) {$\top$};
\end{scope}
\end{tikzpicture}
\caption{The  bilattice reduct of $\JB_{1,n,k}$.}\label{fig:J1nm}
\end{figure}
In $\JB_{1,n,k}$, the constants $\mbf_0,\ldots,\mbf_{k-1}$ take the value $\zero$  and $\mbf_k, \ldots, \mbf_n$ take the value~$\mbf$, while the constants $\mbt_0,\ldots,\mbt_{k-1}$ take the value~$\one$  and $\mbt_k, \ldots, \mbt_n$ take the value~$\mbt$. 
Clearly, $\JB_{1,n,k}$ is term equivalent to $\JB_1$.
Let $\theta_k$ be the equivalence relation on $J_n$ with blocks
\[
\{\top\}, \ \ \{\mbf_0, \dots, \mbf_{k-1}\}, \ \  \{\mbf_k, \dots, \mbf_n\}, \ \  \{\mbt_0, \dots, \mbt_{k-1}\}, \ \  \{\mbt_k, \dots, \mbt_n\}, \ \  \{\bot\}.
\]
By Lemma~\ref{lem:ConJn}, the relation $\theta_k$ is a congruence on $\JB_n$. Clearly, $\JB_n/{\theta_k} \cong \JB_{1,n,k}$. Hence $\JB_{1,n,k}$ is a homomorphic image of $\JB_n$.

For all $n\in \omega$, let $\CV_n =\HSP(\JB_n)$ be the variety generated by $\JB_n$.

\begin{prop}\label{prop:thevariety}\
\begin{enumerate}[\normalfont(1)]

\item Up to isomorphism, the only subdirectly irreducible algebra in the variety $\CV_0$ is $\JB_0$ itself.

\item Let $n\in \omega\setminus\{0\}$. Up to isomorphism,
the variety $\CV_n$ contains $n+1$ subdirectly irreducible algebras, the four-element algebra $\JB_{0,n}$ and the six-element algebras $\JB_{1,n,k}$, for $k\in \{1, \dots, n\}$.

\item The algebras $\JB_{0,n}$ and $\JB_{1,n,k}$, for $k\in \{1, \dots, n\}$, are injective in $\CV_n$.

\item Every algebra in $\CV_n$ embeds into an injective algebra in $\CV_n$.

\item The variety $\CV_n$ has the congruence extension property and the amalgamation property.

\end{enumerate}

\end{prop}

\begin{proof}
Since $\CV_n$ is congruence distributive, a simple application of
J\'onsson's Lemma~\cite[Cor.~3.4]{Jon} tells us that the
subdirectly irreducible algebras in $\CV_n$ are the subdirectly irreducible homomorphic images of~$\JB_n$. We know from Lemma~\ref{lem:ConJn} that $\JB_n$ has $n+1$ meet-irreducible congruences: the unique coatom and its $n$ lower covers. The corresponding subdirectly irreducible quotients of $\JB_n$ are $\JB_{0,n}$ and $\JB_{1,n,k}$, for $k\in \{1, \dots, n\}$.

Since each of these subdirectly irreducible algebras has no proper subalgebras, it is clear that each is injective in the class of subdirectly irreducible algebras in $\CV_n$, and hence each is injective in the variety~$\CV_n$---see Davey~\cite[Corollary 2.3]{D77}. Consequently, every algebra in $\CV_n$ embeds into an injective algebra, from which it follows that $\CV_n$ satisfies both the congruence extension property and the amalgamation property. (See, for example, Taylor~\cite[Theorem 2.3]{T72}.)
\end{proof}

\section{Natural dualities}\label{sec:NatDualities}
It is important to note that a product representation theorem exists for both
distributive bilattices~\cite[Proposition 8]{Fit90} and interlaced pre-bilattices.
(See Davey~\cite{D-PR} for a full historical account.) These representations have been used extensively in the study of bilattices and pre-bilattices.
Duality and representation theorems for bilattices have largely focussed
on product representations. Mobasher, Pigozzi, Slutzki and
Voutsadakis~\cite{Mob00} used the product representation of distributive
bilattices to show that the category of distributive bilattices and the category of
Priestley spaces are dually equivalent. Jung and Rivieccio~\cite{JR12} defined Priestley bispaces and showed that this
new category is dually equivalent to the category
of distributive bilattices. For $n\geqslant 1$, the bilattice $\JB_n$ is not interlaced:
indeed, $\mbf_0 \le_\k \top$ but $\mbf_0 \wedge \bot = \mbf_0 \nleqslant_\k \mbf_n = \top \wedge \bot$.
Hence we are not able to use a product representation to study either the variety or the quasivariety generated by~$\JB_n$.
We will turn to natural duality theory in order to study this new class of default bilattices.

In its simplest form, the theory of natural dualities concerns quasivarieties $\CA =\ISP(\M)$ of algebras generated by a finite algebra~$\M$.
We can always find a discretely topologised structure $\MT$ with the same underlying set $M$ as the algebra $\M$ such that there is a dual adjunction between the quasivariety $\CA$ and the `topological quasivariety' $\CX = \IScP(\MT)$ of topological structures generated by~$\MT$.
As the class operators indicate, the objects of the category $\CX$ are the
isomorphic copies of closed substructures of non-zero powers of the generating structure~$\MT$.
The morphisms of $\CA$ and $\CX$ \emph{qua} categories are all possible homomorphisms and all possible continuous structure-preserving maps, respectively.
The aim is to find a structure $\MT$ such that $\CA$ is dually equivalent to a full subcategory of $\CX$ (\emph{duality}), or better still dually equivalent to $\CX$ itself (\emph{full duality})---see below for the formal definitions.

Two examples of such natural dualities are Stone duality for
Boolean algebras and Priestley duality for distributive lattices.
In both cases the algebra $\M$ has underlying set $\{ 0,1\}$.
In the Boolean case, $\MT$~is the set $\{ 0,1\}$ equipped with just the discrete topology
and no operations nor relations. In the case of  distributive lattices, $\MT$
is $\{ 0,1\}$ equipped with the discrete topology, two constants $0$ and $1$,
and the usual order relation~$\le$, with $0 < 1$.

Let $\M$ be a finite algebra. We search for structures $\MT = \langle  M; \mathcal G, \mathcal H, \mathcal R, \T \rangle$, where $\T$ is the discrete topology and $\mathcal G$, $\mathcal H$ and $\mathcal R$ are sets of finitary operations, partial operations and  relations, respectively, such that the relations in $\mathcal R$
and the graphs of the (partial) operations in $\mathcal G \cup \mathcal H$ are non-empty subuniverses
of  finite powers of~$\M$. If this is the case, we say that the operations, partial operations and relations are \defn{compatible with}~$\M$ (or \defn{algebraic over}~$\M$). We also say that the structure $\MT$ is \defn{compatible with}~$\M$ (or \defn{algebraic over}~$\M$). The structure $\MT$ is referred to as an \defn{alter ego} of~$\M$.

Given an alter ego $\MT$ of $\M$, there is a natural method to obtain a
dual adjunction between $\CA$ and~$\CX$. We denote by $\CA(\A,\M)$ the
set of all $\CA$-homomorphisms from $\A$ into $\M$, and by $\CX(\X,\MT)$ the set of all
continuous structure-preserving maps from $\X$ into $\MT$.
The dually adjoint hom-functors $\mathrm{D}\colon\CA\to\CX$ and $\mathrm{E} \colon \CX \to \CA$ are defined at both the object- and morphism-levels
below.
\begin{alignat*}{2}
& \mathrm{D} \colon \CA \to \CX, \mathrm{D}(\A) := \CA(\A,\M) 
&&(\text{as a closed substructure of the
power } \MT^A) \\
& \mathrm{D} \colon \CA(\A,\B) \to \CX(\mathrm{D}(\B),\mathrm{D}(\A)), \ &&\mathrm{D}(u)(x):= x \circ u, \text{ for } x \in \mathrm{D}(\B). \end{alignat*}
The structure on $\MT = \langle M; \mathcal G,\mathcal H,\mathcal R, \T \rangle$ is extended pointwise to $\CA(\A,\M)$.
For example, if $R$ is an $n$-ary relation in $\mathcal R$, then for $x_1, \dots, x_n \in \CA(\A,\M)$, we have $(x_1, \ldots, x_n) \in R^{\mathrm{D}(\A)}$ if and only if
$(x_1(a), \ldots,x_n(a)) \in R$, for all $a \in A$.)
\begin{alignat*}{2}
&\mathrm{E} \colon \CX \to \CA, \quad \mathrm{E}(\X) := \CX(\X,\MT) \ &&(\text{as a subalgebra of the power
}\M^X) \\
&\mathrm{E} \colon \CX(\X,\Y) \to \CA(\mathrm{E}(\Y),\mathrm{E}(\X)), &&\mathrm{E}(\varphi)(\alpha) := \alpha \circ \varphi, \:\text{ for } \alpha \in \mathrm{E}(\Y). 
\end{alignat*}
The Pre-duality Theorem~\cite[Theorem 1.5.2]{CD98} confirms that these functors are
well defined. What is important here is the fact that the operations, partial operations and relations in $\mathcal G\cup \mathcal H \cup \mathcal R$ are compatible with $\M$.

Given the above setup, we can define embeddings
$e_\A \colon \A \to \mathrm{ED}(\A)$ and $\varepsilon_\X \colon \X \to \mathrm{DE}(\X)$ by
{\allowdisplaybreaks\begin{align*}
e_\A(a)(x)&=x(a), \quad \text{for }a \in A \text{ and } x \in \mathrm{D}(\A), \\
\varepsilon_\X(x)(\alpha)&=\alpha(x), \quad \text{for }x \in X \text{ and } \alpha \in \mathrm{E}(\X).
\end{align*}}
We say that $\MT$ yields a \defn{duality} on $\CA$ (or that $\mathcal G\cup \mathcal H\cup \mathcal R$ yields a \defn{duality} on $\CA$)  if, for every $\A \in \CA$, the embedding $e_\A$ is an isomorphism. We say that $\MT$ yields a \defn{full duality} on $\CA$ if $\MT$ yields a duality on $\CA$ and, for every $\X\in\CX$, the embedding $\varepsilon_\X$ is an isomorphism. If $\MT$ yields a full duality on $\CA$ and $\MT$ is injective in the category $\CX$, then $\MT$ is said to yield a \defn{strong duality} on $\CA$.

When $\M$ is a finite lattice-based algebra, we are able to apply a very powerful theorem to help us find an appropriate dualising
structure $\MT$. The NU Duality Theorem~\cite[Theorem 2.3.4]{CD98} is in fact much more general than the statement given below, but
this special case will be sufficient for our needs. Note that, since $\M$ is lattice based, it has a ternary NU term, namely the lattice median.

\begin{thm}
[\textbf{Special NU Duality Theorem}]\label{NUDT}
Let $\M$ be a finite lattice-based
algebra and let $\mathcal R_\M$ denote the set of non-empty subuniverses of~$\M^2$. Then  $\MT = \langle M; \mathcal R_\M, \T \rangle$ yields a duality on  $\ISP(\M)$.
\end{thm}
In general, the set $\Sub(\M^2)$ of subuniverses of $\M^2$
can be extremely large, even when $\M$ is a small algebra.
For example,
computer calculations reveal that $|\Sub(\JB_3^2)|=200$.
Although we are guaranteed a duality via the entire set $\mathcal R_\M$
of compatible binary relations, we want to reduce the size of the set of relations, ideally to some minimal set.

The first application of natural duality to bilattices was
by Cabrer and Priestley~\cite{CPdbl}, who looked at both bounded and unbounded distributive bilattices. They showed that the knowledge order alone yields a duality on the class $\ISP(\JB_0)$ of  bounded distributive bilattices.
(Except for $\JB_0$,
in our class of bilattices the truth operations do not preserve $\le_\k$,
and hence $\le_\k$ is not a compatible relation and cannot be used in the alter ego.)

\begin{thm}[{\cite[Theorem 4.2]{CPdbl}}]\label{thm:lekdualfour}
Consider the four-element bilattice
\[
\JB_0 = \langle \{\top,\mbf_0,\mbt_0,\bot\}; \otimes, \oplus,\wedge,\vee, \neg, \top, \mbf_0,\mbt_0,\bot\rangle \cong \FOUR.
\]
The alter ego
\[
\JT_0 = \langle \{\top,\mbf_0,\mbt_0,\bot\}; \le_\k, \T \rangle
\]
yields a strong,
and therefore full, duality on $\CV_0 = \ISP(\JB_0)$.
\end{thm}

Let $n\in \omega\comp\{0\}$.
It follows from Proposition~\ref{prop:thevariety} that the variety $\CV_n$ generated by $\JB_n$ satisfies
\[
\CV_n = \ISP(\{\M_0,\M_1,\ldots, \M_n\}),
\]
where $\M_0 = \JB_{0,n}$ and $\M_k = \JB_{1,n,k}$, for $k\in \{1, \dots, n\}$,  are the subdirectly irreducible algebras in~$\CV_n$.
Consequently, it is natural to find a multi-sorted duality for $\CV_n$  using $M_0, \dots, M_n$ as the sorts. We shall give a brief introduction to multi-sorted dualities in the special case that the algebras are lattice based. We refer to Davey and Priestley~\cite[Section 2]{DP-multi}, where they were first introduced, and to Clark and Davey~\cite[Chapter~7]{CD98} for a detailed discussion of multi-sorted dualities in general.

Let $\{\M_0,\dots, \M_n\}$ be a set of finite, pairwise non-isomorphic lattice-based algebras of the same signature and let $\CA = \ISP(\{\M_0,\dots, \M_n\})$ be the quasivariety generated by them.
We shall refer to a non-empty subuniverse of $\M_j \times \M_k$ as a \emph{compatible relation from $\M_j$ to~$\M_k$}, for all $j, k \in \{0, \dots, n\}$.
As an alter ego for the set $\{\M_0,\dots, \M_n\}$, we will use a multi-sorted structure of the following kind:
\[
\MT = \langle M_0\du \cdots \du M_n; \mathcal G, \mathcal R, \T\rangle,
\]
where, for each $g\in \mathcal G$, there exist $j, k \in \{0, \dots, n\}$, such that ${g\colon \M_j \to \M_k}$ is a homomorphism, each relation $R\in
\mathcal R$ is a compatible relation from $\M_j$~to~$\M_k$, for some
$j, k \in \{0, \dots, n\}$, and $\T$ is the disjoint union topology
obtained from the discrete topology on the sorts.
(In general,
multi-sorted operations and relations of higher arity are permitted,
but
we do not require them.)

Objects in the dual category will now be multi-sorted Boolean
topological structures $\X$ in the signature of~$\MT$. Thus,
$\X = \langle X_0 \du \cdots \du X_n ; \mathcal G^\X, \mathcal R^\X,
\T^\X\rangle$, where each $X_j$ carries a Boolean topology and
$\T^\X$ is the corresponding disjoint-union topology, if $g\colon
\M_j\to \M_k$ is in $\mathcal G$, then the corresponding $g^\X\in
\mathcal G^\X$  is a continuous map $g^\X \colon X_j \to X_k$, and
if $R\in \mathcal R$ is a relation from $\M_j$ to $\M_k$, then the
corresponding $R^\X\in \mathcal R^\X$ is a topologically closed
subset of $X_j\times X_k$. Given two such multi-sorted topological
structures $\X$ and $\Y$, a morphism $\varphi\colon \X \to \Y$ is a
continuous map that preserves sorts (so $\varphi(X_j) \subseteq
Y_j$, for all~$j$) and preserves the operations and relations.

For a non-empty set $S$, the power $\MT^S$ is defined in the natural sort-wise way; the underlying set of $\MT^S$ is $M_0^S \du \dots \du M_n^S$ and the operations and relations between the sorts are defined pointwise. The potential dual category is now defined to be the category
$\CX =\IScP(\MT)$ whose objects are isomorphic copies of topologically closed substructures of non-zero powers of~$\MT$, where substructure has its natural multi-sorted meaning.

Given an algebra $\A \in \CA$, its dual $\mathrm{D}(\A)\in \CX$ is defined to be
\[
\mathrm{D}(\A) := \CA(\A, \M_0) \du \dots \du \CA(\A, \M_n),
\]
as a topologically closed substructure of $\MT^A$.
Given $\X\in \CX$, its dual $\mathrm{E}(\X) \in \CA$ is defined to be
\[
\mathrm{E}(\X) := \CX(\X, \MT) \le \M_0^{X_0} \times \dots \times \M_n^{X_n}.
\]
The fact that the structure on $\MT$ is compatible with the set $\{\M_0,\ldots, \M_n\}$ guarantees that $\mathrm{E}(\X)$ is a subalgebra of $\M_0^{X_0} \times \dots \times \M_n^{X_n}$ and hence $\mathrm{E}$ is well defined. The definitions of $\mathrm{D}$ on homomorphisms and $\mathrm{E}$ on morphisms are the natural extensions of the single-sorted case defined in full above.

The definitions of the unit $e_\A \colon \A \to
\mathrm{ED}(\A)$ and counit $\varepsilon_\X \colon \X \to \mathrm{DE}(\X)$
in the single-sorted case extend naturally to this multi-sorted setting.
The concepts of \emph{duality}, \emph{full duality} and
\emph{strong duality} are defined exactly as they were in the
single-sorted case.

We now present a version of the Multi-sorted NU Strong Duality Theorem~\cite[Theorem~7.1.2]{CD98} that is tailored to the variety~$\CV_n$.

\begin{thm}[\textbf{Special Multi-sorted NU Strong Duality Theorem}]\label{MultiNUDT}
Assume that $\M_0$, \dots, $\M_n$ are finite,
pairwise non-isomorphic lattice-based
algebras of the same signature. Assume also that, for all $k\in \{0, \dots, n\}$, the algebra $\M_k$ is subdirectly irreducible and every element of $\M_k$ is a constant.  Define
\[
\MT = \langle  M_0 \du \cdots\du  M_n;
\mathcal G, \mathcal R, \T \rangle,
\]
where $\mathcal G = \bigcup \{\, \CA(\M_j, \M_k)\mid j, k\in\{0,\dots,n\}\,\}$ is the set of all homomorphisms between the sorts and $ \mathcal R= \bigcup \{\, \Sub(\M_j\times\M_k)\mid j, k\in\{0,\dots,n\}\,\}$ is the set of all compatible relations between the sorts. Then $\MT$ yields a multi-sorted strong, and therefore full, duality on $\ISP(\{\M_0,\dots, \M_n\})$.
\end{thm}

Assume that $\JT_n$ yields a single-sorted duality on the quasivariety $\CJ_n = \ISP(\JB_n)$ and that $\MT_n$ yields an $(n{+}1)$-sorted duality on the variety $\CV_n = \HSP(\JB_n) = \ISP(\{\M_0,\dots, \M_n\})$.
The $S$-generated free algebra $\F_{\CV_n}\!(S)$ in the variety $\CV_n$ is isomorphic to the subalgebra of $\JB_n^{J_n^S}$ generated by the projections and so lies in the quasivariety $\CJ_n$. We can therefore use either duality to find the free algebras in~$\CV_n$.
Indeed, $\F_{\CV_n}\!(S)$ is isomorphic to $\mathrm{E}(\JT_n^S)$ in the single-sorted case and $\mathrm{E}(\MT_n^S)$ in the multi-sorted case. The difference is that $\mathrm{E}(\JT_n^S)$ represents $\F_{\CV_n}\!(S)$ as a subalgebra of $\JB_n^{J_n^S}$ while $\mathrm{E}(\MT_n^S)$ represents $\F_{\CV_n}\!(S)$ as a subalgebra of $\M_0^{M_0^S} \times \dots \times \M_n^{M_n^S}$---see Remark~\ref{rem:compare}.

\section{A natural duality for the quasivariety $\CJ_n$}\label{sec:CJduality}

In this section, we describe an alter ego of $\JB_n$ that yields an optimal duality on the quasivariety $\CJ_n = \ISP(\JB_n)$ generated by~$\JB_n$. The proof that the alter ego yields a duality is in Section~\ref{sec:singleproof} and its optimality is proved in Section~\ref{sec:opt}.

\subsection{The relation $S_{n,n}$}\label{subsec:Snn}
Let $n\in \omega$. Define the subset $S_{n,n}$ of $J_n^2$ by
\[
S_{n,n} := \big(J_n \times \{\top\}\big) \cup \big(\{\bot\} \times J_n\big) \cup   \bF^2 \cup \bT^2.
\]
The relation $S_{n,n}$ is a quasi-order on $J_n$---see Figure~\ref{fig:Sni}.
(When depicting a quasi-order $R$, we draw $x$ and $y$ in the same block if $x\mathbin{R}y$ and $y\mathbin{R}x$.)

\begin{figure}[ht]
\centering
\begin{tikzpicture}
\node at (-1.75,0) {$S_{n,n}$};
\node[quasi] (bot) at (0,0) {$\bot$};
\node[quasi] (f) at (-1.25,1) {$\mbf_0,\dots,\mbf_n$};
\node[quasi] (t) at (1.25,1) {$\mbt_0,\dots,\mbt_n$};
\node[quasi] (top) at (0,2) {$\top$};
\draw[order] (bot) -- (f) -- (top);
\draw[order] (bot) -- (t) -- (top);
\end{tikzpicture}
\caption{The binary relation $S_{n,n}$ drawn as a quasi-order.}\label{fig:Sni}
\end{figure}
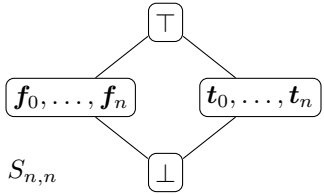

Note that the relation $S_{0,0}$ is just the knowledge order $\le_\k$ on $\JB_0$.
For $n > 0$, the quasi-order $S_{n,n}$ is not an order.

\subsection{The relation $S_{n,i}$}\label{subsec:Sni}
For $n \in \omega\setminus \{0\}$ and $i \in \{0,\ldots,n-1\}$, define the subset $S_{n,i}$ of $J_n^2$ by:
\begin{multline*}
S_{n,i} := \{(\top, \top), (\bot, \bot) \}\cup \Big(\bF^2\setminus \big(\{\mbf_0, \dots, \mbf_i\}\times \{\mbf_{i+1},\dots, \mbf_n\}\big)\Big)\notag\\
\cup \Big(\bT^2\setminus \big(\{\mbt_0, \dots, \mbt_i\}\times \{\mbt_{i+1},\dots, \mbt_n\}\big)\Big).
\end{multline*}

The relation $S_{n,i}$ is also a quasi-order on $J_n$---see Figure~\ref{fig:RforJ3}. Note that, unlike the quasi-order $S_{n,n}$,  both $\top$ and $\bot$ are isolated in the quasi-order~$S_{n,i}$. The quasi-order $S_{n,i}$ is an order if and only if $n = 1$ and $i = 0$.

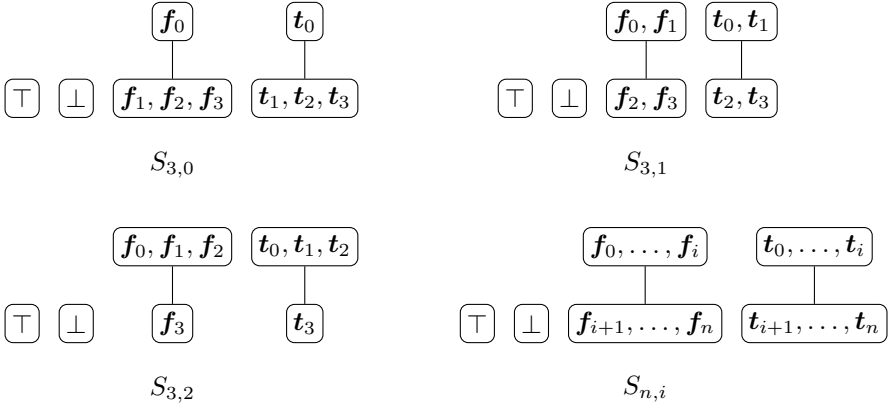
\begin{figure}[!ht]
\centering
\begin{tikzpicture}
\begin{scope}
  \node[quasi] (top) at (0,0) {$\top$};
  \node[quasi,anchor=west,right=0.25cm of top.east] (bot) {$\bot$};
  \node[quasi,anchor=west,right=0.25cm of bot.east] (f1) {$\mbf_1,\mbf_2,\mbf_3$};
  \node[quasi,anchor=south,above=0.5cm of f1.north] (f0) {$\mbf_0$};
  \node[quasi,anchor=west,right=0.25cm of f1.east] (t1) {$\mbt_1,\mbt_2,\mbt_3$};
  \node[quasi,anchor=south,above=0.5cm of t1.north] (t0) {$\mbt_0$};
  \draw[order] (f1) -- (f0);
  \draw[order] (t1) -- (t0);
  \node[anchor=base,below=0.5cm of f1.base] {$S_{3,0}$};
\end{scope}
\begin{scope}[xshift=6.5cm] 
  \node[quasi] (top) at (0,0) {$\top$};
  \node[quasi,anchor=west,right=0.25cm of top.east] (bot) {$\bot$};
  \node[quasi,anchor=west,right=0.25cm of bot.east] (f2) {$\mbf_2,\mbf_3$};
  \node[quasi,anchor=south,above=0.5cm of f2.north] (f0) {$\mbf_0,\mbf_1$};
  \node[quasi,anchor=west,right=0.25cm of f2.east] (t2) {$\mbt_2,\mbt_3$};
  \node[quasi,anchor=south,above=0.5cm of t2.north] (t0) {$\mbt_0,\mbt_1$};
  \draw[order] (f2) -- (f0);
  \draw[order] (t2) -- (t0);
  \node[anchor=base,below=0.5cm of f2.base] {$S_{3,1}$};
\end{scope}
\begin{scope}[yshift=-3cm]
  \node[quasi] (top) at (0,0) {$\top$};
  \node[quasi,anchor=west,right=0.25cm of top.east] (bot) {$\bot$};
  \node[quasi,anchor=south west,above right=0.5cm and 0.25cm of bot.north east] (f0) {$\mbf_0,\mbf_1,\mbf_2$};
  \node[quasi,anchor=north,below=0.5cm of f0.south] (f3) {$\mbf_3$};
  \node[quasi,anchor=west,right=0.25cm of f0.east] (t0) {$\mbt_0,\mbt_1,\mbt_2$};
  \node[quasi,anchor=north,below=0.5cm of t0.south] (t3) {$\mbt_3$};
  \draw[order] (f3) -- (f0);
  \draw[order] (t3) -- (t0);
  \node[anchor=base,below=0.5cm of f3.base] {$S_{3,2}$};
\end{scope}
\begin{scope}[xshift=6.0cm,yshift=-3cm] 
  \node[quasi] (top) at (0,0) {$\top$};
  \node[quasi,anchor=west,right=0.25cm of top.east] (bot) {$\bot$};
  \node[quasi,anchor=west,right=0.25cm of bot.east] (fn) {$\mbf_{i+1},\dots,\mbf_n$};
  \node[quasi,anchor=south,above=0.5cm of fn.north] (f0) {$\mbf_0,\dots,\mbf_i$};
  \node[quasi,anchor=west,right=0.25cm of fn.east] (tn) {$\mbt_{i+1},\dots,\mbt_n$};
  \node[quasi,anchor=south,above=0.5cm of tn.north] (t0) {$\mbt_0,\dots,\mbt_i$};
  \draw[order] (fn) -- (f0);
  \draw[order] (tn) -- (t0);
  \node[anchor=base,below=0.5cm of fn.base] {$S_{n,i}$};
\end{scope}
\end{tikzpicture}
\caption{The binary relations $S_{3,0}$, $S_{3,1}$, $S_{3,2}$ and $S_{n,i}$ drawn as quasi-orders.}\label{fig:RforJ3}
\end{figure}

\subsection{The relation $R_{n,i,j}$}\label{subsec:Rnij}
For $n\in \omega\setminus\{0, 1\}$ and $i,j \in \{0,\ldots,n-1\}$, we also need the union
\[
R_{n, i, j}:= S_{n,i}\cup S_{n,j}.
\]
It is easily seen that, if $i < j$, then
\begin{multline*}
R_{n, i, j} = \{(\top, \top), (\bot, \bot) \} \cup \Big(\bF^2\setminus (\{\mbf_0, \dots, \mbf_i\}\times \{\mbf_{j+1},\dots, \mbf_n\}\big)\Big)\\
 \cup \Big(\bT^2\setminus \big(\{\mbt_0, \dots, \mbt_i\}\times \{\mbt_{j+1},\dots, \mbt_n\}\big)\Big).
\end{multline*}

We shall see in Section~\ref{sec:CompOnHoms} that each of the relations $S_{n,n}$, $S_{n,i}$ and $R_{n,i,j}$ defined above is a compatible relation on $\JB_n$ and hence may be used as part of the structure on an alter ego of~$\JB_n$---see Lemma~\ref{lem:Sle=Sgle} for $S_{n,n}$ and Lemma~\ref{prop:Sab-subalg} for $S_{n,i}$ and $R_{n,i,j}$.

We can now state our single-sorted duality theorem. The dualities for $\CJ_1$ and $\CJ_2$ were obtained via computer calculations in the first author's DPhil thesis~\cite{C-thesis}.

\begin{thm}\label{cor:bigduality}
Let $n\in \omega$. Define the alter ego $\JT_n = \langle J_n; \mathcal{R}_{(n)}, \T\rangle$ of $\JB_n$, where
$\mathcal R_{(n)}$ is the set of compatible binary relations on $J_n$ given by
\begin{alignat*}{2}
\mathcal R_{(0)} &= \{S_{0,0}\}, &\quad \mathcal R_{(1)} &= \{S_{1,0}, S_{1,1}\},\\
\mathcal R_{(2)} &= \{S_{2,0}, S_{2,1}, S_{2,2}\}, &\quad \mathcal R_{(3)} &= \{S_{3,0}, S_{3,1}, S_{3,2}, S_{3,3}, R_{3,0,2}\},
\end{alignat*}
and, in general, for $n\ge 3$,
\[
\mathcal R_{(n)} = \big\{\,S_{n,i} \mid 0\le i \le n\,\big\} \cup  \big\{\,R_{n,i,j}\mid i,j \in \{0,\ldots,n-1\} \text{ with } i < j-1\,\big\}.
\]
\begin{enumerate}[\normalfont(1)]

\item
The alter ego $\JT_n$ yields an optimal duality on $\CJ_n =\ISP(\JB_n)$.

\item $\JT_0$ and $\JT_1$ yield strong, and therefore full, dualities on $\CJ_0$ and $\CJ_1$, respectively.

\item For all $n\ge 2$, the duality on $\CJ_n$ can be upgraded to a strong, and therefore full, duality by adding all compatible $n$-ary partial operations on $\JB_n$ to the structure of the alter ego $\JT_n$.

\item $|\mathcal R_{(0)}| = 1$ and $|\mathcal R_{(n)}| = \frac12(n^2 - n +4)$, for all $n\in \omega\setminus\{0\}$.

\end{enumerate}
\end{thm}

Note that when $n = 0$, the duality is the strong duality given by Cabrer and Priestley~\cite{CPdbl} as stated in our Theorem~\ref{thm:lekdualfour}.

\section{A natural duality for the variety $\CV_n$}\label{sec:CVduality}

Fix $n\in \omega\comp\{0\}$. In this section, we describe a strong, multi-sorted natural duality for the variety $\CV_n$ generated by~$\JB_n$. The proof that the alter ego yields a strong duality is contained in Section~\ref{sec:multiproof} and its optimality is proved in Section~\ref{sec:optmulti}.

To simplify the notation, we denote $\JB_{0,n}$ by $\M_0$ and, for $k\in \{1, \dots,n\}$, we denote $\JB_{1,n,k}$ by~$\M_k$. Throughout this section, we shall label the
elements of $\M_0$ and $\M_k$, for $k\in \{1, \dots, n\}$, as shown in Figure~\ref{fig:multi}.

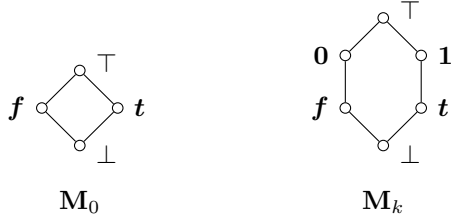
\begin{figure}[hbt]
\centering
\begin{tikzpicture}
\begin{scope}[scale=0.5]
  \node at (0,-1.5) {$\M_0$};
  \node[unshaded] (bot) at (0,0) {};
  \node[unshaded] (f) at (-1,1) {};
  \node[unshaded] (t) at (1,1) {};
  \node[unshaded] (top) at (0,2) {};
  \draw[order] (bot) -- (f) -- (top);
  \draw[order] (bot) -- (t) -- (top);
  \node[label,anchor=west,yshift=-3pt] at (bot) {$\bot$};
  \node[label,anchor=east] at (f) {$\mbf$};
  \node[label,anchor=west] at (t) {$\mbt$};
  \node[label,anchor=west,yshift=3pt] at (top) {$\top$};
\end{scope}
\begin{scope}[scale=0.5,xshift=8cm]
  \node at (0,-1.5) {$\M_k$};
  \node[unshaded] (bot) at (0,0) {};
  \node[unshaded] (f) at (-1,1) {};
  \node[unshaded] (t) at (1,1) {};
  \node[unshaded] (0) at (-1,2.4) {};
  \node[unshaded] (1) at (1,2.4) {};
  \node[unshaded] (top) at (0,3.4) {};
  \draw[order] (bot) -- (f) -- (0) -- (top);
  \draw[order] (bot) -- (t) -- (1) -- (top);
  \node[label,anchor=west,yshift=-3pt] at (bot) {$\bot$};
  \node[label,anchor=east] at (f) {$\mbf$};
  \node[label,anchor=west] at (t) {$\mbt$};
  \node[label,anchor=east] at (0) {$\zero$};
  \node[label,anchor=west] at (1) {$\one$};
  \node[label,anchor=west,yshift=3pt] at (top) {$\top$};
\end{scope}
\end{tikzpicture}
\caption{The  bilattice reducts of $\M_0$ and $\M_k$ in their knowledge order. \label{fig:multi}}
\end{figure}
Note that, for all $k\in \{1, \dots, n\}$, the algebra  $\M_k$ has underlying set
\[
M_k = \{\top, \zero, \mbf, \one, \mbt, \bot\}.
\]
Thus, all algebras $\M_k$, with $k\in\{1, \dots,n\}$, have the same  bilattice reduct but their constants $\mbf_0, \ldots, \mbf_n$ and $\mbt_0, \ldots, \mbt_n$ take different values:
\begin{itemize}

\item in $\M_0$ all of the `false' constants $\mbf_0,\ldots,\mbf_n$ 
take the value $\mbf$ and all of the `true' constants $\mbt_0, \ldots, \mbt_n$ take the value $\mbt$,

\item in $\M_k$, for $k\in \{1, \dots, n\}$, the constants $\mbf_0,\ldots,\mbf_{k-1}$ take the value $\zero$  and $\mbf_k, \ldots, \mbf_n$ 
take the value ~$\mbf$, and similarly the constants $\mbt_0,\ldots,\mbt_{k-1}$ 
take the value $\one$  and $\mbt_k, \ldots, \mbt_n$ 
take the value ~$\mbt$.
\end{itemize}

It follows from Proposition~\ref{prop:thevariety} that
\[
\CV_n = \ISP(\{\M_0, \M_1, \dots, \M_n\}); 
\]
so we will use an alter ego with $n + 1$ sorts. Strictly
speaking, to make the sorts disjoint we should take the underlying
set of $M_k$ to be $\{\top, \zero, \mbf, \one, \mbt, \bot\} \times
\{k\}$, for all $k\in \{1, \dots, n\}$. To keep the notation simple,
we will refrain from doing this but, as in the dot points below,
will always make it clear which sort is intended.

We will require the following multi-sorted relations; each is a compatible relation from $\M_j$ to $\M_k$, for some $j, k\in \{1, \dots,n\}$ with $j\le k$.
\begin{itemize}
\item ${\lez} = \big(M_0 \times \{\top\}\big) \cup \big(\{\bot\} \times M_0\big) \cup \{(\mbf, \mbf), (\mbt, \mbt) \}$ interpreted as a binary relation on $\M_0$,

\item ${\lek} = \{(\top, \top), (\bot, \bot) \} \cup \big(\{(\mbf, \mbf), (\mbf, \zero), (\zero, \zero)\}\big) \cup
\big(\{ (\mbt, \mbt), (\mbt,\one), (\one, \one)\}\big)$ interpreted as a binary relation on $\M_k$, for $k\in \{1, \dots, n\}$,

\item ${\lejk} = \{(\top, \top), (\bot, \bot)\} \cup \big(\{(\mbf, \mbf), (\mbf, \zero), (\zero, \zero)\}\big) \cup
\big(\{ (\mbt, \mbt), (\mbt,\one), (\one, \one)\}\big)$ interpreted as a relation from $\M_j$ to $\M_k$, for $j, k\in \{1, \dots, n\}$ with $j < k$.

\end{itemize}
Note that $\lez$ is the relation $S_{0,0}$ on $\JB_0$ interpreted as a relation on~$\M_0$,
that $\lek$ is the relation $S_{1,0}$ on $\JB_1$ interpreted as a  relation on $M_k$,
and that $\lejk$ is the relation $S_{1,0}$ on $\JB_1$ interpreted as a relation from $\M_j$ to $\M_k$.
For all $k\in \{0, \dots, n\}$, the relation $\lek$ is an order; in particular, $\lez$ is the knowledge order on $\JB_0$ interpreted as an order on $M_0$. The relation $\lejk$ can be thought of as the order relation $S_{1,0}$ on $\JB_1$ `stretched' from $M_j$ to $M_k$. See Figure~\ref{fig:orders}.

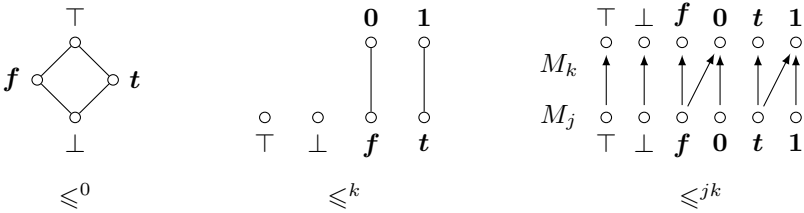
\begin{figure}[ht]
\centering
\begin{tikzpicture}[scale=0.5]
\begin{scope}
    \node[anchor=north] at (0,-1.5) {$\lez$};
  \node[unshaded] (bot) at (0,0) {};
  \node[unshaded] (f) at (-1,1) {};
  \node[unshaded] (t) at (1,1) {};
  \node[unshaded] (top) at (0,2) {};
  \draw[order] (bot) -- (f) -- (top);
  \draw[order] (bot) -- (t) -- (top);
  \node[label,anchor=north] at (bot) {$\bot$};
  \node[label,anchor=east] at (f) {$\mbf$};
  \node[label,anchor=west] at (t) {$\mbt$};
  \node[label,anchor=south] at (top) {$\top$};
\end{scope}
\begin{scope}[xshift=5cm] 
    \node[anchor=north] at (2.1,-1.5) {$\lek$};
  \node[unshaded] (top) at (0,0) {};
  \node[unshaded] (bot) at (1.4,0) {};
  \node[unshaded] (f) at (2.8,0) {};
  \node[unshaded] (0) at (2.8,2) {};
  \node[unshaded] (t) at (4.2,0) {};
  \node[unshaded] (1) at (4.2,2) {};
  \draw[order] (f) -- (0);
  \draw[order] (t) -- (1);
  \node[label,anchor=north] at (bot) {$\bot$};
  \node[label,anchor=north] at (f) {$\mbf$};
  \node[label,anchor=north] at (t) {$\mbt$};
  \node[label,anchor=south] at (0) {$\zero$};
  \node[label,anchor=south] at (1) {$\one$};
  \node[label,anchor=north] at (top) {$\top$};
\end{scope}
\begin{scope}[xshift=14cm] 
    \node[anchor=north] at (2.5,-1.5) {$\lejk$};
    \node at (-1.25,0) {$M_j$};
  \node at (-1.25,1.4) {$M_k$};
  \node[unshaded] (topj) at (0,0) {};
  \node[unshaded] (botj) at (1,0) {};
  \node[unshaded] (fj) at (2,0) {};
  \node[unshaded] (0j) at (3,0) {};
  \node[unshaded] (tj) at (4,0) {};
  \node[unshaded] (1j) at (5,0) {};
  \node[unshaded] (topk) at (0,2) {};
  \node[unshaded] (botk) at (1,2) {};
  \node[unshaded] (fk) at (2,2) {};
  \node[unshaded] (0k) at (3,2) {};
  \node[unshaded] (tk) at (4,2) {};
  \node[unshaded] (1k) at (5,2) {};
  \draw[arrow] (topj) -- (topk);
  \draw[arrow] (botj) -- (botk);
  \draw[arrow] (fj) -- (fk);
  \draw[arrow] (fj) -- (0k);
  \draw[arrow] (0j) -- (0k);
  \draw[arrow] (tj) -- (tk);
  \draw[arrow] (tj) -- (1k);
  \draw[arrow] (1j) -- (1k);
  \node[label,anchor=north] at (botj) {$\bot$};
  \node[label,anchor=north] at (fj) {$\mbf$};
  \node[label,anchor=north] at (tj) {$\mbt$};
  \node[label,anchor=north] at (0j) {$\zero$};
  \node[label,anchor=north] at (1j) {$\one$};
  \node[label,anchor=north] at (topj) {$\top$};
  \node[label,anchor=south] at (botk) {$\bot$};
  \node[label,anchor=south] at (fk) {$\mbf$};
  \node[label,anchor=south] at (tk) {$\mbt$};
  \node[label,anchor=south] at (0k) {$\zero$};
  \node[label,anchor=south] at (1k) {$\one$};
  \node[label,anchor=south] at (topk) {$\top$};
\end{scope}
\end{tikzpicture}
\caption{The relations $\lez$, $\lek$ and $\lejk$.}\label{fig:orders}
\end{figure}

For all $k\in \{1, \dots, n\}$, let $g_k \colon \M_k \to \M_0$ be the homomorphism that maps $\mbf$ and $\zero$ to $\mbf$ and maps $\mbt$ and $\one$ to $\mbt$.

\begin{thm}\label{cor:bigmultiduality}
Let $n\in \omega\comp\{0\}$. Define the multi-sorted alter ego
\[
\MT_n = \langle M_0\du M_1\du \cdots \du M_n;
\mathcal G_{(n)}, \mathcal{S}_{(n)}, \T\rangle,
\]
where
\begin{align*}
\mathcal{G}_{(n)} &= \big\{\,g_k \mid k\in \{1, \dots, n\}\,\big\}, \text{ and}\\
\mathcal S_{(n)} &= \{\,{\lek} \mid k\in \{0, \dots, n\}\,\} \cup \big\{\,{\lejk} \mid j, k\in \{1, \dots,n\} \text{ with } j < k\,\big\}.
\end{align*}
\begin{enumerate}[\normalfont(1)]

\item
The alter ego $\MT_n$ yields a strong, and therefore full,
 duality on $\CV_n =
 \HSP(\JB_n) = \ISP(\{\M_0, \M_1, \dots, \M_n\})$.

\item $|\mathcal S_{(n)}\cup \mathcal{G}_{(n)}| = \frac12(n^2 + 3n + 2)$.

\end{enumerate}
\end{thm}

\begin{eg}\label{eg:multi:n=1}
The multi-sorted alter ego
\[
\MT_1 =\langle M_0\du M_1;
g_1, \lez,  \le^1, \T\rangle
\]
yields a strong duality on the variety $\CV_1 = \HSP(\JB_1)$, and the multi-sorted alter ego
\[
\MT_2 =\langle M_0\du M_1\du M_2;
g_1,g_2,\lez, \le^1, \le^2, \le^{12},   \T\rangle
\]
yields a strong duality on the variety $\CV_2 = \HSP(\JB_2)$. 
\end{eg}

\begin{rem}[Comparing the dualities]\label{rem:compare}
Perhaps because we are more accustomed to working with orders rather than quasi-orders, the multi-sorted duality for the variety appears to be simpler than the single-sorted duality for the quasivariety. For example, the duality for the quasivariety $\CJ_1$ has an alter ego consisting of the order $S_{1,0}$ and the quasi-order $S_{1,1}$ on the six-element base set~$J_1$. The multi-sorted duality for the variety $\CV_1$ has an alter ego with two sorts: the four-element set $M_0$, equipped with the knowledge order $\lez$, and the six-element set $M_1$, equipped with the order $\le^1$ (which is the order $S_{1,0}$ on $J_1$ interpreted on $M_1$), along with a connecting map $g_1\colon M_1 \to M_0$.
Since the free algebras in the variety $\CV_n$ lie in the quasivariety $\CJ_n$,  we can use either the single sorted-duality for $\CJ_n$ or the multi-sorted duality for $\CV_n$ to find the free algebras in~$\CV_n$.
The authors used both dualities to verify that the size of the free algebra $\F_{\CV_1}\!(1)$ is 266. That is, we found all maps from $J_1$ to $J_1$ that preserve $S_{1,0}$ and $S_{1,1}$, thus representing $\F_{\CV_1}\!(1)$ as a subalgebra of $\JB_1^{J_1}$, and we found all multi-sorted maps from $M_0\cup M_1$ to $M_0\cup M_1$ that preserve $\lez$, $\le^1$ and $g_1$, thus representing $\F_{\CV_1}\!(1)$ as a subalgebra of $\M_0^{M_0}\times \M_1^{M_1}$. We found the latter calculation much easier as we were first able to find the 36
maps from $M_0$ to $M_0$ that preserve $\lez$ and then to link each of these via $g_1$ to a number of $\le^1$-preserving maps from $M_1$ to $M_1$.
\end{rem}

\section{Subuniverses of products of homomorphic images of $\JB_n$}\label{sec:CompOnHoms}

The Special NU Duality Theorem~\ref{NUDT} and the Special
Multi-sorted NU Strong Duality Theorem~\ref{MultiNUDT} tell us that
the set of all subuniverses of $\JB_n^2$ yields a duality on the
quasivariety $\CJ_n$ and that the set of all subuniverses of
$\M_j\times \M_k$, for $j,k\in \{0, \dots, n\}$, yields a duality on
the variety~$\CV_n$. It is always possible to restrict to
subuniverses that are meet-irreducible in $\Sub(\JB_n^2)$ and in
${\Sub(\M_j\times \M_k)}$---see Definition~\ref{def:entails} and the
discussion that follows it. Since $\M_k$ is a non-trivial
homomorphic image of $\JB_n$, we will treat both cases
simultaneously and describe the meet-irreducible members of the
lattice $\Sub(\A\times \B)$, where $\A$ and $\B$ are non-trivial
homomorphic images of~$\JB_n$.

We shall use the following observations, usually without comment.
\begin{enumerate}[(a)]

\item Let $u\colon \JB_n\to \A$ and $v\colon \JB_n\to \B$ be homomorphisms. Since each element of $\JB_n$
is a constant, every subuniverse of $\A \times \B$ contains the set
$K = \{\, (u(c), v(c))\mid c\in J_n\,\}$ of constants of $\A \times \B$.

\item Let $u\colon \JB_n\to \A$ be a homomorphism with $\A$ non-trivial. For all
$c \in J_n$, if $c\not\in \{\top, \bot\}$, then $u(c) \not\in
\{\top, \bot\}$. (For example, if $u(c)=\top$, then in $\A$ we have
$\bot = u(\bot) = u(c \otimes \neg c) = u(c) \otimes \neg(u(c)) =
\top \otimes \top =\top$, a~contradiction.)

\item Let $u\colon \JB_n\to \A$ be a surjective homomorphism with $\A$ non-trivial.

\begin{enumerate}[(i)]

\item
The bilattice reduct of $\A$ is isomorphic to the bilattice reduct of $\JB_k$, for some $k~\in~\{0, \dots, n\}$,

\item $u$ is the unique homomorphism from $\JB_n$ to $\A$
 (since $u$ preserves the constants).
\end{enumerate}
\end{enumerate}

We begin with a lemma that gives simple sufficient conditions for a subuniverse of $\A\times \B$ to contain large rectangular blocks, that is, large subsets of the form $A'\times B'$, for some $A'\subseteq A$ and $B'\subseteq B$.

Given $\A \in \CV_n$, let $\bsF_\A$ and $\bsT_\A$ denote, respectively, 
the sets of `false' constants and `true' constants in~$\A$.
Note that $\bF = \bsF_{\JB_n}$ and $\bT = \bsT_{\JB_n}$.

\begin{lem}\label{lem:rectangle1}
Let $n \in \omega$, let $\A$ and $\B$ be non-trivial homomorphic images of $\JB_n$ and let $S$ be a subuniverse of $\A \times \B$.
\begin{enumerate}[ \normalfont(a)]
\item The following are equivalent:
\begin{enumerate}[\normalfont (i)]
\item $(a,\top) \in S$ for some $a \in A\comp \{\top\}$;
\item $(\bot,b) \in S$ for some $b \in B\comp\{\bot\}$;
\item $A\times \{\top\}\subseteq S$;
\item $\{\bot\}\times B \subseteq S$.
\end{enumerate}

\item  The following are equivalent:
\begin{enumerate}[\normalfont(i)]
\item $(\top, b) \in S$ for some $b \in B\comp\{\top\}$;
\item $(a, \bot) \in S$ for some $a \in A\comp\{\bot\}$;
\item $\{\top\} \times B \subseteq S$;
\item $A\times \{\bot\} \subseteq S$.
\end{enumerate}

\item If $S$ satisfies any of the eight conditions listed in \textup{(a)} and \textup{(b)}, then $\bsF_\A \times \bsF_\B \subseteq S$ and $\bsT_\A \times \bsT_\B \subseteq S$.
\end{enumerate}
\end{lem}

\begin{proof}
By symmetry, to prove both (a) and (b),
it suffices to prove (a). Let $u\colon \JB_n\to \A$, $v\colon \JB_n\to \B$ be surjective homomorphisms. Since the
bilattice reduct of a non-trivial homomorphic image of $\JB_n$ is
isomorphic to the  bilattice reduct of~$\JB_k$, for some $k\in \{0,
\dots, n\}$, both $\A$ and $\B$ satisfy
\[
x \ne \top \implies x \otimes \neg x = \bot \qquad\text{and}\qquad
x \ne \bot \implies x \oplus \neg x = \top.\tag{$\dagger$}
\]
The implications (a)(iii) $\Rightarrow$ (a)(i) and (a)(iv) $\Rightarrow$ (a)(ii) are of course trivial. We now use $(\dagger)$ to prove
the implications
(a)(i)~$\Rightarrow$~(a)(iii) and (a)(ii)~$\Rightarrow$~(a)(iv).

(a)(i)~$\Rightarrow$~(a)(iii):
Assume that (a)(i) holds, i.e., there exists $a \in A\comp \{\top\}$ with $(a,\top) \in S$. Let
$a'\in A$ and choose $c\in J_n$ with $u(c) = a'$. By $(\dagger)$ we
have $a\otimes \neg a = \bot$ and thus
\[
(a', \top) = (\bot, \top) \oplus (a', v(c)) = \big((a,\top) \otimes \neg(a,\top)\big) \oplus (u(c), v(c))\in S,
\]
as $(a, \top), (u(c), v(c)) \in S$. So (a)(iii) holds. The implication (a)(ii)~$\Rightarrow$~(a)(iv) is
similar.

(a)(i)~$\Rightarrow$~(a)(ii): Assume again that (a)(i) holds
and let $a \in A\comp \{\top\}$ with $(a,\top) \in S$.
If $a = \bot$, then we can conclude (a)(ii) immediately as $B$ is non-trivial. Assume now that $a \neq \bot$ and let $c\in J_n$ with $u(c) = a$. Note that $c\ne \bot$. As $a\ne \top$, by $(\dagger)$ we have
$a\otimes \neg a = \bot$ and thus
\[
(\bot,\neg v(c)) = (a, \top) \otimes \neg(a, v(c)) = (a,\top) \otimes \neg(u(c), v(c)) \in S,
\]
as $(a, \top), (u(c),v(c))\in S$. Note that $c\ne \bot$ implies $v(c) \ne \bot$, since $\B$ is non-trivial, and consequently $\neg v(c)\ne \bot$. Hence (a)(ii) holds. The implication (a)(ii)~$\Rightarrow$~(a)(i) holds by symmetry and duality.

(c) By symmetry, it is enough to assume that
the equivalent conditions in (a) hold. Let $(a, b) \in \bsF_\A \times \bsF_\B$. Then $(a, b) = (a,\top) \wedge (\bot, b)\in S$ as $(a,\top),  (\bot,b) \in S$, whence $\bsF_\A \times \bsF_\B \subseteq S$. Similarly, $\bsT_\A \times \bsT_\B \subseteq S$.
\end{proof}

The following lemma provides a test for whether a subuniverse of $\A \times \B$ is proper.

\begin{lem}\label{lem:AtimesB}
Let $n \in \omega$, let $\A$ and $\B$ be non-trivial homomorphic images of $\JB_n$ and let $S$ be a subuniverse of $\A\times \B$. The following are equivalent:
\begin{enumerate}[ \normalfont(i)]

\item $S = A\times B$;

\item $(\top, \bot), (\bot, \top)  \in S$;

\item $(\bsF_\A\times \bsT_\B)\cap S \ne \varnothing$;

\item $(\bsT_\A\times \bsF_\B)\cap S \ne \varnothing$.

\end{enumerate}
\end{lem}

\begin{proof} (i) implies (iv) is trivial and (iv) is equivalent to (iii) by applying~$\neg$. Now assume (iii), and let
$(a, b)\in (\bsF_\A\times \bsT_\B)\cap S$.
Let $u\colon \JB_n\to \A$ and $v\colon \JB_n\to \B$ be surjective homomorphisms and let $c\in J_n$ with $u(c) = a$.
We have $c\in \bF$ since $a\in \bsF_\A$, thus $v(c)\in \bsF_\B$, whence $v(c)\otimes b = \bot$. Hence
\[
(a, \bot) = (a, v(c)) \otimes (a, b) = (u(c), v(c)) \otimes (a, b) \in S,
\]
since $(u(c), v(c))\in S$, and so
\[
(\top, \bot) = (a, \bot) \oplus \neg (a, \bot) \in S.
\]
Similarly, $(\bot, \top) \in S$. Hence (ii) holds.
Finally, assume (ii). Let $(a, b)\in A \times B$ and assume that $c, d \in J_n$ with $u(c) = a$ and $v(d) = b$. Then
\[
(a,b) = (u(c), v(d)) =\big((u(c), v(c)) \otimes (\top, \bot)\big) \oplus \big((u(d), v(d))\otimes (\bot, \top)\big) \in S,
\]
since $(u(c), v(c)), (u(d), v(d))\in S$. Hence (i) holds.
\end{proof}

\subsection{The relations $S_\le$ and $S_\ge$}\label{subsec:S_le}

Let $n\in \omega$ and fix two surjective homomorphisms $u\colon \JB_n\to \A$ and $v\colon \JB_n\to \B$ with $\A$ and $\B$ non-trivial. Define subsets $S_\le$ and $S_\ge$ of $A\times B$ by
\begin{align*}
S_\le &= \big(A\times \{\top\}\big)\cup \big(\{\bot\}\times B\big) \cup \big(\bsF_\A \times \bsF_\B\big) \cup \big(\bsT_\A \times \bsT_\B\big),\\ S_\ge &= \big(A\times \{\bot\}\big)\cup \big(\{\top\}\times B\big) \cup \big(\bsF_\A \times \bsF_\B\big) \cup \big(\bsT_\A \times \bsT_\B\big). \end{align*}
Note that if $\A$ and $\B$ are both $\JB_n$, then $S_\le$ and $S_\ge$ are the relations $S_{n,n}$ and $S\convsub{n,n}$, respectively.

Since, up to isomorphism, $\A$ has the same  bilattice reduct as $\JB_k$, for some $k \in \{0, \dots, n\}$,
there is a unique homomorphism $u_0 \colon \A\to \JB_{0, n}$. Similarly there is a unique homomorphism $v_0 \colon \B\to \JB_{0, n}$. We shall use these homomorphisms to show that $S_\le$ and $S_\ge$ are subuniverses of $\A \times \B$. Recall that we denote the knowledge order on $\JB_n$ by $\le_\k^n$. We shall also denote the knowledge order on $\JB_{0, n}$ by~$\le_\k^0$.

\begin{lem}\label{lem:Sle=Sgle}
Let $n\in \omega$ and let $u\colon \JB_n\to \A$ and $v\colon \JB_n\to \B$ be surjective homomorphisms with $\A$ and $\B$ non-trivial. Then $S_\le$ and $S_\ge$ are subuniverses of $\A \times \B$. Indeed,
\begin{align*}
\sg_{\A\times \B}\big((u, v)(\le_\k^n)\big) &= S_\le = (u_0, v_0)^{-1}(\le_\k^0),\\
\sg_{\A\times \B}\big((u, v)(\ge_\k^n)\big) &= S_\ge = (u_0, v_0)^{-1}(\ge_\k^0).
\end{align*}
\end{lem}

\begin{proof}
It suffices to prove the result for $S_\le$. Since $u$ satisfies $u(\bF) \subseteq \bsF_\A$ and $u(\bT) \subseteq \bsT_\A$, and similarly for $v$, we have $(u, v)(\le_\k^n)\subseteq S_\le$. It follows at once from Lemma~\ref{lem:rectangle1}
that $S_\le \subseteq \sg_{\A\times \B}\big((u, v)(\le_\k^n)\big)$. Hence
\[
(u, v)(\le_\k^n)\subseteq S_\le \subseteq \sg_{\A\times \B}\big((u, v)(\le_\k^n)\big).
\]
Since $\le_\k^0$ is a subuniverse of $\JB_{0,n}^2$, it follows that
$(u_0, v_0)^{-1}(\le_\k^0)$ is a subuniverse of $\A\times \B$. As
the knowledge order on the bilattice
$\JB_{0,n}$ is given by ${\le}_\k^0 =
\big(J_{0,n}\times \{\top\}\big) \cup  \big(\{\bot\} \times
J_{0,n}\big)$, it is clear that
\[
(u_0, v_0)^{-1}(\le_\k^0) = \big(A \times \{\top\}\big) \cup \big(\{\bot\} \times B\big) \cup \big(\bsF_\A \times \bsF_\B\big) \cup \big(\bsT_\A \times \bsT_\B\big) = S_\le.
\]
Hence $S_\le$ is a subuniverse of $\A \times \B$, and
$\sg_{\A\times \B}\big((u, v)(\le_\k^n)\big) = S_\le$ follows immediately.
\end{proof}

\begin{lem}\label{lem:maximal}
Let $n\in \omega$ and let $u\colon \JB_n\to \A$ and $v\colon \JB_n\to \B$ be surjective homomorphisms with $\A$ and $\B$ non-trivial. Then $S_\le$ and $S_\ge$ are the only maximal proper subuniverses of $\A \times \B$.
\end{lem}

\begin{proof}
Let $S$ be a proper subuniverse of $\A\times \B$. We have $(\bsF_\A\times \bsT_\B)\cap S = \varnothing$ and $(\bsT_\A\times \bsF_\B)\cap S = \varnothing$, by Lemma~\ref{lem:AtimesB}. Hence
\[
S\subseteq \big(A \times \{\top, \bot\}\big) \cup  \big(\{\top, \bot\} \times B\big) \cup \big(\bsF_\A \times \bsF_\B\big) \cup \big(\bsT_\A\times \bsT_\B\big). \tag{$*$}
\]
Suppose, by way of contradiction, that $S\nsubseteq S_\le$ and $S\nsubseteq S_\ge$. Thus,
there exist 
$(a,b)\in S\comp S_\le$ and $(c,d)\in S\comp S_\ge$.
By $(*)$ we have
\[
(a, b) \in \big(A\times \{\bot\}\big)\comp \{(\bot, \bot)\} \text{ or } (a, b)\in \big(\{\top\}\times B\big)\comp\{(\top, \top)\}.
\]
By Lemma~\ref{lem:rectangle1}, both cases yield $(\top, \bot)\in S$. Similarly, $(c,d)\in S\comp S_\ge$ yields $(\bot, \top)\in S$. By Lemma~\ref{lem:AtimesB}, this gives $S = A\times B$, a contradiction.
\end{proof}

\subsection{The relations $S_{ab}$}\label{subsec:Sab}

By Lemma~\ref{lem:maximal}, both
$S_\le$ and $S_\ge$ are meet-irreducible in $\Sub(\A\times \B)$. Our next step is to describe the non-maximal meet-irreducibles. To do this we first require a simple lemma. Recall that, given homomorphisms $u\colon \JB_n\to \A$ and $v\colon \JB_n\to \B$, we define
\[
K = \{\, (u(c), v(c))\mid c\in J_n\,\}\subseteq A\times B.
\]

Given $\C\in \CV_n$, let
$\bFC =\langle F_\C; \le_\k\rangle$ and $\bTC =\langle T_\C; \le_\k\rangle$ be the chains consisting of the `false' constants and the `true' constants of $\C$, respectively, in their knowledge order. We shall abbreviate $\mathbf F_{\!\JB_n}$ and $\mathbf T_{\!\JB_n}$ to $\bFn$ and $\bTn$, respectively.

Note that the following lemma says nothing when at least one of $\A$ and $\B$ is isomorphic to $\JB_{0,n}$. For example, if $\A \cong \JB_{0,n}$, then $|\bsF_\A| = 1$ and so $\bsF_\A \times \bsF_\B \subseteq K$;
whence (i), (ii) and (iii) of Part (1) of the lemma are false. 

\begin{lem}\label{lem:abNotInD}
Let $n\in \omega$ and let $u\colon \JB_n\to \A$ and $v\colon \JB_n\to \B$ be surjective homomorphisms with $\A$ and $\B$ non-trivial. Let $(a, b)\in \bsF_\A \times \bsF_\B$.
\begin{enumerate}[\normalfont(1)]

\item The following are equivalent:

\begin{enumerate}[\normalfont(i)]

\item $(a, b)\in \big(\bsF_\A \times \bsF_\B\big) \comp K$;

\item $u^{-1}(a) \cap v^{-1}(b) = \varnothing$;

\item {\normalfont(du)}\, $\max_\k(u^{-1}(a)) <_\k \min_\k(v^{-1}(b))$ \ or

\item[] {\normalfont(ud)}\, $\min_\k(u^{-1}(a)) >_\k \max_\k(v^{-1}(b))$.

\end{enumerate}

\item Condition {\normalfont(du)} holds if and only if $(\down_{\bFA} a \times \up_{\bFB}b)\cap K = \varnothing$.

\item Condition {\normalfont(ud)} holds  if and only if $(\up_{\bFA} a \times \down_{\bFB}b)\cap K = \varnothing$.

\end{enumerate}
\end{lem}

Conditions (2) and (3) explain the notation: (du) and (ud) are abbreviations for \emph{down-up} and \emph{up-down}, respectively.
\begin{proof}
We have
$\begin{aligned}[t]
&(a, b) \in \big(\bsF_\A \times \bsF_\B\big)\comp K\\
 \iff{}& (\forall c \in \bF) \ (u(c), v(c)) \ne (a, b)\\
  \iff{}& u^{-1}(a) \cap v^{-1}(b) = \varnothing.
\end{aligned}$

As $u^{-1}(a)$ and $v^{-1}(b)$ are intervals in $\bFn$, we have
$u^{-1}(a) \cap v^{-1}(b) = \varnothing$ if and only if
\[
{\textstyle\max_\k}(u^{-1}(a)) <_\k {\textstyle\min_\k}(v^{-1}(b)) \ \text{ or } \ {\textstyle\min_\k}(u^{-1}(a)) >_\k {\textstyle\max_\k}(v^{-1}(b)).
\]
This proves (1). Since $u^{-1}(a')$ and $v^{-1}(b')$ are intervals in  $\bFn$ for all $a'\in \bsF_\A$ and all $b' \in \bsF_\B$, we have
\begin{align*}
&{\textstyle\max_\k}(u^{-1}(a)) <_\k {\textstyle\min_\k}(v^{-1}(b))\\
\iff{}& (\forall a' \in\down_{\bFA} a)(\forall b' \in \up_{\bFB}b) \ {\textstyle\max_\k}(u^{-1}(a')) <_\k {\textstyle\min_\k}(v^{-1}(b'))\\
\iff {}&(\down_{\bFA} a \times \up_{\bFB}b)\cap K = \varnothing.
\end{align*}
Hence (2) holds, and therefore (3) holds by symmetry.
\end{proof}

Given $(a, b) \in \big(\bsF_\A \times \bsF_\B\big)\comp K $,
precisely one of the conditions (du) and (ud) in
Lemma~\ref{lem:abNotInD}(1)(iii) holds. If $(a,b)\models
\text{(du)}$, then we define
\begin{multline*}
S_{ab} := \{(\top, \top), (\bot, \bot) \} \cup \big((\bsF_\A \times \bsF_\B)\comp (\down_{\bFA} a \times \up_{\bFB}b)\big)\\
\cup \big((\bsT_\A \times \bsT_\B)\comp (\down_{\bTA} \neg a \times \up_{\bTB}\neg b)\big),
    \end{multline*}
and if $(a,b)\models \text{(ud)}$, then we define
{\allowdisplaybreaks\begin{multline*}
S_{ab} := \{(\top, \top), (\bot, \bot)\} \cup \big((\bsF_\A \times \bsF_\B)\comp (\up_{\bFA} a \times \down_{\bFB}b)\big)\\
\cup \big((\bsT_\A \times \bsT_\B)\comp (\up_{\bTA} \neg a \times \down_{\bTB}\neg b)\big).
\end{multline*}}
Assume $(a,b)\models \text{(du)}$. As $\bFA$ and $\bFB$ are chains,
\[
F_{ab} := (\bsF_\A \times \bsF_\B)\comp (\down_{\bFA} a \times
\up_{\bFB}b)\ \text{ and } \  T_{ab} := (\bsT_\A \times
\bsT_\B)\comp (\down_{\bTA} \neg a \times \up_{\bTB}\neg b)
\]
form sublattices of $\bFA \times \bFB$ and $\bTA \times \bTB$, respectively.
The  knowledge and truth orders on the subset $S_{ab}
= \{(\top, \top), (\bot, \bot)\} \cup F_{ab} \cup T_{ab}$ of $A\times B$ are shown in Figure~\ref{fig:Sab}.
Note that in Figure~\ref{fig:Sab}, and in later figures, we abbreviate $(a,b)$ to $ab$ for read\-ability.
With this diagram in hand, the following lemma is almost immediate.

\begin{lem}\label{prop:Sab-subalg}
Let $n\in \omega$ and let $u\colon \JB_n\to \A$ and $v\colon \JB_n\to \B$ be surjective homomorphisms with $\A$ and $\B$ non-trivial. Then $S_{ab}$ is a subuniverse of $\A \times \B$, for all $(a, b) \in \big(\bsF_\A \times \bsF_\B\big)\comp K $.
\end{lem}

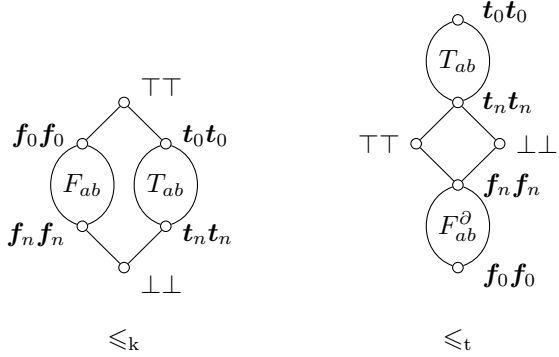
\begin{figure}[!ht]
\centering
\begin{tikzpicture}[scale=0.55]
\begin{scope}
  \node at (0,-1.75) {$\le_\k$};
  \node[unshaded] (bot) at (0,0) {};
  \node[unshaded] (fn) at (-1,1) {};
  \node[unshaded] (tn) at (1,1) {};
  \node[unshaded] (f0) at (-1,3) {};
  \node[unshaded] (t0) at (1,3) {};
  \node[unshaded] (top) at (0,4) {};
  \draw[order] (bot) -- (fn);
  \draw[curvy] (fn) to [bend left] (f0);
  \draw[curvy] (fn) to [bend right] (f0);
  \draw[order] (f0) -- (top);
  \draw[order] (bot) -- (tn);
  \draw[curvy] (tn) to [bend left] (t0);
  \draw[curvy] (tn) to [bend right] (t0);
  \draw[order] (t0) -- (top);
  \node[label,anchor=west,yshift=-6pt] at (bot) {$\bot\bot$};
  \node[label,anchor=east,yshift=-3pt] at (fn) {$\mbf_n\mbf_n$};
  \node[label] at ($0.5*(fn)+0.5*(f0)$) {$F_{ab}$};
  \node[label,anchor=east,yshift=3pt] at (f0) {$\mbf_0\mbf_0$};
  \node[label,anchor=west,yshift=-3pt] at (tn) {$\mbt_n\mbt_n$};
  \node[label] at ($0.5*(tn)+0.5*(t0)$) {$T_{ab}$};
  \node[label,anchor=west,yshift=3pt] at (t0) {$\mbt_0\mbt_0$};
  \node[label,anchor=west,yshift=6pt] at (top) {$\top\top$};
\end{scope}
\begin{scope}[xshift=8cm]
  \node at (0,-1.75) {$\le_\t$};
  \node[unshaded] (f0) at (0,0) {};
  \node[unshaded] (fn) at (0,2) {};
  \node[unshaded] (top) at (-1,3) {};
  \node[unshaded] (bot) at (1,3) {};
  \node[unshaded] (tn) at (0,4) {};
  \node[unshaded] (t0) at (0,6) {};
  \draw[curvy] (f0) to [bend left] (fn);
  \draw[curvy] (f0) to [bend right] (fn);
  \draw[order] (fn) -- (top) -- (tn);
  \draw[order] (fn) -- (bot) -- (tn);
  \draw[curvy] (tn) to [bend left] (t0);
  \draw[curvy] (tn) to [bend right] (t0);
  \node[label,anchor=west,xshift=3pt,yshift=-3pt] at (f0) {$\mbf_0\mbf_0$};
  \node[label] at ($0.5*(fn)+0.5*(f0)$) {$F_{ab}^\partial$};
  \node[label,anchor=west,xshift=3pt] at (fn) {$\mbf_n\mbf_n$};
  \node[label,anchor=east] at (top) {$\top\top$};
  \node[label,anchor=west] at (bot) {$\bot\bot$};
  \node[label,anchor=west,xshift=3pt] at (tn) {$\mbt_n\mbt_n$};
  \node[label] at ($0.5*(tn)+0.5*(t0)$) {$T_{ab}$};
  \node[label,anchor=west,xshift=3pt,yshift=3pt] at (t0) {$\mbt_0\mbt_0$};
\end{scope}
\end{tikzpicture}
\caption{The subuniverse $S_{ab}$ of $\A \times \B$.}\label{fig:Sab}
\end{figure}

Let $\mathcal F$ be a topped intersection structure on a non-empty set~$X$, that is, $\mathcal F$ contains $X$ and is closed under intersections of non-empty families, and let $x\in X$. An element $Y$ of $\mathcal F$ is a \defn{value at $x$} if $Y$ is maximal in $\mathcal F$ with respect to not containing $x$. The following lemma will help us to identify the meet-irreducible elements of the lattice $\Sub(\A\times \B)$. The proof is very easy---see \cite[Lemma 8.5.1]{CD98} for the proof in the case that $\mathcal F$ is the lattice of subuniverses of some algebra.

\begin{lem} \label{lem:CMI=Value}
Let $\mathcal F$ be a topped intersection structure on a non-empty set~$X$. An element $Y$ of $\mathcal F$ is completely meet-irreducible in the lattice $\mathcal F$ if and only if $Y$ is a value at $x$ for some $x\in X$.
\end{lem}

Given a topped intersection structure $\mathcal F$ on $X$ and $x\in X$, let $\Val(x)$ denote the set of values of $\mathcal F$ at~$x$.
Note that, by Lemma~\ref{lem:CMI=Value}, the union over all $(a, b)\in A\times B$ of the sets $\Val(a,b)$ is the set of all meet-irreducible
elements of the lattice $\Sub(\A\times \B)$.

Note that Case (d) in the following theorem arises only when neither $\A$ nor $\B$ is isomorphic to~$\JB_{0,n}$.

\begin{thm}\label{thm:MI}
Let $n\in \omega$ and let $u\colon \JB_n\to \A$ and $v\colon \JB_n\to \B$ be surjective homomorphisms with $\A$ and $\B$ non-trivial.
The meet-irreducible elements in the lattice $\Sub(\A\times \B)$ are the sets $S_\le$ and $S_\ge$, and $S_{ab}$, for all pairs
$(a, b) \in \big(\bsF_\A \times \bsF_\B\big)\comp K$.
Indeed,
\begin{enumerate}[ \normalfont(a)]

\item $\Val(a, b) = \{S_\le, S_\ge\}$, for all $(a, b) \in \big(A\times B\big)\comp \big(S_\le \cup S_\ge\big)$,

\item $\Val(a, b) = \{S_\le\}$, for all $(a, b) \in S_\ge\comp S_\le$,

\item $\Val(a, b) = \{S_\ge\}$, for all $(a, b) \in S_\le\comp S_\ge$,

\item $\Val(a,b) = \{S_{ab}\}$, for all $(a, b) \in \big(S_\le \cap S_\ge\big)\comp K$,

\item $\Val(a, b) = \varnothing$, for all $(a, b)\in K$.
\end{enumerate}\end{thm}

\begin{proof}
(a) By Lemma~\ref{lem:maximal}, $S_\le$ and $S_\ge$ are the
only maximal subuniverses of $\A\times \B$. It is therefore
trivial that, for all $(a, b) \in
\big(A\times B\big)\comp \big(S_\le \cup S_\ge\big)$, we have  $\Val(a, b) = \{S_\le, S_\ge\}$.

(b) Let $(a, b) \in S_\ge\comp S_\le =
\big(\{\top\}\times B\big) \cup \big(A\times \{\bot\}\big)\comp \{(\top, \top), (\bot, \bot)\}$ and assume that $S$ is a value at $(a, b)$. Since $S$ is a proper subuniverse, by Lemma~\ref{lem:maximal} we have either (i) $S\subseteq S_\le$ or (ii) $S\subseteq S_\ge$. Assume that (ii) holds. By Lemma~\ref{lem:rectangle1}, $S$~is disjoint from $\big(\{\top\}\times B\big) \cup \big(A\times \{\bot\}\big)\comp\{(\top, \top), (\bot, \bot)\}$ and so
\[
S\subseteq \{(\top, \top), (\bot, \bot)\} \cup \big(\bsF_\A \times \bsF_\B\big) \cup \big(\bsT_\A \times \bsT_\B\big) \subseteq S_\le.
\]
Hence, in both cases (i) and (ii) we have $S\subseteq S_\le$. As $(a, b)\notin S_\le$, the maximality of $S$ yields $S = S_\le$.

(c) If $S$ is a value at $(a, b) \in S_\le\comp S_\ge$, then similarly we derive $S = S_\ge$.

(d) Let $(a, b) \in \big(S_\le \cap S_\ge\big)\comp K = \big(\big(\bsF_\A \times \bsF_\B\big)\cup \big(\bsT_\A \times \bsT_\B\big)\big)\comp K
$. By symmetry, we may assume that ${(a, b) \in \big(\bsF_\A \times \bsF_\B\big)\comp K}$. By
Lemma~\ref{lem:abNotInD} we may assume without loss of
generality that $(a,b)\models \text{(du)}$, in which case
\[
S_{ab} = \{(\top, \top), (\bot, \bot)\} \cup F_{ab} \cup T_{ab}.
\]
Assume $S$ is a value at $(a, b)$. By Lemma~\ref{lem:rectangle1}, we know that $S$ is
disjoint from
\[
\big((A \times \{\bot, \top\})\cup (\{\bot, \top\}\times B)\big)\setminus \{(\top, \top), (\bot, \bot)\}.
\]
As $S$ is a proper subuniverse of $\A\times \B$, by Lemma~\ref{lem:AtimesB} it is also disjoint from $(\bsF_\A\times \bsT_\B)\cup (\bsT_\A\times \bsF_\B)$. Hence
\[
S\subseteq \{(\top, \top), (\bot, \bot)\} \cup \big((\bsF_\A \times \bsF_\B)\big)
\cup \big((\bsT_\A \times \bsT_\B)\big). 
\]
We prove that $(\down_{\bFA} a \times \up_{\bFB}b)\cap S = \varnothing$. Suppose
$(c,d)\in (\down_{\bFA} a \times \up_{\bFB}b) \cap S$, whence $c\le_\k a$ and $d\ge_\k b$.
Choose $a'\in u^{-1}(a)$, $b'\in v^{-1}(b)$. As $(a,b)\models \text{(du)}$, by Lemma~\ref{lem:abNotInD}(1) we have $a' <_\k b'$, whence $u(b') \ge_\k u(a') = a \ge_\k c$
and $v(a') \le_\k v(b') = b \le_\k d$. Thus,
\begin{align*}
(a, b) &= (c \oplus a, b \oplus v(a'))\\
&= \big((c \otimes u(b'))\oplus u(a'), (d\otimes v(b'))\oplus v(a')\big)\\
&= \big((c, d) \otimes (u(b'), v(b'))\big)\oplus(u(a'), v(a')).
\end{align*}
It follows that $(a, b) \in S$ since $(c, d)\in S$ by assumption, and since we have
$(u(b'), v(b')), (u(a'), v(a'))\in K$, and $K\subseteq S$. This contradiction shows that $(\down_{\bFA} a \times \up_{\bFB}b)\cap S = \varnothing$.
By applying $\neg$, we get $(\down_{\bTA} \neg a \times \up_{\bTB}\neg b)\cap S = \varnothing$. We conclude that
\begin{align*}
S\subseteq  &\{(\top, \top), (\bot, \bot)\} \cup F_{ab} \cup T_{ab} = S_{ab}.
\end{align*}
Since $(a, b)\notin S_{ab}$ and $S$ is maximal with respect to not containing $(a, b)$, we have $S = S_{ab}$.

(e) It is trivial that $\Val(a,b) =\varnothing$, for all $(a, b)\in K$, as $K$ is the set of constants of $\A\times \B$.
\end{proof}

\section{The proof that $\protect\JT_n$ yields a duality on $\protect\CJ_n$}\label{sec:singleproof}

Recall from the NU Duality Theorem~\ref{NUDT} that $\JT^2_n = \langle J_n; \Sub(\JB_n^2), \T\rangle$ yields a duality on the quasivariety
$\CJ_n =\ISP(\JB_n)$, where $\Sub(\JB_n^2)$
is the set of all compatible binary relations on $\JB_n$. Our aim is to remove relations from the set $\Sub(\JB_n^2)$ without destroying the duality until we arrive at the set $\mathcal R_{(n)}$ described in Theorem~\ref{cor:bigduality}.

The concept of \emph{entailment}~\cite[Section 2.4]{CD98} is crucial to understanding how and why it is possible to reduce the number of compatible relations required to yield a duality.

\begin{df}\label{def:entails}
Let $\M$ be a finite algebra, let $\mathcal{R}\cup \{S\}$ be a set of compatible finitary relations on $\M$ and let
$\CA =\ISP(\M)$. We say that $\mathcal{R}$ \defn{entails} $S$
and write $\mathcal{R} \vdash S$
if, for every $\A \in \CA$, every continuous map $u \colon  \mathrm{D}(\A) \to M$ that preserves the relations in $\mathcal{R}$  also preserves $S$.

There is an obvious extension of the concept of entailment to the multi-sorted setting that we will use just once in Section~\ref{sec:optmulti}.
\end{df}

The significance of entailment is that if $\MT = \langle M; \mathcal R, \T\rangle$ yields a duality on $\CA$ and $\mathcal R\setminus \{S\}\vdash S$, then
$\MT' = \langle M; \mathcal R\setminus \{S\}, \T\rangle$ also yields a duality on~$\CA$. There are several admissible constructs that yield entailment;
for example, every compatible binary relation,
$R$, on $\M$ entails its converse, $R\conv$, and every pair
$R, S$ of compatible binary relations on $\M$ entail their intersection:  
\[
R\vdash R\conv \qquad \text{and}\qquad  \{R, S\} \vdash R \cap S.
\]
Thus, we can certainly remove all meet-reducible members of the lattice $\Sub(\JB_n^2)$ without destroying the duality. Hence our first task is to describe the meet-irreducible members of $\Sub(\JB_n^2)$.

With the help of the Universal Algebra Calculator
(UAC)~\cite{UACalc}, we have drawn the lattice $\Sub(\JB_2^2)$---see
Figure~\ref{fig:subalg-J2sq}.
Computer calculations yield 200 compatible
binary relations on $\JB_3$ (107 up to converses). While the size of
$\Sub(\JB_n^2)$ grows very quickly with~$n$, we will see that the
set of meet-irreducible elements of $\Sub(\JB_n^2)$ is much more
manageable and has size~$O(n^2)$.

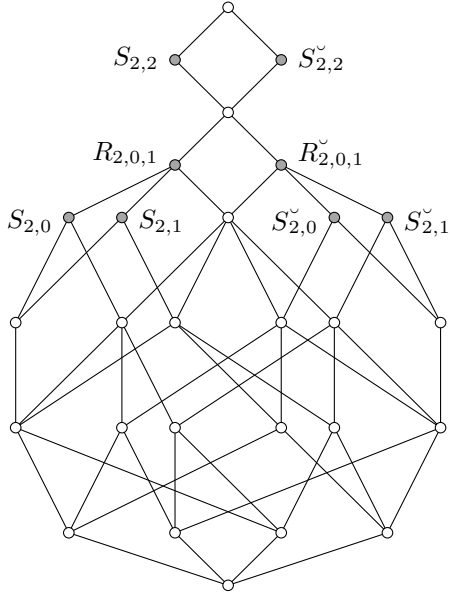
\begin{figure}[h!]
\centering
\begin{tikzpicture}[scale=0.7]
\draw[order] (0,4) -- (3,5) -- (4,7) -- (4,9) -- (3,11) -- (1,12) -- (0,13) -- (1,14) -- (0,15);
\draw[order] (-1,7) -- (2,9);
\draw[order] (-1,9) -- (1,7) -- (1,9);
\draw[order] (1,5) -- (-1,7) -- (-2,9);
\draw[order] (-2,9) -- (-2,7) -- (1,9);
\draw[order] (-1,9) -- (2,7) -- (2,9);
\draw[order] (-1,5) -- (-1,7);
\draw[order] (3,5) -- (1,7);
\draw[order] (-3,5) -- (1,7);
\draw[order] (-3,5) -- (-2,7) -- (-1,5) -- (4,7);
\draw[order] (3,5) -- (2,7) -- (1,5) -- (-4,7);
\draw[order] (-3,11) -- (-2,9) -- (0,11) -- (2,9) -- (3,11);
\draw[order] (-4,9) -- (-2,11) -- (-1,9) -- (0,11) -- (1,9) -- (2,11);
\draw[order] (0,4) -- (1,5);
\draw[order] (0,4) -- (-1,5);
\draw[order] (0,4) -- (-3,5) -- (-4,7) -- (-4,9) -- (-3,11) -- (-1,12) -- (0,13) -- (-1,14) -- (0,15);
\draw[order] (-2,11) -- (-1,12) -- (0,11) -- (1,12) -- (2,11) -- (4,9);
\draw[order] (-2,9) -- (-4,7) -- (-1,9);
\draw[order] (1,9) -- (4,7) -- (2,9);
\node[unshaded] at (0,15) {};
\node[shaded] at (-1,14) {};
\node[shaded] at (1,14) {};
\node[unshaded] at (0,13) {};
\node[shaded] at (-1,12) {};
\node[shaded] at (1,12) {};
\node[shaded] at (-3,11) {};
\node[shaded] at (-2,11) {};
\node[unshaded] at (0,11) {};
\node[shaded] at (2,11) {};
\node[shaded] at (3,11) {};
\node[unshaded] at (-4,9) {};
\node[unshaded] at (-2,9) {};
\node[unshaded] at (-1,9) {};
\node[unshaded] at (1,9) {};
\node[unshaded] at (2,9) {};
\node[unshaded] at (4,9) {};
\node[unshaded] at (-4,7) {};
\node[unshaded] at (-2,7) {};
\node[unshaded] at (-1,7) {};
\node[unshaded] at (1,7) {};
\node[unshaded] at (2,7) {};
\node[unshaded] at (4,7) {};
\node[unshaded] at (-3,5) {};
\node[unshaded] at (-1,5) {};
\node[unshaded] at (1,5) {};
\node[unshaded] at (3,5) {};
\node[unshaded] at (0,4) {};
\node[label,anchor=east] at (-1,14) {$S_{2,2}$};
\node[label,anchor=west] at (1,14) {$S\convsub{2,2}$};
\node[label,anchor=east,yshift=4pt] at (-1,12) {$R_{2,0,1}$};
\node[label,anchor=west,yshift=4pt] at (1,12) {$R\convsub{2,0,1}$};
\node[label,anchor=east,yshift=-2pt] at (-3,11) {$S_{2,0}$};
\node[label,anchor=west,yshift=-2pt] at (-2,11) {$S_{2,1}$};
\node[label,anchor=east,yshift=-2pt] at (2,11)  {$S\convsub{2,0}$};
\node[label,anchor=west,yshift=-2pt] at (3,11) {$S\convsub{2,1}$};
\end{tikzpicture}
\caption{The lattice $\Sub(\JB_2^2)$ with its meet-irreducible elements shaded and labelled.}\label{fig:subalg-J2sq}
\end{figure}

\begin{thm}\label{thm:MIJ_n^2}
Let $n\in \omega$. The meet-irreducibles of the lattice $\Sub(\JB_n^2)$ are
\[
S_{n,i}, \text{ for }0\le i\le n\quad \text{ and }\quad R_{n, i, j}, \text{ for }
0\le i< j\le n-1 \text{ when } n\ge 2,
\]
and their converses.
\end{thm}

\begin{proof}
We shall apply Theorem~\ref{thm:MI} in the case that $\A = \B = \JB_n$. In this case,
 $S_\le = S_{n,n}$ and $S_\ge = S\convsub{n,n}$. Since $(\mbf_i, \mbf_j) \in \big(\bsF_\A \times \bsF_\B\big)\comp K$ if and only if $i \ne j$, the remaining meet-irreducibles in $\Sub(\JB_n^2)$ are of
the form $S_{\mbf_i \mbf_j}$, for some $i, j \in \{0, \dots, n\}$ with~$i\ne j$. If $i < j$, then $(\mbf_i, \mbf_j)\models \text{(ud)}$
and hence the relation $S_{\mbf_i \mbf_j}$ can be expressed as
\begin{align*}
&\{(\top, \top), (\bot, \bot)\} \cup \big((\bsF_n^2)\comp (\up_{\bFn} \mbf_i \times \down_{\bFn}\mbf_j)\big) \cup
\big((\bsT_n^2)\comp (\up_{\bTn} \neg \mbf_i \times \down_{\bTn}\neg \mbf_j)\big)\\
={}&\{(\top, \top), (\bot, \bot)\} \cup \big((\bsF_n^2)\comp (\up_{\bFn} \mbf_i \times \down_{\bFn}\mbf_j)\big) \cup
\big((\bsT_n^2)\comp (\up_{\bTn}  \mbt_i \times \down_{\bTn} \mbt_j)\big),
\end{align*}
which is the relation $R_{n, i, j-1}$.
In particular, we have $S_{\mbf_i \mbf_{i+1}}= R_{n,i,i} = S_{n, i}$.
If $i > j$, then $(\mbf_i, \mbf_j)\models \text{(du)}$
and hence the relation $S_{\mbf_i \mbf_j}$ is
\begin{align*}
&\{(\top, \top), (\bot, \bot)\} \cup \big((\bsF_n^2)\comp (\down_{\bFn} \mbf_i \times \up_{\bFn}\mbf_j)\big) \cup
\big((\bsT_n^2)\comp (\down_{\bTn} \neg \mbf_i \times \up_{\bTn}\neg \mbf_j)\big)\\
={}&\{(\top, \top), (\bot, \bot)\} \cup \big((\bsF_n^2)\comp (\down_{\bFn} \mbf_i \times \up_{\bFn}\mbf_j)\big) \cup
\big((\bsT_n^2)\comp (\down_{\bTn}  \mbt_i \times \up_{\bTn} \mbt_j)\big),
\end{align*}
which is the relation $R\convsub{n, j, i-1}$. In particular, we have $S_{\mbf_{i+1} \mbf_i}= R\convsub{n,i,i} = S\convsub{n, i}$.
Note that relations of the form $R_{n, i, j}$ that are not of the form $S_{n,i}$ occur only when $n\ge 2$; hence the restriction $n
\ge 2$ in the statement of the theorem.
\end{proof}

For $i,j\in \{0, \dots, n\}$ with $i\le j$, define
\[
\bF^{i,j} = \bF^2\setminus (\{\mbf_0, \dots, \mbf_i\}\times \{\mbf_{j+1},\dots, \mbf_n\}\big).
\]
Note that $\bF^{i,j}$ forms a sublattice of $\bFn^2$ and is obtained from $\bFn^2$ by removing a product of an up-set and a down-set from the
right-hand corner---see Figure~\ref{fig:ker(u)} for a drawing of~$\mathbf F_{\!3}^{0,1}$. The sublattice $\bTn^{i,j}$ of $\bTn^2$ is defined analogously.
The relation $R_{n, i, j}$ as a subuniverse of $\JB_n^2$ is then as shown in Figure~\ref{fig:Rnij}.

\begin{figure}[h!t]
\centering
\begin{tikzpicture}[scale=0.55]
\begin{scope}
  \node at (0,-1.75) {$\le_\k$};
  \node[unshaded] (bot) at (0,0) {};
  \node[unshaded] (fn) at (-1.25,1) {};
  \node[unshaded] (tn) at (1.25,1) {};
  \node[unshaded] (f0) at (-1.25,3) {};
  \node[unshaded] (t0) at (1.25,3) {};
  \node[unshaded] (top) at (0,4) {};
  \draw[order] (bot) -- (fn);
  \draw[fatcurvy] (fn) to [bend left] (f0);
  \draw[fatcurvy] (fn) to [bend right] (f0);
  \draw[order] (f0) -- (top);
  \draw[order] (bot) -- (tn);
  \draw[fatcurvy] (tn) to [bend left] (t0);
  \draw[fatcurvy] (tn) to [bend right] (t0);
  \draw[order] (t0) -- (top);
  \node[label,anchor=west,yshift=-6pt] at (bot) {$\bot\bot$};
  \node[label,anchor=east,yshift=-3pt] at (fn) {$\mbf_n\mbf_n$};
  \node[label] at ($0.5*(fn)+0.5*(f0)$) {$F_n^{ij}$};
  \node[label,anchor=east,yshift=5pt] at (f0) {$\mbf_0\mbf_0$};
  \node[label,anchor=west,yshift=-5pt] at (tn) {$\mbt_n\mbt_n$};
  \node[label] at ($0.5*(tn)+0.5*(t0)$) {$T_n^{ij}$};
  \node[label,anchor=west,yshift=5pt] at (t0) {$\mbt_0\mbt_0$};
  \node[label,anchor=west,yshift=6pt] at (top) {$\top\top$};
\end{scope}
\begin{scope}[xshift=8cm]
  \node at (0,-1.75) {$\le_\t$};
  \node[unshaded] (f0) at (0,0) {};
  \node[unshaded] (fn) at (0,2) {};
  \node[unshaded] (top) at (-1,3) {};
  \node[unshaded] (bot) at (1,3) {};
  \node[unshaded] (tn) at (0,4) {};
  \node[unshaded] (t0) at (0,6) {};
  \draw[fatcurvy] (f0) to [bend left] (fn);
  \draw[fatcurvy] (f0) to [bend right] (fn);
  \draw[order] (fn) -- (top) -- (tn);
  \draw[order] (fn) -- (bot) -- (tn);
  \draw[fatcurvy] (tn) to [bend left] (t0);
  \draw[fatcurvy] (tn) to [bend right] (t0);
  \node[label,anchor=west,xshift=3pt,yshift=-3pt] at (f0) {$\mbf_0\mbf_0$};
  \node[label] at ($0.5*(fn)+0.5*(f0)$) {$(F_n^{ij})^\partial$};
  \node[label,anchor=west,xshift=3pt,yshift=3pt] at (fn) {$\mbf_n\mbf_n$};
  \node[label,anchor=east] at (top) {$\top\top$};
  \node[label,anchor=west] at (bot) {$\bot\bot$};
  \node[label,anchor=west,xshift=3pt,yshift=-3pt] at (tn) {$\mbt_n\mbt_n$};
  \node[label] at ($0.5*(tn)+0.5*(t0)$) {$T_n^{ij}$};
  \node[label,anchor=west,xshift=3pt,yshift=3pt] at (t0) {$\mbt_0\mbt_0$};
\end{scope}
\end{tikzpicture}
\caption{The subuniverse $R_{n,i,j}= S_{n,i}\cup S_{n,j}$ of $\JB_n^2$.}\label{fig:Rnij}
\end{figure}
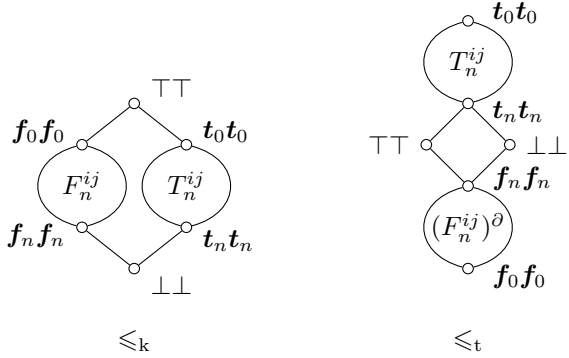

\begin{eg}
By Theorem~\ref{thm:MIJ_n^2},
up to converses, sets of meet-irreducible members of $\Sub(\JB_2^2)$ and $\Sub(\JB_3^2)$ are
\begin{align*}
\mathcal R_{(2)}^{\mathrm{mi}}&=\{S_{2,0}, S_{2,1}, S_{2,2},  R_{2,0,1}\} \text{ and }\\
\mathcal R_{(3)}^{\mathrm{mi}}&=\{S_{3,0}, S_{3,1}, S_{3,2}, S_{3,3}, R_{3,0,1}, R_{3,0,2}, R_{3,1,2}\},
\end{align*}
respectively. It follows at once that $\JT_2^{\mathrm{mi}} = \langle J_2; \mathcal{R}_{(2)}^{\mathrm{mi}}, \T\rangle$ yields a duality on $\CJ_2$ and that $\JT_3^{\mathrm{mi}} = \langle J_3; \mathcal{R}_{(3)}^{\mathrm{mi}}, \T\rangle$ yields a duality on $\CJ_3$. To obtain the $n = 2$ and $n = 3$ versions of the duality given in Theorem~\ref{cor:bigduality} we must prove that the relations $R_{2,0,1}$, $R_{3,0,1}$ and $R_{3,1,2}$ can be removed without destroying the dualities.
\end{eg}

Admissible constructs for entailment were investigated by Davey, Haviar and Priestley~\cite{DHP-SSE}.
They showed that there is a finite number of
admissible constructs that can be used to obtain $S$ from $\mathcal{R}$ whenever $\mathcal{R}\vdash S$.
An extensive list of constructs is given in~\cite[2.4.5]{CD98}. Here, in addition to intersection and converse, we need only one further construct.

Given compatible binary relations $R$ and $S$, their relational product
\[
R \cdot S := \{\, (a, b) \in M^2 \mid (\exists c\in M)\ (a, c) \in R \And (c, b)\in S\,\}
\]
is also a compatible binary relation and we denote the corresponding subalgebra of $\M^2$ by $\mathbf R \cdot \mathbf S$. In general, $\{R, S\}$ does not entail $R\cdot S$, but there is one important case where it does.

\begin{df}   Let $\M$ be a finite algebra and let $R$ and $S$ be compatible binary relations on~$\M$. We say that the relational product $R \cdot S$ is a \emph{homomorphic relational product} if there exists a homomorphism $u\colon \mathbf R \cdot \mathbf S \to \M$ such
that $(a, u(a, b)) \in R$ and $(u(a, b), b) \in S$, for all $(a, b) \in R\cdot S$. It is straightforward to check that if $R\cdot S$ is a homomorphic relational product, then $\{R, S\} \vdash R\cdot S$.
\end{df}

\begin{prop}\label{thm:homrelprod}
Let $n \in \omega\setminus\{0, 1\}$ and assume that $0\le i < n-1$. Then
${S_{n,i} \cdot S_{n, i+1}}$ is a homomorphic relational product and equals $R_{n,i,i+1}$.
Consequently, ${\{S_{n,i}, S_{n, i+1}\} \vdash R_{n,i,i+1}}$.
\end{prop}

\begin{proof}
The inclusion $R_{n,i,i+1} \subseteq S_{n,i} \cdot S_{n,i+1}$ is evident 
since the relations $S_{n,i}$ and $S_{n, i+1}$ are reflexive.
For the reverse inclusion, let us suppose that $(\mbf_k,\mbf_l)~\in~(S_{n,i} \cdot S_{n,i+1})\setminus R_{n,i,i + 1}$. From the description of $R_{n,i,i+ 1}$ given in Subsection~\ref{subsec:Rnij},  we know that $0\le k\le
i$ and $i + 2\le l\le n$. Hence $\mbf_k$ is from the top block of the relation $S_{n,i}$ and $\mbf_l$ is from the bottom block of the
relation $S_{n,i+1}$---see Figure~\ref{fig:RforJ3}.
Since $(\mbf_k,\mbf_l)\in S_{n,i} \cdot S_{n,i+1}$, there exists $c\in J_n$ such that $(\mbf_k,c)\in S_{n,i}$ and $(c,\mbf_l)\in S_{n,i+1}$.
Consequently, $c\in \{\mbf_0, \dots, \mbf_i\}\cap \{\mbf_{i+2}, \dots, \mbf_n\}$, a contradiction. The same argument can be applied to the pairs
$(\mbt_k,\mbt_l)\in (S_{n,i} \cdot S_{n,i+1})\setminus R_{n,i,i +1}$. Hence we obtain $R_{n,i,i+1} = S_{n,i} \cdot S_{n,i+1}$.

To show that this relational product is homomorphic, we define a
homomorphism $u\colon \mathbf{R}_{n,i,i+1} \to \JB_n$ such that for all
$(\mbf_k,\mbf_l)\in R_{n,i,i + 1}$  we have
\[
(\mbf_k,u(\mbf_k,\mbf_l))\in S_{n,i} \text{ and }
(u(\mbf_k,\mbf_l),\mbf_l)\in S_{n,i+1},
\]
and likewise for $(\mbt_k,\mbt_l)\in R_{n,i,i + 1}$. The homomorphism $u$ is defined on the pairs $(\mbf_k,\mbf_l)\in R_{n,i,i+1}$
by:
\[
u(\mbf_k,\mbf_l)
:= \begin{cases} \mbf_k & \quad\text{if $k < i + 1$},\\
\mbf_{i+1} & \quad\text{if $k \ge i + 1 \text{ and } l\le i + 1$},\\
\mbf_l & \quad\text{if $l > i + 1$},
\end{cases}
\]
and likewise on the pairs $(\mbt_k,\mbt_l)\in R_{n,i,i + 1}$. We note that the cases $k < i + 1$ and $l > i + 1$ cannot happen
simultaneously since $(\mbf_k,\mbf_l)\in
R_{n,i,i + 1}$ (see the description of $R_{n,i,i+1}$ in Subsection~\ref{subsec:Rnij}), and that for $k < i +1$ the map $u$ behaves as the first
projection on $(\mbf_k,\mbf_l)$, while for $l > i +1$ it behaves as the second projection on $(\mbf_k,\mbf_l)$.

The fact that the map $u$ is a lattice homomorphism is easy to see;
Figure~\ref{fig:Rnij} gives a drawing of $R_{n,i,i+1}$ and
Figure~\ref{fig:ker(u)} shows the kernel of the map $u \colon
\mathbf{R}_{3,0,1} \to \JB_3$ restricted to $F_{\!3}^{0,1}$. By
construction, the map $u$ preserves~$\neg$. Finally, $u$ preserves
the constants as each block of $\ker(u)$ contains a unique element
of the diagonal of $\JB_n^2$ (this is where we use the fact that we
are dealing with $R_{n,i,i+1}$ rather than a general $R_{n,i,j}$). Hence
indeed the map
$u\colon \mathbf{R}_{n,i,i+1} \to \JB_n$ is a homomorphism.

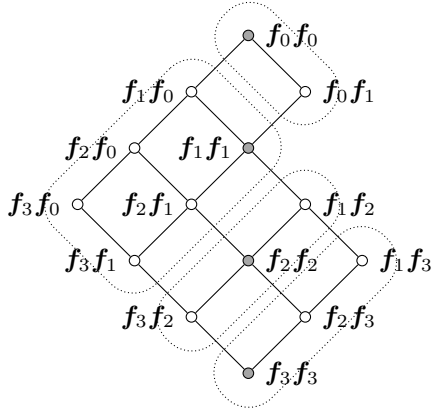
\begin{figure}
\centering
\begin{tikzpicture}[scale=0.75]
  \node[shaded] (f33) at (0,0) {};
  \node[unshaded] (f23) at (1,1) {};
  \node[unshaded] (f13) at (2,2) {};
  \node[unshaded] (f32) at (-1,1) {};
  \node[shaded] (f22) at (0,2) {};
  \node[unshaded] (f12) at (1,3) {};
  \node[unshaded] (f31) at (-2,2) {};
  \node[unshaded] (f21) at (-1,3) {};
  \node[shaded] (f11) at (0,4) {};
  \node[unshaded] (f01) at (1,5) {};
  \node[unshaded] (f30) at (-3,3) {};
  \node[unshaded] (f20) at (-2,4) {};
  \node[unshaded] (f10) at (-1,5) {};
  \node[shaded] (f00) at (0,6) {};
  \draw[order] (f33) -- (f23) -- (f13);
  \draw[order] (f32) -- (f22) -- (f12);
  \draw[order] (f31) -- (f21) -- (f11) -- (f01);
  \draw[order] (f30) -- (f20) -- (f10) -- (f00);
  \draw[order] (f33) -- (f32) -- (f31) -- (f30);
  \draw[order] (f23) -- (f22) -- (f21) -- (f20);
  \draw[order] (f13) -- (f12) -- (f11) -- (f10);
  \draw[order] (f01) -- (f00);
  \node[label,anchor=west] at (f33) {$\mbf_3\mbf_3$};
  \node[label,anchor=west] at (f23) {$\mbf_2\mbf_3$};
  \node[label,anchor=west] at (f13) {$\mbf_1\mbf_3$};
  \node[label,anchor=east] at (f32) {$\mbf_3\mbf_2$};
  \node[label,anchor=west] at (f22) {$\mbf_2\mbf_2$};
  \node[label,anchor=west] at (f12) {$\mbf_1\mbf_2$};
  \node[label,anchor=east] at (f31) {$\mbf_3\mbf_1$};
  \node[label,anchor=east] at (f21) {$\mbf_2\mbf_1$};
  \node[label,anchor=east] at (f11) {$\mbf_1\mbf_1$};
  \node[label,anchor=west] at (f01) {$\mbf_0\mbf_1$};
  \node[label,anchor=east] at (f30) {$\mbf_3\mbf_0$};
  \node[label,anchor=east] at (f20) {$\mbf_2\mbf_0$};
  \node[label,anchor=east] at (f10) {$\mbf_1\mbf_0$};
  \node[label,anchor=west] at (f00) {$\mbf_0\mbf_0$};
  \node[blob,rotate=45,inner xsep=1.5cm,inner ysep=0.45cm] at (f23) {};
  \node[blob,rotate=45,inner xsep=1.5cm,inner ysep=0.45cm] at (f22) {};
  \node[blob,rotate=45,inner xsep=1.5cm,inner ysep=0.95cm] at ($0.5*(f20)+0.5*(f21)$) {};
  \node[blob,rotate=-45,inner xsep=0.95cm,inner ysep=0.45cm] at ($0.5*(f00)+0.5*(f01)$) {};
\end{tikzpicture}
\caption{The kernel of $u \colon \mathbf{R}_{3,0,1} \to \JB_3$ restricted to $\F_{\!3}^{0,1}$}\label{fig:ker(u)}
\end{figure}

To complete the proof we first assume that $(\mbf_k,\mbf_l)\in R_{n,i,i + 1}$ with $k < i +1$. Then
\[
(\mbf_k,u(\mbf_k,\mbf_l))= (\mbf_k,\mbf_k)\in S_{n,i}. 
\]
As $k < i +1$ and $(\mbf_k,\mbf_l)\in R_{n,i,i + 1}$, it follows that
$l < i + 1$ and hence
\[
(u(\mbf_k,\mbf_l),\mbf_l)=(\mbf_k,\mbf_l)\in S_{n,i + 1}.
\]
Now assume that $(\mbf_k,\mbf_l)\in R_{n,i,i + 1}$
with $l > i + 1$. Then
necessarily $k > i$ and so
\[
(\mbf_k,u(\mbf_k,\mbf_l))= (\mbf_k,\mbf_l)\in S_{n,i} \text{ and }
(u(\mbf_k,\mbf_l),\mbf_l)=(\mbf_l,\mbf_l)\in S_{n,i + 1}.
\]
The final case that is that $k \ge i + 1$ and $l\le i + 1$. Then
\[
(\mbf_k,u(\mbf_k,\mbf_l))= (\mbf_k,\mbf_{i+1})\in S_{n,i} \text{ and }
(u(\mbf_k,\mbf_l),\mbf_l)=(\mbf_{i+1},\mbf_l)\in S_{n,i + 1}.
\]
The same arguments apply for the pairs $(\mbt_k,\mbt_l)\in R_{n,i,i + 1}$. Hence $R_{n,i,i + 1}$ is a homomorphic relational product of $S_{n,i}$ and $S_{n, i+1}$.
\end{proof}

\begin{rem}\label{rem:2nd-opt}
Let $n \in \omega\setminus\{0, 1\}$ and let $0\le i < j < n$. The first half of the proof of
Proposition~\ref{thm:homrelprod}
is easily modified to show that $R_{n,i,j} = S_{n,i} \cdot S_{n, j}$. We will see in Proposition~\ref{prop:Rnij-hom-min} that, for $j > i +1$, the relation $R_{n,i,j}$ is not entailed by $\{S_{n,i}, S_{n, j}\}$. Hence the relational product $S_{n,i} \cdot S_{n, j}$ is homomorphic if and only if $j = i + 1$.
\end{rem}

\begin{proof}[\textbf{Proof of Theorem~\ref{cor:bigduality}: duality}]

As we already observed, the NU Duality Theorem~\ref{NUDT} implies that $\JT^2_n = \langle J_n; \Sub(\JB_n^2), \T\rangle$, yields a duality on
$\CJ_n =\ISP(\JB_n)$, where $\Sub(\JB_n^2)$
is the set of all compatible binary relations on $\JB_n$. Moreover, by using the admissible constructs of intersection and converse, it follows from Theorem~\ref{thm:MI} that $\JT_n^{\mathrm{mi}} = \langle J_n; \mathcal{R}_{(n)}^{\mathrm{mi}}, \T\rangle$ yields a duality on $\CJ_n$, where
\begin{align*}
\mathcal \mathcal{R}_{(0)}^{\mathrm{mi}} &= \{S_{0,0}\}, \qquad \mathcal{R}_{(1)}^{\mathrm{mi}} = \{S_{1,0}, S_{1,1}\},
\qquad \mathcal{R}_{(2)}^{\mathrm{mi}} = \{S_{2,0}, S_{2,1}, S_{2,2}, R_{2,0,1}\}, \\
 \mathcal{R}_{(3)}^{\mathrm{mi}} &= \{S_{3,0}, S_{3,1}, S_{3,2}, S_{3,3}, R_{3,0,1}, R_{3,0,2}, R_{3,1,2}\},
\end{align*}
and, in general, for $n\ge 2$,
\[
\mathcal{R}_{(n)}^{\mathrm{mi}} = \big\{\,S_{n,i} \mid 0\le i \le n\,\big\} \cup \big\{\,R_{n,i,j}\mid i,j \in \{0,\ldots,n-1\} \text{ with } i < j\,\big\}.
\]
By Theorem~\ref{thm:homrelprod}, we can remove all the relations of the form $R_{n,i,i+1}$ from the alter egos without destroying the duality. This proves that, for all $n\in \omega$, the alter ego $\JT_n$ yields a duality on the quasivariety~$\CJ_n$. Other than the optimality claim, which will be proved in the next section, this proves~(1).

Both $\JB_0$ and $\JB_1$ are subdirectly irreducible
and have no proper subalgebras. Therefore they both have irreducibility index equal to~$1$---see \cite[page~82]{CD98}.
Hence, by  \cite[Theorem 3.3.7]{CD98}, the dualities induced by $\JT_0$ and $\JT_1$ are strong as the only compatible unary partial operations on $\JB_0$ and $\JB_1$ are the identity maps on $J_0$ and $J_1$, respectively. Hence (2) holds.
(The fact that these dualities are strong also follows from the
single-sorted version of our Special Multi-sorted NU Strong Duality Theorem~\ref{MultiNUDT}.)

Now let $n\ge 2$. As $\JB_n$ has no proper subalgebras and $\Con(\JB_n) \cong \two^n \oplus \one$, it follows that the irreducibility index of $\JB_n$ equals $n$. Hence, again by \cite[Theorem 3.3.7]{CD98}, the duality given by $\JT_n$ may be upgraded to a strong duality by adding all compatible $n$-ary partial operations on $\JB_n$ to the structure of the alter ego $\JT_n$. This proves (3).

Finally, (4) is an easy calculation.
\end{proof}

\section{Proving that the duality on $\CJ_n$ given by $\JT_n$ is optimal}\label{sec:opt}

The last step in the proof of Theorem~\ref{cor:bigduality} is to prove that the duality is optimal for all $n\in \omega$. We will do this in two steps.

\begin{description}
\item[\textit{Step 1}] We prove that, for all $n\in \omega$ and for $0\le i< n$, none of the relations $S_{n, i}$ can be deleted from $\mathcal R_{(n)}$ without destroying the duality.
\end{description}
For the second step of the proof, we require another definition. A compatible binary relation $R$ on a finite algebra $\M$ is \emph{absolutely unavoidable} within $\Sub(\M^2)$ if, for every set $\mathcal{R}$ of compatible binary relations on $\M$ such that $\MT = \langle M; \mathcal{R}, \T\rangle$ yields a duality on $\ISP(\M)$, we have $\mathcal{R}\cap \{R, R\conv\,\} \neq \varnothing$.

\begin{description}
\item[\textit{Step 2}] We prove that all of the remaining relations, that is, $S_{n,n}$, for $n\in \omega$, and $R_{n, i, j}$, for $n\ge 3$ and $i,j \in \{0,\ldots,n-1\}$  with $i < j-1$, are absolutely unavoidable and therefore cannot be deleted from the set $\mathcal R_{(n)}$ without destroying the duality.
\end{description}

To show that $S_{n, i}$ cannot be removed from the alter ego $\JT_n$ without destroying the duality, we must find an algebra $\A\in \CJ_n$ and a continuous map $\gamma \colon \mathrm{D}(\A) \to J_n$ that preserves all the relations in $\mathcal R_{(n)} \setminus \{S_{n, i}\}$ but does not preserve~$S_{n, i}$. The Test Algebra Lemma~\cite[8.1.3]{CD98} tells us that we can choose $\A$ to be the subalgebra $\S_{n, i}$ of $\JB_n^2$ with underlying set~$S_{n, i}$.

Let $S$ be a compatible binary relation on $\JB_n$ and let $\S$ be the subalgebra of $\JB_n^2$ with underlying set~$S$. Throughout this section, much use will be made of the two restricted projections
$\rho_i^\S := \pi_i\rest S \colon \S \to \JB_n$, for $i\in \{1, 2\}$. Since it will always be clear which restriction is intended, to simplify the notation we will write $\rho_i$ rather than $\rho_i^\S$.

The following proposition completes \emph{Step 1}.
\begin{prop}\label{prop:Sni-in-dual}
Let $n \in \omega\setminus \{0\}$. For each of the relations $S_{n,i}$ such that $0\le i \le n-1$, there is a map
\[
\gamma\colon \mathrm{D}(\mathbf{S}_{n,i}) \to J_n
\]
that preserves all the relations $S_{n,j}$, for $j\in \{0,\dots,n\}\setminus \{i\}$, but does not preserve the relation $S_{n,i}$. In addition, for all $n\ge 3$, the map $\gamma$ preserves all the relations $R_{n, j, k}$, for $j,k \in \{0,\ldots,n-1\}$  with $j < k-1$.
\end{prop}

\begin{proof}
Define a map $\gamma\colon \mathrm{D}(\mathbf{S}_{n,i}) \to J_n$ by
\[
\gamma(h) := \begin{cases} \mbf_i & \quad\text{if $h=\rho_1$},\\
\mbf_{i+1} & \quad\text{otherwise}.
\end{cases}
\]
The fact that $\gamma$ does not preserve the relation $S_{n,i}$ is witnessed on the pair $(\rho_1,\rho_2)$  from the dual $\mathrm{D}(\mathbf{S}_{n,i})$ as 
$(\rho_1,\rho_2)\in S_{n,i}^{\mathrm{D}(\mathbf{S}_{n,i})}$ yet
$(\gamma(\rho_1),\gamma(\rho_2))\notin S_{n,i}$.
That $\gamma$ preserves all the relations $S_{n,j}$, for $j\in \{0,\dots,n-1\}\setminus \{i\}$, follows by observing that $\{\mbf_i, \mbf_{i+1}\}^2\subseteq S_{n,j}$, for $j\in \{0,\dots,n-1\}\setminus \{i\}$. Indeed, for $h_1,h_2\in \mathrm{D}(\mathbf{S}_{n,i})$ with $(h_1,h_2)\in
S_{n,j}^{\mathrm{D}(\mathbf{S}_{n,i})}$ we have
$(\gamma(h_1),\gamma(h_2))\in \{\mbf_i, \mbf_{i+1}\}^2$ and so
$(\gamma(h_1),\gamma(h_2))\in S_{n,j}$. Now assume 
$n\ge 3$. Since $R_{n, j, k} = S_{n,j}\cup S_{n,k}$, the assumption that $j <
k-1$ guarantees that at least one of $j$ and $k$ is not $i$. Hence
the same argument shows that $\gamma$ preserves $R_{n, j, k}$, for
all $j,k \in \{0,\ldots,n-1\}$ with $j < k-1$.
\end{proof}

We now turn to \emph{Step 2}.
The following general result is new and provides a useful sufficient condition for a compatible binary relation to be absolutely unavoidable. Recall that a compatible binary relation $S$ on $\M$ is \defn{hom-minimal} if $\mathrm{D}(\S) = \{\rho_1, \rho_2\}$.

\begin{prop}\label{lem:hommin}
Let $\M$ be a finite algebra and let $S$ be a compatible binary relation on $\M$.
\begin{enumerate}[ \normalfont(1)]

\item
If $S$ is hom-minimal, is a value at $(a, b)$ and satisfies
$a\in \rho_1(S)$ and $b\in \rho_2(S)$, then $S$ is absolutely unavoidable within $\Sub(\M^2)$.

\item If $S$ is hom-minimal, diagonal and meet-irreducible in $\Sub(\M^2)$, then $S$ is absolutely unavoidable within $\Sub(\M^2)$.

\end{enumerate}
\end{prop}

\begin{proof}
Since (2) is an immediate consequence of (1), by Lemma~\ref{lem:CMI=Value}, we prove only~(1). Assume $S$ is hom-minimal, is a value at $(a, b)$ and satisfies $a\in \rho_1(S)$ and $b\in \rho_2(S)$. Since $S$ is hom-minimal we may define $\gamma \colon \mathrm{D}(\mathbf S)\to M$ by $\gamma(\rho_1) = a$ and $\gamma(\rho_2) = b$. Since $(\rho_1, \rho_2) \in S^{\mathrm{D}(\mathbf S)}$ and $(a, b) \notin S$, the map $\gamma$ does not preserve~$S$, and it remains to prove that $\gamma$ preserves every
relation $R$ in $\mathcal{R}_{\M}\comp \{S, S\conv\,\}$.

Let $R\in \mathcal{R}_{\M}\comp \{S, S\conv\,\}$ and assume $(x,y) \in R^{\mathrm{D}(\mathbf S)}$, for some $x, y \in \mathrm{D}(\mathbf{S})$. We must prove that $(\gamma(x), \gamma(y)) \in R$.
We consider separately the four cases for the pair $(x, y)$.

Assume that $(\rho_1, \rho_2) \in R^{\mathrm{D}(\mathbf S)}$. Then $S\subseteq R$. As $S\ne R$ and $S$ is a value at $(a, b)$,
we have
\[
(\gamma(x), \gamma(y)) = (\gamma(\rho_1), \gamma(\rho_2)) = (a,b) \in R.
\]

Assume that $(\rho_2, \rho_1) \in R^{\mathrm{D}(\mathbf S)}$. Then $S\conv\subseteq R$ and hence $S\subseteq R\conv$. As $S\ne R\conv$, we have $(a, b) \in R\conv$, whence $(b, a)\in R$. Thus,
\[
(\gamma(x), \gamma(y)) = (\gamma(\rho_2), \gamma(\rho_1)) = (b,a) \in R.
\]

Now assume that $(\rho_1, \rho_1) \in R^{\mathrm{D}(\mathbf S)}$. Then $\{\, (c, c) \mid c\in \rho_1(S)\,\}\subseteq R$. As $a\in \rho_1(S)$, by assumption we have
\[
(\gamma(x), \gamma(y)) = (\gamma(\rho_1), \gamma(\rho_1)) = (a,a) \in R.
\]
The case where $(\rho_2, \rho_2) \in R^{\mathrm{D}(\mathbf S)}$ follows by symmetry using the fact that $b\in \rho_2(S)$.
Hence $\gamma$ preserves $R$, as required.
\end{proof}

We will now show that, for all $n\in \omega\setminus\{0\}$, the relation $S_{n,n}$ is hom-minimal, and that, for all $n\ge 3$, the relation $R_{n,i,j}$ is hom-minimal for all $i,j \in \{0,\ldots,n-1\}$ with $i < j-1$; it will then follow, by Proposition~\ref{lem:hommin}(2),
that each of these relations is absolutely unavoidable within $\Sub(\JB_n^2)$. While $S_{0,0}$ is not hom-minimal (indeed, a simple calculation shows that $|\CJ_0(\S_{0,0}, \JB_0)| = 6$), we will prove directly that $S_{0,0}$ is absolutely unavoidable within $\mathcal{R}_{\JB_0}$.

\begin{prop}\label{prop:dualSnnSIX=pi1pi2}
For all $n\in \omega\setminus\{0\}$, the relation $S_{n,n}$ is hom-minimal.
\end{prop}

\begin{proof}
Let $n\in \omega\setminus\{0\}$ and let $h\colon \mathbf{S}_{n,n}
\to \JB_n$ be a homomorphism. To show $h\in \{\rho_1,
\rho_2\}$, we first analyse the structure of the truth-lattice
reduct of $\S_{n,n}$.

The truth-lattice order on $\S_{n,n}$ is best viewed as
\begin{multline*}
\big(\{\mbf_0, \dots, \mbf_n, \bot\} \times \{\mbf_0, \dots,
\mbf_n, \top\}\big)\\ \cup \big(\{\bot, \mbt_n, \dots, \mbt_0\} \times
\{\top, \mbt_n, \dots, \mbt_0\}\big) \cup \{(\top,\top), (\bot,\bot)\},
\end{multline*}
that is, the union of a product of two chains with
another product of two chains, that overlap only at $(\bot,\top)$,
with two additional elements, $(\top,\top)$ and $(\bot,\bot)$,
added---see Figure~\ref{fig:S00S11}.

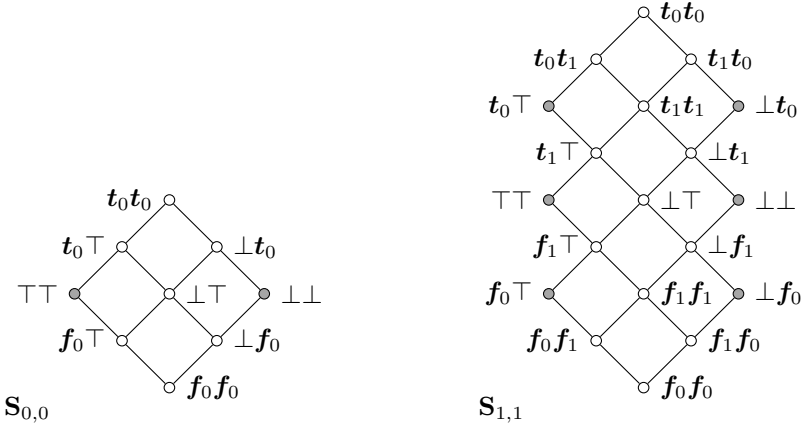
\begin{figure}[ht]
\centering
\begin{tikzpicture}[scale=0.625]
  \begin{scope}
    \node[anchor=north] at (-3,0) {$\S_{0,0}$};
    \node[unshaded] (f0f0) at (0,0) {};
    \node[unshaded] (botf0) at (1,1) {};
    \node[shaded] (botbot) at (2,2) {};
    \node[unshaded] (f0top) at (-1,1) {};
    \node[unshaded] (bottop) at (0,2) {};
    \node[unshaded] (bott0) at (1,3) {};
    \node[shaded] (toptop) at (-2,2) {};
    \node[unshaded] (t0top) at (-1,3) {};
    \node[unshaded] (t0t0) at (0,4) {};
    \draw[order] (f0f0) -- (botf0) -- (botbot);
    \draw[order] (f0top) -- (bottop) -- (bott0);
    \draw[order] (toptop) -- (t0top) -- (t0t0);
    \draw[order] (f0f0) -- (f0top) -- (toptop);
    \draw[order] (botf0) -- (bottop) -- (t0top);
    \draw[order] (botbot) -- (bott0) -- (t0t0);
    \node[label,anchor=west] at (f0f0) {$\mbf_0\mbf_0$};
    \node[label,anchor=west] at (botf0) {$\bot\mbf_0$};
    \node[label,anchor=west] at (botbot) {$\bot\bot$};
    \node[label,anchor=east] at (f0top) {$\mbf_0\top$};
    \node[label,anchor=west] at (bottop) {$\bot\top$};
    \node[label,anchor=west] at (bott0) {$\bot\mbt_0$};
    \node[label,anchor=east] at (toptop) {$\top\top$};
    \node[label,anchor=east] at (t0top) {$\mbt_0\top$};
    \node[label,anchor=east] at (t0t0) {$\mbt_0\mbt_0$};
  \end{scope}
  \begin{scope}[xshift=10cm]
    \node[anchor=north] at (-3,0) {$\S_{1,1}$};
    \node[unshaded] (f0f0) at (0,0) {};
    \node[unshaded] (f1f0) at (1,1) {};
    \node[shaded] (botf0) at (2,2) {};
    \node[unshaded] (f0f1) at (-1,1) {};
    \node[unshaded] (f1f1) at (0,2) {};
    \node[unshaded] (botf1) at (1,3) {};
    \node[shaded] (botbot) at (2,4) {};
    \node[shaded] (f0top) at (-2,2) {};
    \node[unshaded] (f1top) at (-1,3) {};
    \node[unshaded] (bottop) at (0,4) {};
    \node[unshaded] (bott1) at (1,5) {};
    \node[shaded] (bott0) at (2,6) {};
    \node[shaded] (toptop) at (-2,4) {};
    \node[unshaded] (t1top) at (-1,5) {};
    \node[unshaded] (t1t1) at (0,6) {};
    \node[unshaded] (t1t0) at (1,7) {};
    \node[shaded] (t0top) at (-2,6) {};
    \node[unshaded] (t0t1) at (-1,7) {};
    \node[unshaded] (t0t0) at (0,8) {};
    \draw[order] (f0f0) -- (f1f0) -- (botf0);
    \draw[order] (f0f1) -- (f1f1) -- (botf1) -- (botbot);
    \draw[order] (f0top) -- (f1top) -- (bottop) -- (bott1) -- (bott0);
    \draw[order] (toptop) -- (t1top) -- (t1t1) -- (t1t0);
    \draw[order] (t0top) -- (t0t1) -- (t0t0);
    \draw[order] (f0f0) -- (f0f1) -- (f0top);
    \draw[order] (f1f0) -- (f1f1) -- (f1top) -- (toptop);
    \draw[order] (botf0) -- (botf1) -- (bottop) -- (t1top) -- (t0top);
    \draw[order] (botbot) -- (bott1) -- (t1t1) -- (t0t1);
    \draw[order] (bott0) -- (t1t0) -- (t0t0);
    \node[label,anchor=west] at (f0f0) {$\mbf_0\mbf_0$};
    \node[label,anchor=west] at (f1f0) {$\mbf_1\mbf_0$};
    \node[label,anchor=west] at (botf0) {$\bot\mbf_0$};
    \node[label,anchor=east] at (f0f1) {$\mbf_0\mbf_1$};
    \node[label,anchor=west] at (f1f1) {$\mbf_1\mbf_1$};
    \node[label,anchor=west] at (botf1) {$\bot\mbf_1$};
    \node[label,anchor=west] at (botbot) {$\bot\bot$};
    \node[label,anchor=east] at (f0top) {$\mbf_0\top$};
    \node[label,anchor=east] at (f1top) {$\mbf_1\top$};
    \node[label,anchor=west] at (bottop) {$\bot\top$};
    \node[label,anchor=west] at (bott1) {$\bot\mbt_1$};
    \node[label,anchor=west] at (bott0) {$\bot\mbt_0$};
    \node[label,anchor=east] at (toptop) {$\top\top$};
    \node[label,anchor=east] at (t1top) {$\mbt_1\top$};
    \node[label,anchor=west] at (t1t1) {$\mbt_1\mbt_1$};
    \node[label,anchor=west] at (t1t0) {$\mbt_1\mbt_0$};
    \node[label,anchor=east] at (t0top) {$\mbt_0\top$};
    \node[label,anchor=east] at (t0t1) {$\mbt_0\mbt_1$};
    \node[label,anchor=west] at (t0t0) {$\mbt_0\mbt_0$};
  \end{scope}
\end{tikzpicture}
\caption{The truth-lattice reducts of $\S_{0,0}$ and $\S_{1,1}$.}\label{fig:S00S11}
\end{figure}

It is easily seen that the set $D$ of doubly-irreducible elements of the truth-lattice reduct of $\S_{n,n}$ is
\[
D = \{(\top,\top), (\mbf_0,\top), (\bot,\mbf_0), (\mbt_0,\top), (\bot,\mbt_0), (\bot,\bot) \}.
\]
See the right of Figure~\ref{fig:S00S11}, where doubly-irreducible elements are shaded. (Note that this uses our assumption that $n\ne 0$ as it fails in $\S_{0,0}$; see the left of Figure~\ref{fig:S00S11}.) It is also easily seen that the truth-lattice reduct of $\S_{n,n}$ is generated as a lattice by
\[
D\cup
\{\,(\mbf_i,\mbf_i) \mid 1 \le i\le n\,\} \cup \{\,(\mbt_i,\mbt_i) \mid 1 \le i\le n\,\}.
\]
(For example, $(\mbf_n,\top) = (\mbf_0,\top) \vee (\mbf_n,\mbf_n)$,
$(\bot,\top) = (\mbf_0,\top) \vee (\bot,\mbf_0)$, and, if $i\le j$, then
$(\mbf_i,\mbf_j) =  ((\mbf_i,\mbf_i) \vee (\mbf_0,\top))\wedge ((\mbf_j,\mbf_j) \vee (\bot,\mbf_0))$.)
Hence, to prove that $h = \rho_1$, for example, it suffices to prove that $h$ acts as the first projection on $D$, and thus, since $h$ preserves $\neg$, it suffices to prove that
$h((\mbf_0,\top)) = \mbf_0$ and that $h((\bot,\mbf_0)) = \bot$.

Since
$\neg(\bot,\top) = (\bot,\top)$, we have $h((\bot,\top)) \in \{\bot,\top\}$.
Without loss of generality we may assume that
$h((\bot,\top)) = \bot$. From the equalities
\[
(\mbf_0,\top) \vee (\bot,\mbf_0) =(\bot,\top), \quad (\mbf_0,\top) \wedge (\bot,\mbf_0) = (\mbf_0,\mbf_0), \quad h((\bot,\top)) =
\bot,
\]
and $h((\mbf_0,\mbf_0)) = \mbf_0$, it follows that
\begin{multline*}
\big(h((\mbf_0,\top)) = \mbf_0 \And h((\bot,\mbf_0)) = \bot\big)
\\
 \text{or}\quad\big(h((\mbf_0,\top)) = \bot \And h((\bot,\mbf_0)) = \mbf_0\big).
\end{multline*}
As the former completes
the proof, suppose that the latter holds. Since $h$ preserves
$\le_\k$ and $(\mbf_n,\top) \le_\k (\mbf_0,\top)$, we have
$h((\mbf_n,\top)) = \bot$, whence ${\mbf_n = h((\mbf_n,\mbf_n))
\le_\k h((\mbf_n,\top)) = \bot}$,
a contradiction. Hence
$h((\mbf_0,\top)) = \mbf_0$ and $h((\bot,\mbf_0)) = \bot$, as
required.
\end{proof}

\begin{prop}\label{prop:absunavoid}
Let $n \in \omega$. The relation $S_{n,n}$ is absolutely unavoidable
within $\Sub(\JB_n^2)$.
\end{prop}
\begin{proof}
For $n\in \omega\setminus\{0\}$, this follows at once from
Proposition~\ref{lem:hommin} and Proposition~\ref{prop:dualSnnSIX=pi1pi2}.
It remains to prove that $S_{0,0}$ is absolutely unavoidable
within~$\mathcal{R}_{\JB_0}$.
This is an easy consequence of Theorem~\ref{thm:MI}. Indeed,
as case (d) of Theorem~\ref{thm:MI} does not apply to
$\JB_0$, the only meet-irreducibles in
the lattice $\mathcal{R}_{\JB_0}$ are $S_{0,0}$ and $S\convsub{0,0}$, whence $\mathcal{R}_{\JB_0} = \{\Delta, S_{0,0}, S\convsub{0,0}, J_0^2\}$, where $\Delta$ is the diagonal relation. It follows that $S_{0,0}$ is absolutely unavoidable
within $\mathcal{R}_{\JB_0}$.
\end{proof}

The hom-minimality and absolute unavoidability of the relation $S_{n,n}$ was proved in~\cite{C-thesis} via quite different proofs.

We now turn our attention to the hom-minimality of $R_{n, i, j}$. Our first lemma applies to every compatible binary relation on $\JB_n$ that contains
the pair $(\mbf_n, \mbf_0)$. The areas $A$, $B$, and $C$ referred to in the lemma are shown in Figure~\ref{fig:ABC}. For all $(a, b), (c, d)\in J_n^2$ with $(a, b)\le_\k (c, d)$ we denote the interval from $(a, b)$ to $ (c, d)$ in the knowledge order by $[(a, b), (c, d)]_\k$.

\begin{lem}\label{lem:hxnx0-implies-rho}
Let $n\in \omega\setminus\{0\}$. Assume that $S$ is a compatible binary relation on $\JB_n$ with $(\mbf_n, \mbf_0)\in S$, let $h\colon \mathbf{S}\to \JB_n$ be a homomorphism and assume that $h((\mbf_n, \mbf_0)) = \mbf_m$, with $0 \le m\le n$.
\begin{enumerate}[ \normalfont (a)]

\item \emph{Area A:} $h$ equals $\rho_1$ on $S\cap [(\mbf_m, \mbf_m), (\mbf_0, \mbf_0)]_\k$.

\item \emph{Area B:} $h$ equals $\rho_2$ on $S\cap [(\mbf_n, \mbf_n), (\mbf_m, \mbf_m)]_\k$.

\item \emph{Area C:} $h((\mbf_s, \mbf_t)) = \mbf_m$, for all $(\mbf_s, \mbf_t)\in S \cap [((\mbf_n, \mbf_m)), (\mbf_m, \mbf_0)]_\k$.

\item $h = \rho_1$ on $S \cap \bF^2$ if and only if $h((\mbf_n, \mbf_0)) = \mbf_n$, and $h = \rho_2$ on $S \cap \bF^2$ if and only if $h((\mbf_n, \mbf_0)) = \mbf_0$.

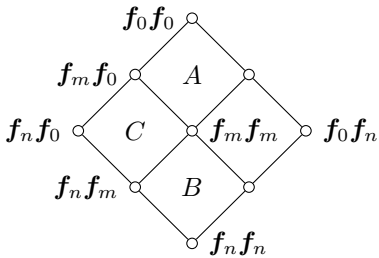
\begin{figure}[ht]
\centering
\begin{tikzpicture}[scale=0.75]
  \node[unshaded] (xnxn) at (0,0) {};
  \node[unshaded] (xmxn) at (1,1) {};
  \node[unshaded] (x0xn) at (2,2) {};
  \node[unshaded] (xnxm) at (-1,1) {};
  \node[unshaded] (xmxm) at (0,2) {};
  \node[unshaded] (x0xm) at (1,3) {};
  \node[unshaded] (xnx0) at (-2,2) {};
  \node[unshaded] (xmx0) at (-1,3) {};
  \node[unshaded] (x0x0) at (0,4) {};
  \draw[order] (xnxn) -- (xmxn) -- (x0xn);
  \draw[order] (xnxm) -- (xmxm) -- (x0xm);
  \draw[order] (xnx0) -- (xmx0) -- (x0x0);
  \draw[order] (xnxn) -- (xnxm) -- (xnx0);
  \draw[order] (xmxn) -- (xmxm) -- (xmx0);
  \draw[order] (x0xn) -- (x0xm) -- (x0x0);
  \node[label,anchor=west] at (xnxn) {$\mbf_n\mbf_n$};
  \node[label,anchor=west] at (x0xn) {$\mbf_0\mbf_n$};
  \node[label,anchor=east] at (xnxm) {$\mbf_n\mbf_m$};
  \node[label,anchor=west] at (xmxm) {$\mbf_m\mbf_m$};
  \node[label,anchor=east] at (xnx0) {$\mbf_n\mbf_0$};
  \node[label,anchor=east] at (xmx0) {$\mbf_m\mbf_0$};
  \node[label,anchor=east] at (x0x0) {$\mbf_0\mbf_0$};
  \node[label] at ($0.5*(xnxn)+0.5*(xmxm)$) {$B$};
  \node[label] at ($0.5*(xnx0)+0.5*(xmxm)$) {$C$};
  \node[label] at ($0.5*(x0x0)+0.5*(xmxm)$) {$A$};
\end{tikzpicture}
\caption{The areas $A$, $B$ and $C$.}\label{fig:ABC}
\end{figure}
\end{enumerate}
\end{lem}
\begin{proof}
(a) Let $(\mbf_s, \mbf_t)\in S$ with $(\mbf_m, \mbf_m) \le_\k (\mbf_s, \mbf_t)$. Then
\[
(\mbf_n, \mbf_0) \oplus (\mbf_s, \mbf_t) = (\mbf_s, \mbf_0) = (\mbf_n, \mbf_0) \oplus (\mbf_s, \mbf_s).
\]
Applying $h$, we have
$\mbf_m \oplus h((\mbf_s, \mbf_t)) = \mbf_m \oplus \mbf_s$.
Lastly, $(\mbf_m, \mbf_m)\le_\k (\mbf_s, \mbf_t)$ gives us that $\mbf_m \le_\k h((\mbf_s, \mbf_t))$ and hence
$h((\mbf_s, \mbf_t)) = \mbf_s$.

(b) This follows from (a) by duality and symmetry.

(c) Let $(\mbf_s, \mbf_t)\in S$ with $(\mbf_n, \mbf_m)\le_\k (\mbf_s, \mbf_t)\le_\k (\mbf_m, \mbf_0)$.
It follows that
$\mbf_s \le_\k \mbf_m \le_\k \mbf_t$, and hence
\[
(\mbf_s, \mbf_t) = \big((\mbf_n, \mbf_0) \otimes (\mbf_t, \mbf_t)\big) \oplus (\mbf_s, \mbf_s).
\]
Consequently,
\[
h((\mbf_s, \mbf_t)) = (\mbf_m \otimes \mbf_t) \oplus \mbf_s = \mbf_m.
\]

(d) This follows immediately from (a) and (b).
\end{proof}

The next lemma gives restrictions on the image of the element $(\mbf_n,\mbf_0)$ (and therefore restrictions on the image of the element $(\mbt_n,\mbt_0)$) under a homomorphism
$h\colon \mathbf{S}\to \JB_n$ when $S_{n,i} \subseteq S$.
We will use the following observation. Let $n\in \omega\setminus\{0\}$ and let $0\le i <n$, then
\begin{equation}
\begin{aligned}
&\big(\forall s\in \{0, \dots, n-1\}\big)\big(\forall t\in \{0, \dots, i\}\big)\ (\mbf_s, \mbf_t)\in S_{n, i}, \quad\text{and}\quad \\
&\big(\forall s\in \{i+1, \dots, n\}\big)\big(\forall t\in \{0, \dots, n-1\}\big)\ (\mbf_s, \mbf_t)\in S_{n, i} .
\end{aligned}\tag{$*$}
\end{equation}

\begin{lem}\label{lem:hxnx0}
Let $n\in \omega\setminus\{0\}$ and let $0\le i <n$. Assume that $S$ is a compatible binary relation on $\JB_n$ with $S_{n,i} \subseteq S$ and let $h\colon \mathbf{S}\to \JB_n$ be a homomorphism. Then
$h((\mbf_n,\mbf_0)) \in \{ \mbf_0, \mbf_i, \mbf_{i+1}, \mbf_n\}$.
\end{lem}
\begin{proof}
Assume that $h((\mbf_n,\mbf_0)) = \mbf_m$, with $0 \le m\le n$. We shall prove that $\mbf_m\in \{ \mbf_0, \mbf_i, \mbf_{i+1}, \mbf_n\}$.

By $(*)$, we have $(\mbf_0, \mbf_i)\in S$. Applying $h$ to both sides of the equation $(\mbf_n, \mbf_0) \oplus (\mbf_0, \mbf_i) = (\mbf_0, \mbf_0)$ gives $\mbf_m \oplus h((\mbf_0, \mbf_i)) = \mbf_0$. Hence either $\mbf_m = \mbf_0$ or $h((\mbf_0, \mbf_i)) = \mbf_0$. If $h((\mbf_0, \mbf_i)) = \mbf_0$, then applying $h$ to the inequality $(\mbf_n, \mbf_0) \otimes (\mbf_0, \mbf_i) = (\mbf_n, \mbf_i) \le_\k (\mbf_i, \mbf_i)$ gives $\mbf_m \otimes \mbf_0 \le_\k \mbf_i$ and therefore $\mbf_m \le_\k \mbf_i$. We have proved that
\[
\mbf_m = \mbf_0 \quad\text{or}\quad  \mbf_m \le_\k \mbf_i.\tag{$\dagger_1$}
\]

By $(*)$, we have $(\mbf_{i+1}, \mbf_n)\in S$. Applying $h$ to both sides of the equation $(\mbf_n, \mbf_0) \otimes (\mbf_{i+1}, \mbf_n) = (\mbf_n, \mbf_n)$ gives $\mbf_m \otimes h((\mbf_{i+1}, \mbf_n)) = \mbf_n$. Hence either $\mbf_m = \mbf_n$ or $h((\mbf_{i+1}, \mbf_n)) = \mbf_n$. If $h((\mbf_{i+1}, \mbf_n)) = \mbf_n$, then applying $h$ to
$(\mbf_n, \mbf_0) \oplus (\mbf_{i+1}, \mbf_n) = (\mbf_{i+1}, \mbf_0) \ge_\k (\mbf_{i+1}, \mbf_{i+1})$ gives $\mbf_m \oplus \mbf_n \ge_\k \mbf_{i+1}$ and therefore $\mbf_m \ge_\k \mbf_{i+1}$. We have proved that
\[
\mbf_m = \mbf_n \quad\text{or}\quad  \mbf_m \ge_\k \mbf_{i+1}.\tag{$\dagger_2$}
\]
Combining $(\dagger_1)$ and $(\dagger_2)$ gives
\begin{multline*}
(\mbf_m = \mbf_0 \And \mbf_m = \mbf_n) \text{ or } (\mbf_m = \mbf_0 \And \mbf_m \ge_\k \mbf_{i+1})\\
 \text{ or } (\mbf_m \le_\k \mbf_i \And \mbf_m = \mbf_n)  \text{ or }  (\mbf_m \le_\k \mbf_i \And \mbf_m \ge_\k \mbf_{i+1}).
\end{multline*}
Since $\mbf_0 \ne \mbf_n$, we conclude that $\mbf_m = \mbf_0$ or $\mbf_m = \mbf_n$ or $\mbf_{i+1} \le_\k \mbf_m \le_\k \mbf_i$, as required.
\end{proof}

Applying Lemma~\ref{lem:hxnx0} to $S=R_{n,i,j}$ gives us the following result.

\begin{lem}\label{lem:hxnx0-on-Rnij}
Let $n\ge 3$ and $i,j \in \{0,\ldots,n-1\}$  with $i < j-1$ and let $h \colon \mathbf{R}_{n,i,j} \to \JB_n$ be a homomorphism. Then $h((\mbf_n,\mbf_0))\in\{\mbf_0,\mbf_n\}$.
\end{lem}
\begin{proof}
Since $S_{n,i} \subseteq R_{n,i,j}$ and $S_{n,j} \subseteq R_{n,i,j}$, by Lemma~\ref{lem:hxnx0}  we have
\[
h((\mbf_n,\mbf_0))\in \{\mbf_0,\mbf_i,\mbf_{i+1},\mbf_n\} \cap \{\mbf_0,\mbf_j,\mbf_{j+1},\mbf_n\}.
\]
Since $0\leqslant i < j - 1$ we have $\{\mbf_0,\mbf_i,\mbf_{i+1},\mbf_n\} \cap \{\mbf_0,\mbf_j,\mbf_{j+1},\mbf_n\}=\{\mbf_0,\mbf_n\}$, as required.
\end{proof}

\begin{prop}\label{prop:Rnij-hom-min} Let $n \ge 3$ and let $i,j \in \{0,\ldots,n-1\}$ with $i < j-1$. Then the relation $R_{n,i,j}$
on $\JB_n$ is hom-minimal and therefore absolutely unavoidable within $\Sub(\JB_n^2)$.
\end{prop}
\begin{proof}
Let $h\colon \mathbf{R}_{n,i,j}\to \JB_n$ be a homomorphism.
Lemmas~\ref{lem:hxnx0-on-Rnij} and~\ref{lem:hxnx0-implies-rho}(4) imply that $h$ is a projection,
say $\rho_1$, when restricted to $R_{n,i,j} \cap \bF^2$. As $h$
preserves the negation $\neg$, it follows that $h$ equals $\rho_1$
when restricted to $R_{n,i,j} \cap \bT^2$. As $R_{n,i,j} \setminus
(\bF^2 \cup \bT^2) = \{(\top, \top), (\bot, \bot) \}$, we conclude
that $h = \rho_1$. Hence $R_{n,i,j}$ is hom-minimal and therefore
absolutely unavoidable within $\Sub(\JB_n^2)$ by
Theorem~\ref{thm:MIJ_n^2} and Proposition~\ref{lem:hommin}.
\end{proof}

This completes \emph{Step 2}. We may now complete the proof of Theorem~\ref{cor:bigduality}.

\begin{proof}[\textbf{Proof of Theorem~\ref{cor:bigduality}: optimality}]
By Proposition~\ref{prop:Sni-in-dual}, for $n \in \omega\setminus \{0\}$ and $0\le i \le n-1$, none of the relations $S_{n,i}$ can be removed from $\mathcal R_{(n)}$ without destroying the duality. By Proposition~\ref{prop:absunavoid},
the relation $S_{n,n}$ is absolutely unavoidable, for all $n\in\omega$, and hence cannot be removed from $\mathcal R_{(n)}$ without destroying the duality. Finally, by Proposition~\ref{prop:Rnij-hom-min}, for all $n \ge 3$ and all $i,j \in \{0,\ldots,n-1\}$ with $i < j-1$, the relation $R_{n,i,j}$ is absolutely unavoidable and so cannot be removed from $\mathcal R_{(n)}$ without destroying the duality. Hence the duality induced by $\JT_n$ is optimal.
\end{proof}

\section{The proof that $\protect\MT_n$ yields a multi-sorted duality on $\protect\CV_n$}\label{sec:multiproof}

By the Special Multi-sorted NU Strong Duality Theorem~\ref{MultiNUDT},
the structure
\[
\MT' = \langle M_0\du M_1\du \cdots \du M_n; \mathcal G, \mathcal R,  \T\rangle,
\]
where
\begin{itemize}

\item $\mathcal G = \bigcup \{\, \CA(\M_j, \M_k)\mid j, k\in\{0,1,\dots,n\}\,\}$ and

\item $ \mathcal R= \bigcup \{\, \Sub(\M_j\times\M_k)\mid j, k\in\{0,1,\dots,n\}\,\}$,
\end{itemize}
yields a multi-sorted strong duality on the variety~$\CV_n$. Our first step in refining this into a proof of Theorem~\ref{cor:bigmultiduality} is to describe the meet-irreducibles in the lattice $\Sub(\M_j\times \M_k)$, for $j, k\in \{0, \dots, n\}$ with $j \le k$.

We have already introduced the relations $\lek$ (and their converses $\gek$), for $k\in \{0, \dots, n\}$, and the relations~$\lejk$, for $j,k\in \{1, \dots, n\}$ with $j < k$.
In addition to these, we also require the compatible relations $S_\le$ and $S_\ge$ (from Subsection~\ref{subsec:S_le}) with $\A = \M_j$ and $\B = \M_k$, for $j = 0$ and $k\in \{1, \dots, n\}$, and for $j,k\in \{1, \dots, n\}$ with $j \le k$. We shall denote these multi-sorted relations from $\M_j$ to $\M_k$ by $S_\le^{jk}$ and $S_\ge^{jk}$:
\begin{itemize}
\item $S_\le^{0k} = \big(M_0\times \{\top\}\big)\cup \big(\{\bot\}\times M_k\big) \cup \{(\mbf, \mbf), (\mbf, \zero)\}\cup \{(\mbt, \mbt), (\mbt, \one)\}$ as a relation from $\M_0$ to $\M_k$, for $k\in \{1, \dots, n\}$,

\item $S_\ge^{0k} = \big(M_0\times \{\bot\}\big)\cup \big(\{\top\}\times M_k\big) \cup \{(\mbf, \mbf), (\mbf, \zero)\}\cup \{(\mbt, \mbt), (\mbt, \one)\}$ as a relation from $\M_0$ to $\M_k$, for $k\in \{1, \dots, n\}$,

\item $S_\le^{jk} = \big(M_j \times \{\top\}\big) \cup \big(\{\bot\} \times M_k\big) \cup \big(\{\mbf, \zero\}\times \{\mbf, \zero\}\big) \cup \big(\{\mbt, \one\}\times \{\mbt, \one\}\big)$
as a relation from $\M_j$ to $\M_k$, for $j, k\in \{1, \dots, n\}$ with $j < k$,

\item $S_\ge^{jk} = \big(M_j \times \{\bot\}\big) \cup \big(\{\top\} \times M_k\big) \cup \big(\{\mbf, \zero\}\times \{\mbf, \zero\}\big) \cup \big(\{\mbt, \one\}\times \{\mbt, \one\}\big)$
as a relation from $\M_j$ to $\M_k$, for $j, k\in \{1, \dots, n\}$ with $j < k$.

\end{itemize}

\begin{thm}\label{thm:MImulti}
Let $n\in \omega\comp\{0\}$.
\begin{enumerate}[\normalfont (1)]

\item The meet-irreducible elements of
$\Sub(\M_0\times \M_0)$ are $\lez$ and $\gez$.

\item For all $k\in \{1, \dots, n\}$, the meet-irreducible elements of
$\Sub(\M_k\times \M_k)$ are $\lek$, $\gek$, $S_\le^{kk}$ and $S_\ge^{kk}$.

\item For all $k\in \{1, \dots, n\}$, the meet-irreducible elements of $\Sub(\M_0\times \M_k)$ are $S_\le^{0k}$ and $S_\ge^{0k}$.

\item For all $j,k\in \{1, \dots, n\}$ with $j < k$, the meet-irreducible elements of $\Sub(\M_j\times \M_k)$ are $\lejk$, $S_\le^{jk}$ and $S_\ge^{jk}$.

\end{enumerate}
\end{thm}

\begin{proof}
Let $j,k\in \{0, \dots, n\}$ with $j \le k$, and let $u\colon \JB_n \to \M_j$ and $v\colon \JB_n \to \M_k$ be the unique homomorphisms. Hence, $u$ maps $\mbf_0,\ldots,\mbf_{j-1}$ to $\zero$ and maps $\mbf_j, \ldots, \mbf_n$ to~$\mbf$, and $v$ maps $\mbf_0,\ldots,\mbf_{k-1}$ to $\zero$ and maps $\mbf_k, \ldots, \mbf_n$ to~$\mbf$ (and similarly for the `true' constants).

By Theorem~\ref{thm:MI}, the meet-irreducibles in $\Sub(\M_j \times \M_k)$ are the appropriate versions of $S_\le$, $S_\ge$ and $S_{ab}$, for $(a, b) \in \big(\bsF^{\M_j} \times \bsF^{\M_k}\big)\comp K$.
Inspection shows that $S_\le$ and $S_\ge$ yield the relations $\lez$ and~$\gez$, when $j = k = 0$, and yield $S_\le^{jk}$ and $S_\ge^{jk}$ otherwise. It remains to calculate the relations~$S_{ab}$, for $(a, b) \in \big(\bsF^{\M_j} \times \bsF^{\M_k}\big)\comp K$.
Since $\big(\bsF^{\M_0} \times \bsF^{\M_k}\big)\comp K = \varnothing$, for all $k\in \{0, \dots, n\}$, we must calculate~$S_{ab}$, for $(a, b) \in \big(\bsF^{\M_j} \times \bsF^{\M_k}\big)\comp K$, with $0 <j \le k$. We need to distinguish two cases: $j < k$ and $j = k$.

First consider the case where $j < k$. We then have
\begin{align*}
K &= \{\, (u(c), v(c))\mid c\in J_n\,\} \\
&= \{(\top, \top), (\zero, \zero), (\mbf,\zero), (\mbf, \mbf), (\one, \one), (\mbt,\one), (\mbt, \mbt), (\bot, \bot)\}.
\end{align*}
Thus, $\big(\bsF^{\M_j} \times \bsF^{\M_k}\big)\comp K= \{(\zero,\mbf)\}$, whence $S_{\zero\mbf} = {\lejk}$ is the only meet-irreducible of the form $S_{ab}$ that occurs in this case.

Now consider the case where $j = k$. We then have
\[
K = \{\, (u(c), v(c))\mid c\in J_n\,\} =
\{(\top, \top), (\zero, \zero), (\mbf, \mbf), (\one, \one), (\mbt, \mbt), (\bot, \bot)\}.
\]
Thus, $\big(\bsF^{\M_k} \times \bsF^{\M_k}\big)\comp K= \{(\zero, \mbf), (\mbf, \zero)\}$, whence
$S_{\zero\mbf} = {\lek}$ and $S_{\mbf\zero} ={\gek}$ (and no others) occur as meet-irreducibles of the form $S_{ab}$ in this case.
\end{proof}

We are now ready to prove the duality statement in Theorem~\ref{cor:bigmultiduality}. We will need a multi-sorted generalisation of an entailment construct known as \emph{action by an endomorphism}---see \cite[2.4.5(15)]{CD98}. If $A$, $B$, $C$ and $D$ are sorts, $g\colon A \to C$, $h\colon B\to D$ and $S\subseteq C\times D$, then define
\[
(g,h)^{-1}(S) := \{\, (a, b)\in A\times B \mid (g(a), h(b))\in S\,\}.
\]
A simple calculation shows that $\{g, h, S\}\vdash (g,h)^{-1}(S)$.

\begin{proof}[\textbf{Proof of Theorem~\ref{cor:bigmultiduality}: duality}]
As already observed, the Special Multi-sorted NU Strong Duality Theorem~\ref{MultiNUDT} implies that the alter ego $\MT'$
yields a multi-sorted duality on the variety~$\CV_n$. Since each relation $R$ from $\M_j$ to $\M_k$ entails its converse $R\conv$ from $\M_k$ to~$\M_j$, it suffices to restrict to relations in $\Sub(\M_j \times \M_k)$, for $j \le k$. As each set of relations in $\Sub(\M_j \times \M_k)$ entails its intersection, we can further restrict to the meet-irreducibles in $\Sub(\M_j \times \M_k)$. By comparing the relations and maps in $\mathcal{S}_{(n)}\cup \mathcal G_{(n)}$ with the meet-irreducibles listed in Theorem~\ref{thm:MImulti}, we see that it remains to show that $\mathcal{S}_{(n)}\cup \mathcal G_{(n)}$ entails the following relations
\[
\gez \,  \text{ and } \, \gek, \ S_\le^{kk}, \ S_\ge^{kk}, \ S_\le^{0k}, \ S_\ge^{0k}, \ S_\le^{jk},  S_\ge^{jk}, \text{ for  $j,k\in \{0, \dots, n\}$ with $j < k$}.
\]
Since ${\lek} \in \mathcal{S}_{(n)}$ and $\gek$ is the converse of $\lek$, it is clear that $\mathcal{S}_{(n)}\cup \mathcal G_{(n)}$ entails~$\gek$, for all $k\in \{0, \dots, n\}$. We now turn to the relations of the form $S_\le^{jk}$ or $S_\ge^{jk}$.
By Lemma~\ref{lem:Sle=Sgle}, each of these relations is of the form $(g_j,g_k)^{-1}(\lez)$ or $(g_j,g_k)^{-1}(\gez)$, for some $j,k\in \{0, \dots, n\}$, and hence is entailed by $\mathcal{S}_{(n)}\cup \mathcal G_{(n)}$.
This completes the proof that $\MT_n$ yields a duality on the variety~$\CV_n$.

Finally, to show that the duality is strong we need to compare
\[
\mathcal G = \bigcup \{\, \CA(\M_j, \M_k)\mid j, k\in\{0,1,\dots,n\}\,\}
\]
with the set
$\mathcal G_{(n)}$. The values of the constants in each of the algebras $\M_k$, for $k\in \{0, \dots, n\}$, guarantee that,
for all $j, k\in \{0, \dots, n\}$, the only homomorphisms $u\colon \M_k \to \M_j$ are the identity maps $\id_{M_k}$ along with the maps $g_k\colon \M_k \to \M_0$. Since the identity maps can be removed from any alter ego without destroying a strong duality, we are done.
\end{proof}

\section{Proving that the duality on $\CV_n$ given by $\MT_n$ is optimal}\label{sec:optmulti}

In this section we shall prove that, for all $n\in \omega\comp\{0\}$, the multi-sorted duality for  the variety $\CV_n$ given in Theorem~\ref{cor:bigmultiduality} is optimal, that is, none of the
operations and relations in
\begin{align*}
\mathcal{G}_{(n)} &= \big\{\,g_k \mid k\in \{1, \dots, n\}\,\big\}\quad\text{and}\\
\mathcal{S}_{(n)} &=\{\,{\lek} \mid k\in \{0, \dots, n\}\,\} \cup \big\{\,{\lejk} \mid j, k\in \{1, \dots,n\} \text{ with } j < k\,\big\} \end{align*}
can be removed from the alter ego $\MT_n = \langle M_0\du M_1\du \cdots \du M_n; \mathcal{G}_{(n)}, \mathcal{S}_{(n)}, \T\rangle$
without destroying the duality.

Recall that, for every algebra $\B$ in $\CV_n$, the underlying set of the multi-sorted dual of $\B$ is given by
\[
\mathrm{D}(\B) = \CV_n(\B, \M_0) \du \CV_n(\B, \M_1) \du \dots\du \CV_n(\B, \M_n).
\]
We shall see that if $\B$ is finite, then the sorts $\CV_n(\B, \M_k)$ of $\mathrm{D}(\B)$ have a very simple structure.
We need the following special case of J\'onsson's Lemma~\cite[Lemma 3.1]{Jon}.

\begin{lem}[J\'onsson]\label{lem:Jon}
Let $\B$ be a subalgebra of\/
$\prod_{i\in I} \A_i$
with $I$ finite, let $\C$ be subdirectly irreducible and let $u\colon \B \to \C$ be a surjective homomorphism.
Then $u = g \circ \pi_i\rest B$, for some $i\in I$ and some homomorphism $g \colon \pi_i(\B) \to \C$.
\end{lem}

Recall that, for $j, k\in \{1, \dots, n\}$, the only endomorphism of $\M_k$ is $\id_M$, there are no homomorphisms from $\M_j$ to $\M_k$ when $j \ne k$, and the only homomorphism from $\M_k$ to $\M_0$ is~$g_k$.

\begin{lem}\label{lem:Jonsson}
Let $\B$ be a subalgebra of\/
 $\prod_{i\in I} \A_i$ with the set $I$ finite and $\A_i\in \{\M_0, \dots, \M_n\}$, for all $i\in I$. Then, for all $k \in \{0, \dots, n\}$, every homomorphism from $\B$ to $\M_k$ is the restriction of a projection or,
when $k = 0$, is the restriction of a projection followed by one of the homomorphisms in $\mathcal{G}_{(n)}$.
\end{lem}
\begin{proof}
Let $k\in \{0, \dots, n\}$ and let $u\colon \B \to \M_k$ be a homomorphism. Since every element of $\M_k$ is the value of a constant, the map $u$  is surjective. Similarly, the restricted projection $\pi_i\rest B \colon \B \to \A_i$ is surjective. Since $\M_k$ is subdirectly irreducible, it follows from Lemma~\ref{lem:Jon} that $u = g \circ \pi_i\rest B$, for some $i\in I$ and some homomorphism $g \colon \A_i \to \M_k$. If $k \ne 0$, then we must have $\A_i = \M_k$ and $g = \id_{M_k}$, whence $u = \pi_i\rest B$. If $k = 0$, then either $\A_i = \M_0$, in which case $g = \id_M$ and hence $u = \pi_i\rest B$, or $\A_i = \M_k$, for some $k \ne 0$, in which case $g = g_k$ and so $u = g_k \circ \pi_i\rest B$, as claimed.
\end{proof}

Let $\CM$ be a finite set of finite algebras. By analogy with the single-sorted situation, a compatible multi-sorted binary relation $R$ on $\CM$ is \emph{absolutely unavoidable} within the set $\mathcal{R}_{\CM}$ of all compatible multi-sorted binary relations on~$\CM$
if, for every subset $\mathcal{R}$ of $\mathcal{R}_{\CM}$ such that $\MT = \langle \bigdu\{\, M\mid \M\in \CM\,\}; \mathcal{R}, \T\rangle$ yields a multi-sorted duality on $\ISP(\CM)$, we have $\mathcal{R}\cap \{R, R\conv\,\} \neq \varnothing$.

\begin{prop}\label{lem:lez}
Let $\CM = \{\M_0, \dots, \M_n\}$.
The relation $\lez$ is absolutely unavoidable within $\mathcal{R}_{\CM}$.
\end{prop}

\begin{proof}
Let $\bm{\leqslant}^0$ denote the algebra with underlying set $\lez$ and  let $\mathcal R$ be any set of compatible multi-sorted binary relations on $\CM$ that yields a duality on $\CV_n$, and therefore yields a duality on the algebra~$\bm{\leqslant}^0$. By Lemma~\ref{lem:Jonsson}, every sort of $\mathrm{D}(\bm{\leqslant}^0)$, other than the $M_0$-sort
is empty. It follows that if $R$ is a compatible multi-sorted binary relation from $\M_j$ to $\M_k$, with at least one of $j$ and $k$ not equal to~$0$, then $R^{\mathrm{D}(\bm{\leqslant}^0)} = \varnothing$. Hence $\mathcal R$ must include a binary relation on~$\M_0$. Since the only subuniverses of $\M_0^2$ are $\lez$, $\gez$ and the trivial relations $\Delta$ and $M_0^2$, it follows that $\mathcal R$ must include $\lez$ or $\gez$, that is, $\lez$ is absolutely unavoidable within $\mathcal{R}_{\CM}$.
\end{proof}

\begin{thm}
The duality on the variety $\CV_n$ yielded by the multi-sorted alter ego $\MT_n$ is optimal.
\end{thm}

\begin{proof}
By Proposition~\ref{lem:lez}, it remains to show that none of the relations in $\mathcal{S}_{(n)}\comp\{\lez\}$ and operations in $\mathcal{G}_{(n)}$
can be removed from the alter ego $\MT_n = \langle
M_0\du M_1\du \cdots \du M_n; \mathcal{G}_{(n)}, \mathcal{S}_{(n)}, \T\rangle$ without destroying the duality.

Let $j, k\in \{0, \dots, n\}$ with $j < k$ and consider the relation
$\lejk$. Let $\bm{\leqslant}^{jk}$ be the algebra with underlying set
$\lejk$ and $\rho_1 \colon \bm{\leqslant}^{jk}\to \M_j$, 
$\rho_2 \colon \bm{\leqslant}^{jk}\to \M_k$ be the restrictions of the
projections. By Lemma~\ref{lem:Jonsson}, the non-empty sorts of
$\mathrm{D}(\bm{\leqslant}^{jk})$ are
$\CV_n(\bm{\leqslant}^{jk}, \M_0) = \{g_j \circ \rho_1, g_k\circ \rho_2\}$, $\CV_n(\bm{\leqslant}^{jk}, \M_j) = \{\rho_1\}$, and $\CV_n(\bm{\leqslant}^{jk},
\M_k) = \{\rho_2\}$.
Define $\gamma\colon \mathrm{D}(\bm{\leqslant}^{jk}) \to M_0 \du\dots\du M_n$ by
\[
\gamma(g_j \circ \rho_1) = \gamma(g_k \circ \rho_2) = \mbf \in M_0, \quad
\gamma(\rho_1) = \zero \in M_j, \quad \text{and} \quad \gamma(\rho_2) = \mbf \in M_k.
\]
We show that $\gamma$ preserves each
relation and operation in $\big(\mathcal S_{(n)}\comp\{\lejk\}\big)\cup \mathcal G_{(n)}$ and does not preserve $\lejk$, whence $\lejk$ cannot be removed without destroying the duality.

Since $(\rho_1, \rho_2) \in {\leqslant^{jk}}$ in ${\mathrm{D}(\bm{\leqslant}^{jk})}$, but $(\gamma(\rho_1), \gamma(\rho_2)) = (\zero, \mbf)\notin {\leqslant^{jk}}$, the map $\gamma$ does not preserve~$\leqslant^{jk}$. As $g_j(\mbf) = g_k(\mbf) = \mbf$, the map $\gamma$ preserves the action of the map $g_k$, for all $k\in\{1, \dots, n\}$.
It remains to prove that $\gamma$ preserves $\lej$ and ${\lek}$, as all other relations in $\mathcal S_{(n)}$ are empty on $\mathrm{D}(\bm{\leqslant}^{jk})$; but this is trivial as $(\zero, \zero)\in {\lej}$ and $(\mbf, \mbf)\in {\lek}$.

Now let $k\in \{1, \dots, n\}$,
let $\bm{\leqslant}^k$ be the algebra with underlying set ${\lek}$, and let $\rho_1, \rho_2 \colon \bm{\leqslant}^k\to \M_k$ be the restrictions of the projections. Again by Lemma~\ref{lem:Jonsson}, the non-empty sorts of $\mathrm{D}(\bm{\leqslant}^k)$ are
\[
\CV_n(\bm{\leqslant}^k, \M_0) =
\{g_k \circ \rho_1, g_k\circ \rho_2\}
\quad\text{and}\quad \CV_n(\bm{\leqslant}^k, \M_k) = \{\rho_1, \rho_2\}.
\]
Define
$\gamma\colon \mathrm{D}(\bm{\leqslant}^k) \to M_0 \du\dots\du M_n$ by
\[
\gamma(g_k \circ \rho_1)
= \gamma(g_k \circ \rho_2) = \mbf \in M_0, \quad
\gamma(\rho_1) = \zero \in M_k, \quad \text{and} \quad \gamma(\rho_2) = \mbf \in M_k.
\]
Again it is easy to see that $\gamma$ does not preserve $\lek$ (as $(\zero,\mbf)\notin {\lek}$), that $\gamma$ preserves (by construction) the action of the map $g_k$, for all $k\in\{1, \dots, n\}$,
that $\gamma$ preserves $\lez$ (as $(\mbf, \mbf)\in {\lez}$), and that $\gamma$ preserves all other relations in $\mathcal S$ (as they are empty on  $\mathrm{D}(\bm{\leqslant}^k)$). Consequently, ${\lek}$ cannot be removed without destroying the duality.

Finally, fix $k \in \{1, \dots, n\}$. We shall show that $g_k$ cannot be removed from $\mathcal G_{(n)}$ without destroying the duality.
A third application of Lemma~\ref{lem:Jonsson} shows that the non-empty sorts of $\mathrm{D}(\M_k)$ are
 \[
\CV_n(\M_k, \M_0) = \{g_k\} \quad\text{and}\quad \CV_n(\M_k, \M_k) = \{\id_{M_k}\}.
 \]
 Define $\gamma\colon \mathrm{D}(\M_k) \to M_0 \du\dots\du M_n$ by
$\gamma(g_k) = \mbf \in M_0$ and $\gamma(\id_{M_k}) = \mbt \in M_k$. Clearly, $\gamma$ does not preserve the action of $g_k$ since
\[
\gamma\big(g_k^{\mathrm{D}(\M_k)}(\id_{M_k})\big) = \gamma\big(g_k \circ \id_{M_k}\big) = \gamma(g_k) = \mbf \ne \mbt = g_k(\mbt) = g_k(\gamma(\id_{M_k})).
\]
The map $\gamma$ preserves $\lez$ since $(\mbf, \mbf)\in {\lez}$, preserves
$\lek$ since $(\mbt, \mbt)\in {\lek}$, and preserves all other relations in $\mathcal S_{(n)}$ as they are empty on $\mathrm{D}(\M_k)$. Hence $g_k$ cannot be deleted from $\mathcal G_{(n)}$ without destroying the duality.
\end{proof}

\section*{Acknowledgements}

The authors would like to thank Jane Pitkethly for carefully
preparing their diagrams in TikZ. The first and second author would
like to thank Matej Bel University for its hospitality during a
research visit in September 2017. The first author would like to
thank Hilary Priestley and Leonardo Cabrer for useful discussions
and guidance during his DPhil studies at the University of Oxford
when the first work on these bilattices took place~\cite{C-thesis}.
The third author acknowledges the support of Slovak grant VEGA
1/0337/16 and the hospitality of La Trobe University during his
stay there in August 2018.


\end{document}